\documentclass[11pt,a4paper]{article}
\usepackage[margin=1in]{geometry}
\usepackage[T1]{fontenc}
\usepackage{lmodern}
\usepackage{amsmath,amssymb,amsthm}
\usepackage[numbers,sort&compress]{natbib}
\usepackage{algorithm}
\usepackage{algpseudocode}
\usepackage{bm}
\usepackage{xcolor}

\usepackage{graphicx}
\usepackage{xurl}
\usepackage{tikz}
\usetikzlibrary{arrows.meta,calc,fit,positioning}
\usepackage{enumitem}
\usepackage{needspace}
\usepackage[hidelinks,hypertexnames=false]{hyperref}
\theoremstyle{plain}
\newtheorem{theorem}{Theorem}
\newtheorem{lemma}{Lemma}
\newtheorem{proposition}{Proposition}
\newtheorem{corollary}{Corollary}
\newtheorem{assumption}{Assumption}
\theoremstyle{definition}
\newtheorem{definition}{Definition}
\newtheorem{example}{Example}
\theoremstyle{remark}
\newtheorem{remark}{Remark}

\newcommand{\R}{\mathbb{R}}
\newcommand{\T}{^{\mathsf T}}
\newcommand{\grad}{\nabla}
\newcommand{\proj}{\operatorname{proj}}
\newcommand{\argmin}{\operatorname*{argmin}}
\newcommand{\ext}{\operatorname{ext}}
\newcommand{\cI}{\mathcal I}
\newcommand{\cA}{\mathcal A}

\newcommand{\cC}{\mathcal C}

\newcommand{\cL}{\mathcal L}

\newcommand{\cP}{\mathcal P}
\newcommand{\cS}{\mathcal S}
\newcommand{\cOmega}{\Omega}

\newcommand{\LP}{\mathrm{LP}}
\newcommand{\RBLA}{\textnormal{\scshape RBLA}}
\newcommand{\RPSA}{\textnormal{\scshape RPSA}}
\newcommand{\TBTA}{\textnormal{\scshape TBTA}}
\newcommand{\LBTRS}{\textnormal{\scshape LBTRS}}
\newcommand{\pred}{\operatorname{pred}}
\newcommand{\ared}{\operatorname{ared}}
\newcommand{\dist}{\operatorname{dist}}
\newcommand{\runinhead}[1]{\par\smallskip\noindent #1\enspace}

\makeatletter
\newenvironment{breakablealgorithm}
{%
  \Needspace{8\baselineskip}%
  \par\addvspace{\intextsep}%
  \noindent\refstepcounter{algorithm}%
  \hrule height .8pt depth 0pt\kern2pt%
  \renewcommand{\caption}[2][\relax]{%
    \noindent\textbf{Algorithm~\thealgorithm}\enspace ##2\par%
    \ifx\relax##1\relax
      \addcontentsline{loa}{algorithm}{\protect\numberline{\thealgorithm}##2}%
    \else
      \addcontentsline{loa}{algorithm}{\protect\numberline{\thealgorithm}##1}%
    \fi
    \kern2pt\hrule\kern2pt%
  }%
}
{%
  \kern2pt\hrule\par\addvspace{\intextsep}%
}

\newcounter{subalgorithm}
\newenvironment{breakablesubalgorithm}
{%
  \Needspace{8\baselineskip}%
  \par\addvspace{\intextsep}%
  \noindent\refstepcounter{subalgorithm}%
  \hrule height .8pt depth 0pt\kern2pt%
  \renewcommand{\caption}[2][\relax]{%
    \noindent\textbf{Subalgorithm~\thesubalgorithm}\enspace ##2\par%
    \kern2pt\hrule\kern2pt%
  }%
}
{%
  \kern2pt\hrule\par\addvspace{\intextsep}%
}
\makeatother

\title{Randomized Branch Methods with Inexact Subproblems for Bouligand Stationarity in Linear and Quadratic Programs with Complementarity Constraints}
\author{%
Ziyin Hu\thanks{Department of Mathematics, Southern University of Science and
Technology, Shenzhen 518055, People's Republic of China.
Email: \href{mailto:12432016@mail.sustech.edu.cn}{\nolinkurl{12432016@mail.sustech.edu.cn}}.}
\and
Xin Liu\thanks{Institute of Computational Mathematics and Scientific/Engineering
Computing, Academy of Mathematics and Systems Science, Chinese Academy of
Sciences, University of Chinese Academy of Sciences, Beijing 100190,
People's Republic of China.
Email: \href{mailto:liuxin@lsec.cc.ac.cn}{\nolinkurl{liuxin@lsec.cc.ac.cn}}.}
\and
Shangzhi Zeng\thanks{National Center for Applied Mathematics Shenzhen, and
Department of Mathematics, Southern University of Science and Technology,
Shenzhen 518055, People's Republic of China.
Email: \href{mailto:zengsz@sustech.edu.cn}{\nolinkurl{zengsz@sustech.edu.cn}}.}
\and
Jin Zhang\thanks{Corresponding author. Department of Mathematics, and National
Center for Applied Mathematics Shenzhen, Southern University of Science and
Technology, Shenzhen 518055, People's Republic of China.
Email: \href{mailto:zhangj9@sustech.edu.cn}{\nolinkurl{zhangj9@sustech.edu.cn}}.}
}
\date{}

\begin{document}
\maketitle

\begin{abstract}

We develop randomized branch methods for finding Bouligand-stationary
(B-stationary) points of linear and quadratic programs with complementarity
constraints (LPCCs and QPCCs) using only linear programming subproblems.  The
methods exploit the finite-union geometry of the feasible set and search for
first-order descent on randomly selected compatible branches.  For LPCCs, an
exact method optimizes over sampled branches, while an inexact simplex method
can accept an improving branch-feasible vertex before solving the sampled
penalty problem to optimality.  For QPCCs, including problems with indefinite
quadratic objectives, branch quadratic programs are replaced by linearized
trust-region subproblems; a ratio test ensures actual decrease, and thresholded
sampling detects branches that emerge only at accumulation points.  Under the
stated assumptions, the LPCC methods stabilize after finitely many changes at
B-stationary points almost surely.  For the QPCC method, finite termination
yields a B-stationary point, and every accumulation point of an infinite run is
B-stationary almost surely.  Experiments on bilevel-induced instances,
instances arising from inverse quadratic programming, and sparse affine
generalized Nash equilibrium instances, together with 129 MacMPEC embedding
tests, show that the methods return points with competitive objective quality
and runtimes on large-scale complementarity systems.


\end{abstract}

\section{Introduction}

This paper studies linear and quadratic programs with complementarity constraints of the
form
\begin{equation}\tag{P}\label{prob:standard-lcc}
\begin{aligned}
    \min_{x\in\mathbb R^n}\quad & f(x)\\
    \text{s.t.}\quad
        & Ax-a=0,\\
        & Bx-b\ge0,\\
        & 0\le x_1\perp x_2\ge0,
\end{aligned}
\end{equation}
where \(x=(x_0,x_1,x_2)\in
\mathbb R^{n_0}\times\mathbb R^m\times\mathbb R^m\), the data have compatible
dimensions, and
\(0\le x_1\perp x_2\ge0\) imposes nonnegativity and componentwise
complementarity.  The objective is either linear, \(f(x)=c\T x\), or
quadratic, \(f(x)=\frac12x\T Qx+q\T x\), where \(Q=Q\T\) may be
indefinite; these cases are called LPCCs and QPCCs, respectively.  Affine
complementarity constraints can be put in this form by introducing explicit
complementarity variables.  Applications include bilevel and inverse
quadratic optimization, affine generalized Nash equilibria, and
complementarity-based optimal control
\cite{YuMitchellPang2019LPCCBC,JaraMoroniPangWachter2018LPCC,%
SchiroPangShanbhag2013,hall2022scp,AllendeStill2013}.  Related generalized
complementarity formulations arise in multimarginal optimal transport models
for strongly correlated electron systems \cite{HuChenLiu2023}, with recent
work exploiting their special multi-block structure for large-scale
computation \cite{HuLiLiuMeng2025}.

Each complementarity condition
$0\le x_{1,i}\perp x_{2,i}\ge0$ is the two-term disjunction
$(x_{1,i}=0,\ x_{2,i}\ge0)\vee(x_{1,i}\ge0,\ x_{2,i}=0)$.
Choosing one alternative for each of the $m$ pairs defines a polyhedral
branch after including the common linear constraints. Let $\cP$ index
these $2^m$ choices and let $\cOmega_J$ denote the branch associated with
$J\in\cP$. The feasible set is therefore
$\cOmega=\bigcup_{J\in\cP}\cOmega_J$, making
\eqref{prob:standard-lcc} a polyhedral disjunctive program
\cite{Balas1979Disjunctive,Vielma2019ConvexUnions,Mehlitz2020}.
Some labelled branches may be empty or identical. Its global optimal value
satisfies
$\inf_{x\in\cOmega}f(x)=\min_{J\in\cP}\inf_{x\in\cOmega_J}f(x)$,
where the infimum over an empty branch is $+\infty$.
Thus, certifying global optimality requires ruling out a better objective
value on every nonempty branch. Explicit enumeration does so by solving
each branch globally and comparing the optimal values, requiring up to
$2^m$ LPs or QPs
\cite{HuMitchellPangBennettKunapuli2008,bai2013convexqpcc,YuMitchellPang2019LPCCBC}.

In this paper, we consider algorithms for finding local stationary points,
specifically Bouligand-stationary (B-stationary) points.  Rather than pursuing
global optimality through complete pattern enumeration, we exploit the
branches that are locally compatible with a given feasible point.  A feasible
point \(\bar x\) is B-stationary if
\(\grad f(\bar x)\T d\geq 0\) for every
\(d\in T_{\cOmega}(\bar x)\), where \(T_{\cOmega}(\bar x)\) is the Bouligand
tangent cone defined in Section~\ref{sec:problem-geometry}.  Since \(\cOmega\) is a finite union of polyhedra,
\(
T_{\cOmega}(\bar x)
=\bigcup_{J\in\cP(\bar x)}T_{\cOmega_J}(\bar x),
\)
where \(\cP(\bar x)\) indexes the branches containing \(\bar x\)
\cite{FlegelKanzowOutrata2007,Mehlitz2020}.  Hence, B-stationarity is
equivalent to the absence of a first-order descent direction on every
compatible branch, a local test that does not require enumeration of all
global patterns.  For LPCCs, B-stationarity is known to be equivalent to local
minimality \cite{HuMitchellPang2012LPCC,JaraMoroniPangWachter2018LPCC};
Proposition~\ref{prop:convex-bstationarity-local-minimum} extends this
equivalence to every differentiable convex objective, including convex QPCCs.

Several methods exploit the combinatorial structure of complementarity
constraints. Logical-Benders and
branch-and-cut schemes can certify global optimality when run to
completion
\cite{HuMitchellPangBennettKunapuli2008,YuMitchellPang2019LPCCBC,JaraMoroniMitchellPangWachter2020};
related cutting-plane relaxations strengthen the search
\cite{DelPiaLinderothZhu2024}.  Progressive mixed-integer methods instead
emphasize incumbent improvement and local-minimum properties without, in
general, a global certificate \cite{ZhangHanPang2026PIP}.  All these routes
retain a combinatorial component whose cost may grow rapidly with \(m\).

One classical approach replaces each complementarity condition
\(0\leq x_{1,i}\perp x_{2,i}\geq 0\) by the equivalent smooth constraints
\(
x_{1,i}\geq 0, x_{2,i}\geq 0,
x_{1,i}x_{2,i}\leq 0,
\)
and then applies standard NLP methods.  This reformulation, however, fails
LICQ and MFCQ at every feasible point, so classical KKT theory cannot be
invoked in the usual way
\cite{scheel2000stationarity,demiguel2005twosided}.  Moreover, stationarity
for the smooth reformulation may fail to capture the full disjunctive geometry
and is generally weaker than B-stationarity
\cite{pang1999cq,scheel2000stationarity}.  Regularization methods recover
smooth NLP subproblems but do not close this gap without additional
assumptions; for example, under MPCC-LICQ alone, the classical Scholtes scheme
guarantees only C-stationarity, whereas B-stationarity requires further
second-order and upper-level strict-complementarity conditions
\cite{scholtes2001regularization}.

A more recent line of work applies the DC Algorithm (DCA) to
difference-of-convex (DC) penalty reformulations \cite{tao1997convex}.  This
approach has been developed for LCPs, LPCCs, and broader MPCCs
\cite{le2011solving,JaraMoroniPangWachter2018LPCC,flocco2026successive}.
For LPCCs, a piecewise-linear penalty leads to LP subproblems and has shown
favorable performance on large tested instances
\cite{JaraMoroniPangWachter2018LPCC}.  Standard DCA theory, however, typically
guarantees only criticality with respect to the chosen DC decomposition,
which can be strictly weaker than directional stationarity
\cite{JaraMoroniPangWachter2018LPCC,PangRazaviyaynAlvarado2017}.  Recent
variants seek directional stationarity for structured DC models through
perturbation or randomized active-set screening under their respective
assumptions \cite{FengYuan2026pDCA,Niu2026RADCA}.

More closely aligned with the branchwise characterization above are active-set
methods.  They explore branches near the current iterate and generate
candidate iterates on selected branches
\cite{fukushima2002active,FukushimaTsengErratum2007,IzmailovSolodov2008,%
	FangLeyfferMunson2012PivotingLPCC,kirches2022slpcc,nurkanovic2025mpecopt}.
Their progress and stationarity tests, however, typically rely on solving the associated LPCC, QP, or NLP subproblems sufficiently
accurately, often to optimality.  Repeatedly performing such solves can become
costly as the problem size grows.  This leads to the question addressed here:
can B-stationarity be reached for both LPCCs and QPCCs using only LP
subproblems, without fully optimizing the original objective over every
sampled branch?

Our proposed strategy builds on the finite-union geometry and the
branchwise viewpoint of the active-set methods discussed above
\cite{FlegelKanzowOutrata2007,Mehlitz2020,kirches2022slpcc,nurkanovic2025mpecopt}. We use
this structure to organize randomized exploration of compatible branches,
while maintaining feasibility and seeking descent in the original
objective. For LPCCs, we draw on the piecewise-linear penalty formulation
\cite{JaraMoroniPangWachter2018LPCC} to express sampled branch problems as
LPs over a common polyhedron. Changing branches then changes only the
linear objective, allowing simplex bases to be reused and improving
branch-feasible vertices to be accepted before LP optimality. For QPCCs,
we adopt the trust-region technique
\cite{conn2000trust,nocedal2006numerical} and use a linear objective
model on each selected branch. A ratio test controls step acceptance and ensures decrease in the
original objective.

We propose randomized branch-based algorithms for computing B-stationary
points of LPCCs and QPCCs using only LP subproblems. To the best of our
knowledge, for LPCCs, we propose the first randomized algorithm that uses
only LP subproblems and integrates branch selection with simplex pivots
to achieve almost-sure B-stationarity;
for QPCCs, we propose the first randomized algorithm using only LP subproblems.
These algorithms admit almost-sure B-stationarity guarantees. We present a branchwise LP certificate for B-stationarity and
establish its equivalence to local minimality under differentiable convex
objectives. For LPCCs, we analyze an exact randomized branch method and
an exact-penalty simplex-path method that may accept a branch-feasible
improving vertex before penalty-LP optimality. Under boundedness of the
feasible set, the exact branch method stabilizes almost surely at a
B-stationary point. We prove the same result for the simplex-path method
under the assumptions that the penalty parameter is sufficiently large
and the pivot rule terminates finitely. For QPCCs, we introduce thresholded auxiliary branch sampling
within a linearized trust-region method. Under compactness of the initial
level set, every accumulation point of an infinite exact run is almost
surely B-stationary. Finite termination also returns a B-stationary point.

Finally, we report computational results on synthetic LPCC and QPCC
instances derived from bilevel optimization, inverse quadratic
optimization, and affine generalized Nash equilibrium models, with up to
$11{,}984$ complementarity pairs. The proposed methods outperform
state-of-the-art solvers on several large-scale synthetic instances sets.
We also evaluate the LPCC methods as subproblem solvers within an existing
framework for general complementarity-constrained optimization
\cite{nurkanovic2025mpecopt}, using 129 MacMPEC instances \cite{MacMPEC}.
These embedding tests demonstrate competitive solution quality and
computation times relative to the original mixed-integer subproblem solver Gurobi.

The remainder of this paper is organized as follows. Section~\ref{sec:problem-geometry} formalizes the polyhedral branch geometry
and characterizes B-stationarity through branch LPs.
Section~\ref{sec:lpcc-algorithms} develops exact and inexact LP methods for
LPCCs, while Section~\ref{sec:qpcc-linearized} presents the linearized branch
trust-region method for QPCCs and analyzes its convergence.
Section~\ref{sec:numerics} reports the numerical experiments, and
Section~\ref{sec:conclusion} concludes.

\section{Branch Geometry and B-Stationarity}
\label{sec:problem-geometry}
Throughout, $\mathbb N:=\{0,1,\ldots\}$,
$\mathbb N_+:=\{1,2,\ldots\}$, and $[r]:=\{1,\ldots,r\}$ for
$r\in\mathbb N_+$.  For a finite set $I$, $|I|$ denotes its cardinality, and
$\mathbf1_A$ is the indicator of an event $A$.  All vectors, including
gradients, are column vectors; vector inequalities are understood
componentwise, and the superscript $\mathsf T$ denotes transpose.  For
$p\in\{1,2,\infty\}$, $\|v\|_p$ denotes the vector $\ell_p$-norm, with
$\|v\|:=\|v\|_2$; $\|M\|_2$ denotes the matrix spectral norm.  For a nonempty
set $S\subseteq\R^n$ and $p\in\{2,\infty\}$, let
$\dist_p(x,S):=\inf_{y\in S}\|x-y\|_p$.  If $S$ is also closed and convex,
$\proj_S$ denotes the Euclidean projection onto $S$.  A point $x\in S$ is
extreme if it is not a nontrivial convex combination of two distinct points of
$S$, and $\ext S$ is the set of its extreme points.  For a convex function
$h:\R^q\to\R$, define its subdifferential by
$\partial h(x):=\{v\in\R^q\mid h(y)\ge h(x)+v\T(y-x)\ \text{for all }y\in\R^q\}$.

\subsection{Polyhedral disjunctive representation and compatible branches}
\label{subsec:notation}

We now formalize the branch representation described in the Introduction.  Let
\[
    \cP:=
    \{J=(D_1,D_2)\mid
      D_1\cup D_2=[m],\ D_1\cap D_2=\varnothing\}
\]
be the set of labelled ordered partitions of $[m]$.  For
$J=(D_1,D_2)\in\cP$, define the polyhedral branch
\[
\cOmega_J:=
\left\{
x\in\R^n\ \middle|\
\begin{array}{l}
Ax-a=0,\quad Bx-b\ge0,\\
x_{1,i}=0,\ x_{2,i}\ge0,\quad i\in D_1,\\
x_{1,i}\ge0,\ x_{2,i}=0,\quad i\in D_2
\end{array}
\right\}.
\]
Consequently, the feasible set of \eqref{prob:standard-lcc} is
\[
    \cOmega
    :=
    \{x\in\R^n\mid
    Ax-a=0,\ Bx-b\ge0,\ 0\le x_1\perp x_2\ge0\}
    =
    \bigcup_{J\in\cP}\cOmega_J.
\]
Here $|\cP|=2^m$, although some
$\cOmega_J$ may be empty or two labels may determine the same polyhedron.

For the local analysis, let
$z=(z_0,z_1,z_2)$ satisfy
$0\le z_1\perp z_2\ge0$, define
\[
\begin{aligned}
    \cI_{+0}(z)&:=\{i\in[m]\mid z_{1,i}>0,\ z_{2,i}=0\},\\
    \cI_{0+}(z)&:=\{i\in[m]\mid z_{1,i}=0,\ z_{2,i}>0\},\\
    \cI_{00}(z)&:=\{i\in[m]\mid z_{1,i}=0,\ z_{2,i}=0\}.
\end{aligned}
\]
These sets partition $[m]$.  The indices in $\cI_{00}(z)$ are called
biactive.  The labels compatible with $z$ form
\[
\cP(z):=
\bigl\{J=(D_1,D_2)\in\cP\ \bigm|\
\cI_{0+}(z)\subseteq D_1,\
\cI_{+0}(z)\subseteq D_2\bigr\}.
\]
Each biactive index can be assigned independently to $D_1$ or $D_2$; hence
$|\cP(z)|=2^{|\cI_{00}(z)|}$.  If $z\in\cOmega$, then
\[
    \cP(z)=\{J\in\cP\mid z\in\cOmega_J\},
\]
so $\cP(z)$ is precisely the compatible label set at $z$.

\subsection{Tangent cones and B-stationarity}
\label{subsec:bstationarity}

The algorithms developed below target
B-stationarity, a local first-order stationarity concept,
rather than global optimality.  We therefore turn from the global disjunctive representation to
the local feasible geometry of $\cOmega$, which is described by the
Bouligand tangent cone.  For a set $S\subseteq\R^n$ and $\bar x\in S$, this
cone is defined by
\[
    T_S(\bar x)
    :=
    \left\{
    d\in\R^n\ \middle|\
    \exists\,t_\ell\downarrow0,\ d^\ell\to d:
    \bar x+t_\ell d^\ell\in S
    \right\}.
\]
Because $\cOmega$ is a finite union of closed polyhedra, its Bouligand tangent
cone is the union of the tangent cones to the active disjuncts
\cite{FlegelKanzowOutrata2007,Gfrerer2014,%
BenkoGfrerer2018,Mehlitz2020}:
\begin{equation}\label{eq:finite-union-tangent}
    T_\cOmega(\bar x)
    =
    \bigcup_{J\in\cP(\bar x)}
    T_{\cOmega_J}(\bar x)
    \qquad(\bar x\in\cOmega).
\end{equation}

Writing $B_r$ for the $r$th row of $B$, let
$\cA(\bar x):=\{r\mid(B\bar x-b)_r=0\}$.  For
$J=(D_1,D_2)\in\cP(\bar x)$, the polyhedral tangent cone
of the branch $\cOmega_J$ is
\[
T_{\cOmega_J}(\bar x)=
\left\{d=(d_0,d_1,d_2)\ \middle|\
\begin{array}{@{}l@{}}
Ad=0,\quad B_r d\ge0\quad(r\in\cA(\bar x)),\\
d_{1,i}=0\ (i\in D_1),\quad d_{2,i}=0\ (i\in D_2),\\
d_{2,i}\ge0\quad(i\in D_1\cap\cI_{00}(\bar x)),\\
d_{1,i}\ge0\quad(i\in D_2\cap\cI_{00}(\bar x))
\end{array}\right\}.
\]
Taking the union gives the tangent cone
of the original feasible set $\cOmega$
\cite{scheel2000stationarity,pang1999cq}
\[
T_\cOmega(\bar x)=
\left\{d=(d_0,d_1,d_2)\ \middle|\
\begin{array}{@{}l@{}}
Ad=0,\quad B_r d\ge0\quad(r\in\cA(\bar x)),\\
d_{1,i}=0\ (i\in\cI_{0+}(\bar x)),\quad
d_{2,i}=0\ (i\in\cI_{+0}(\bar x)),\\
0\le d_{1,i}\perp d_{2,i}\ge0\quad(i\in\cI_{00}(\bar x))
\end{array}\right\}.
\]
We consider the following notion of B-stationarity
\cite{pang1999cq,scheel2000stationarity}.
\begin{definition}\label{def:bstationarity}
A feasible point $\bar x\in\cOmega$ is Bouligand stationary
(B-stationary) for \eqref{prob:standard-lcc} if
\[
    \grad f(\bar x)\T d\ge0
    \qquad\forall d\in T_\cOmega(\bar x).
\]
\end{definition}

The finite-union tangent formula shows that B-stationarity is equivalent
to the absence of first-order descent on every compatible branch:
\begin{equation}\label{eq:branch-bstationarity}
    \grad f(\bar x)\T d\ge0
    \qquad
    \forall d\in T_{\cOmega_J}(\bar x),\
    \forall J\in\cP(\bar x).
\end{equation}

For LPCCs, the equivalence below is
\cite[Proposition~1]{JaraMoroniPangWachter2018LPCC};
the equivalence extends to differentiable convex
objectives as follows.
\begin{proposition}
\label{prop:convex-bstationarity-local-minimum}
Let $f$ be differentiable and convex.  A point $\bar x\in\cOmega$ is
B-stationary if and only if it is a local minimizer of
\eqref{prob:standard-lcc}.
\end{proposition}

\begin{proof}
If $\bar x$ is a local minimizer, the standard first-order necessary condition
\cite[Theorem~6.12]{RockafellarWets1998} gives
$\grad f(\bar x)\T d\ge0$ for every $d\in T_\cOmega(\bar x)$, which is
B-stationarity.  Equivalently, by \eqref{eq:finite-union-tangent}, this
condition holds on every active disjunct.  Conversely, suppose that $\bar x$
is B-stationary.  For every $J\in\cP(\bar x)$ and $y\in\cOmega_J$, convexity
gives $y-\bar x\in T_{\cOmega_J}(\bar x)$ and
$f(y)\ge f(\bar x)+\grad f(\bar x)\T(y-\bar x)\ge f(\bar x)$.  For every
$J\in\cP\setminus\cP(\bar x)$ with $\cOmega_J\ne\varnothing$, closedness
gives $\dist_2(\bar x,\cOmega_J)>0$.  Since $\cP$ is finite, there exists
$\delta>0$ such that
$\{y\in\cOmega\mid\|y-\bar x\|<\delta\}\subseteq
\bigcup_{J\in\cP(\bar x)}\cOmega_J$.  Hence $f(y)\ge f(\bar x)$ for all
feasible $y$ sufficiently close to $\bar x$, proving local minimality.
\end{proof}

\subsection{Branch-LP characterization of B-stationarity}
\label{subsec:branch-lp-characterization}

By \eqref{eq:branch-bstationarity}, failure of B-stationarity yields a
first-order descent direction on a compatible polyhedral branch.  Scaling such
a direction within a prescribed box radius leads to the bounded LP below and
the stationarity test used by the algorithms.

For $x\in\cOmega$, $J\in\cP(x)$, and $\Delta>0$, define the branch-LP
(bLP) predicted descent
\begin{equation}\label{eq:branch-pred}
    p_J(x,\Delta)
    :=
    \max\left\{
    -\grad f(x)\T d
    \ \middle|\
    x+d\in\cOmega_J,\ \|d\|_\infty\le\Delta
    \right\}.
\end{equation}
The zero step is feasible and the feasible region is compact; hence
$p_J(x,\Delta)$ is well defined and nonnegative.

\begin{proposition}
\label{prop:branch-lp-certificate}
Let $\bar x\in\cOmega$.  The following statements are equivalent:
\begin{enumerate}[label=(\roman*)]
    \item $\bar x$ is B-stationary.
    \item For every $J\in\cP(\bar x)$ and every $\Delta>0$,
    \[
        p_J(\bar x,\Delta)=0.
    \]
    \item There exists $\Delta_0>0$ such that
    \[
        p_J(\bar x,\Delta_0)=0
        \qquad\forall J\in\cP(\bar x).
    \]
\end{enumerate}
Moreover, if $\bar x$ is not B-stationary, then there exists
$J_{\rm dec}\in\cP(\bar x)$ such that
\[
p_{J_{\rm dec}}(\bar x,\Delta)>0
    \qquad\forall\Delta>0.
\]
\end{proposition}

\begin{proof}
Suppose (i) holds.  If
$\bar x+d\in\cOmega_J$ for $J\in\cP(\bar x)$, convexity of the
branch gives $d\in T_{\cOmega_J}(\bar x)$.  Consequently,
$-\grad f(\bar x)\T d\le0$.  Since the zero step is feasible,
$p_J(\bar x,\Delta)=0$, proving (ii); (ii) implies (iii) immediately.

Assume (iii) and take $d\in T_\cOmega(\bar x)$.  By the tangent
decomposition,
$d\in T_{\cOmega_J}(\bar x)$ for some $J\in\cP(\bar x)$.  Since
$\cOmega_J$ is polyhedral, choose $t>0$ sufficiently small that
$\bar x+t d\in\cOmega_J$ and
$t\|d\|_\infty\le\Delta_0$.  Then
\[
    -t\grad f(\bar x)\T d
    \le p_J(\bar x,\Delta_0)=0,
\]
which proves (i).

Finally, if $\bar x$ is not B-stationary, there is
$d\in T_\cOmega(\bar x)$ such that $\grad f(\bar x)\T d<0$.
The tangent decomposition yields this direction in
$T_{\cOmega_{J_{\rm dec}}}(\bar x)$ for some
$J_{\rm dec}\in\cP(\bar x)$.  Since $\cOmega_{J_{\rm dec}}$ is polyhedral,
given any $\Delta>0$, choose $t>0$ such that
\[
   \bar x+t d\in\cOmega_{J_{\rm dec}},
    \qquad t\|d\|_\infty\le\Delta.
\]
It follows that \( p_{J_{\rm dec}}(\bar x,\Delta)
    \ge -t\grad f(\bar x)\T d>0\).
This proves the conclusion.
\end{proof}

Proposition~\ref{prop:branch-lp-certificate} characterizes B-stationarity
through the values $p_J$: they vanish on all compatible branches if and only
if $\bar x$ is B-stationary.  Otherwise,
by 
Proposition~\ref{prop:branch-lp-certificate}, some
$J_{\rm dec}\in\cP(\bar x)$ satisfies
$p_{J_{\rm dec}}(\bar x,\Delta)>0$ for every $\Delta>0$, revealing a
compatible branch on which descent can be pursued.

\section{Exact and Inexact LP Methods for LPCCs}
\label{sec:lpcc-algorithms}

This section considers the LPCC case $f(x)=c\T x$ and develops two randomized
methods for LPCCs: Randomized Branch-LP Descent Algorithm (\RBLA) and
Randomized Penalty-objective Simplex search Algorithm (\RPSA).
  We establish their almost-sure stabilization at B-stationary points
under the assumptions stated below.

\begin{assumption}
\label{ass:lpcc-bounded-polyhedra}
The feasible set $\cOmega$ is bounded in the LPCC case.
\end{assumption}

\subsection{Exact branch-LP method for LPCCs}
\label{subsec:lpcc-exact-branch}

For each label $J\in\cP$ whose polyhedral branch $\cOmega_J$ is nonempty, let
\[
    f_J^*:=\min_{y\in\cOmega_J}c\T y.
\]
Assumption~\ref{ass:lpcc-bounded-polyhedra} ensures that this minimum is
attained.  For $x\in\cOmega$ and $J\in\cP(x)$, sufficiently large $\Delta$
makes the box constraint in \eqref{eq:branch-pred} redundant, so
$p_J(x,\Delta)=c\T x-f_J^*$.
Proposition~\ref{prop:branch-lp-certificate} therefore gives the exact
branchwise certificate
\begin{equation}\label{eq:lpcc-full-branch-certificate}
    x\text{ is B-stationary}
    \quad\Longleftrightarrow\quad
    f_J^*=c\T x\quad\forall J\in\cP(x).
\end{equation}
If the certificate fails, some compatible branch satisfies
$f_J^*<c\T x$.

The randomized branch-LP descent algorithm (\RBLA) samples
one compatible branch at each iteration.  Specifically, at the $k$-th iteration, it
samples one label $J^k\in\cP(x^k)$ and solves the LP on $\cOmega_{J^k}$
exactly.

For describing the random behavior of \RBLA{}, we introduce a common probability space
$(\Xi,\mathcal G,\mathbb P)$.  Here $\Xi$ is the set of possible
branch-selection sequences, $\mathcal G$ is the sigma-algebra of events
generated by these selections, and $\mathbb P$ is their probability law.  Let
$\{\mathcal F_k\}_{k\in\mathbb N}$ be the filtration representing all
information available immediately before $J^k$ is sampled.  Thus $x^k$ and
$\cP(x^k)$ are $\mathcal F_k$-measurable, while $J^k$ and the quantities
produced at the $k$-th iteration are $\mathcal F_{k+1}$-measurable.  At each
iteration, $J^k$ is sampled uniformly from $\cP(x^k)$ conditional on
$\mathcal F_k$. For every $J\in\cP$, define
\[
    A_k^J:=\{J\in\cP(x^k)\},\qquad E_k^J:=\{J^k=J\}.
\]
The event $A_k^J$ occurs when branch $J$ is compatible with $x^k$, and
$E_k^J$ occurs when branch $J$ is sampled at the $k$-th iteration.
The conditional sampling rule is
\begin{equation}\label{eq:lpcc-rbla-conditional-uniform}
\mathbb P(E_k^J\mid\mathcal F_k)=\frac{\mathbf1_{A_k^J}}{|\cP(x^k)|}.
\end{equation}
Here and below, probability-one statements refer to $\mathbb P$.  If a branch
LP has multiple minimizers, any one of them may be selected.  The resulting
method is stated in Algorithm~\ref{alg:lpcc-rbla}.

\begin{breakablealgorithm}
\caption{Randomized Branch-LP Descent Algorithm (\RBLA)}
\label{alg:lpcc-branch}
\label{alg:lpcc-rbla}
\begin{algorithmic}[1]
\Require feasible point $x^0\in\cOmega$
\For{$k=0,1,2,\ldots$}
    \State Sample $J^k$ conditionally uniformly from $\cP(x^k)$
    \State Solve the branch LP exactly
    \[
        \widehat x^k\in
        \argmin\{c\T x\mid x\in\cOmega_{J^k}\}
    \]
    \If{$c\T\widehat x^k<c\T x^k$}
        \State Set $x^{k+1}=\widehat x^k$
    \Else
        \State Set $x^{k+1}=x^k$
    \EndIf
\EndFor
\end{algorithmic}
\end{breakablealgorithm}

Because $J^k\in\cP(x^k)$, the sampled branch contains $x^k$ and is therefore
nonempty.  It is compact under
Assumption~\ref{ass:lpcc-bounded-polyhedra}, so the branch LP in line~3 admits
an optimal solution and \RBLA{} is well defined.

\begin{lemma}
\label{lem:lpcc-finite-sampling}
Let $\{x^k\}$ and $\{J^k\}$ be, respectively, the iterate and sampled-label
sequences generated by Algorithm~\ref{alg:lpcc-rbla} (\RBLA{}).  With
probability one, the following implication holds simultaneously for every
$N\in\mathbb N$:
\[
\begin{gathered}
    x^k=x^N\quad\forall k\ge N
    \quad\Longrightarrow\\[2pt]
    E_k^J\text{ occurs for infinitely many }k\ge N
    \quad\forall J\in\cP(x^N).
\end{gathered}
\]
\end{lemma}

\begin{proof}
Fix $J\in\cP$, and recall the events $A_k^J$ and $E_k^J$ defined above.  By
\eqref{eq:lpcc-rbla-conditional-uniform},
$\mathbb P(E_k^J\mid\mathcal F_k)=|\cP(x^k)|^{-1}$ when $A_k^J$ occurs and is
zero otherwise.  Equivalently,
\[
    \mathbb P(E_k^J\mid\mathcal F_k)
    =\frac{\mathbf1_{A_k^J}}{|\cP(x^k)|}
    \ge 2^{-m}\mathbf1_{A_k^J},
\]
where the inequality uses $|\cP(x^k)|\le2^m$.  If $A_k^J$ occurs infinitely
often, then $\sum_{k=0}^\infty\mathbf1_{A_k^J}=\infty$, and consequently
\[
    \sum_{k=0}^{\infty}\mathbb P(E_k^J\mid\mathcal F_k)
    \ge
    2^{-m}\sum_{k=0}^{\infty}\mathbf1_{A_k^J}
    =\infty.
\]
The conditional Borel--Cantelli lemma
\cite{Chen1978ConditionalBorelCantelli} therefore shows that, almost surely,
$E_k^J$ also occurs infinitely often whenever $A_k^J$ does.  Thus the desired
implication holds with probability one for each fixed $J$.  Let
$\mathcal E_{\rm ps}$ be the event that, simultaneously for every
$J\in\cP$, $E_k^J$ occurs infinitely often whenever
$A_k^J$ occurs infinitely often.  Since $|\cP|=2^m$, the union of the
finitely many probability-zero failure events, one for each $J\in\cP$, still
has probability zero; hence $\mathbb P(\mathcal E_{\rm ps})=1$.  Fix a
realization in $\mathcal E_{\rm ps}$, suppose that $x^k=x^N$ for every
$k\ge N$, and take
$J\in\cP(x^N)$.  Then $A_k^J$ occurs for every $k\ge N$, so it occurs
infinitely often and therefore so does $E_k^J$.  Thus $E_k^J$ occurs for infinitely
many $k\ge N$.  This proves the stated implication for every such $N$ and
$J$.
\end{proof}

\begin{lemma}
\label{lem:lpcc-rbla-basic-stabilization}
Under Assumption~\ref{ass:lpcc-bounded-polyhedra}, each realization of
Algorithm~\ref{alg:lpcc-rbla} (\RBLA{}) has a finite, realization-dependent
index $N$ such that its iterate sequence satisfies
\[
    x^k=x^N\qquad\forall k\ge N.
\]
\end{lemma}

\begin{proof}
Define
\[
    \mathcal V_{\rm br}
    :=
    \{c\T x^0\}
    \cup
    \{f_J^*:\ J\in\cP,\ \cOmega_J\ne\varnothing\}.
\]
Since $|\cP|=2^m$, the set $\mathcal V_{\rm br}$ is finite.
By the definition of $\mathcal V_{\rm br}$ and the update in \RBLA{},
$c\T x^k\in\mathcal V_{\rm br}$ for every $k\in\mathbb N$.  Since
$\mathcal V_{\rm br}$ is finite and $\{c\T x^k\}$ is nonincreasing, there
exists a finite index $N$ on each realization such that
$c\T x^k=c\T x^N$ for all $k\ge N$.  By the strict-acceptance update in
\RBLA{}, invariance of $c\T x^k$ implies invariance of $x^k$.
Hence $x^k=x^N$ for all $k\ge N$.
\end{proof}

\begin{theorem}
\label{thm:lpcc-rbla-convergence}
Let $\{x^k\}$ be generated by Algorithm~\ref{alg:lpcc-rbla} (\RBLA{}).
Under Assumption~\ref{ass:lpcc-bounded-polyhedra}, let $x^\infty:=x^N$, where
$N$ is as in Lemma~\ref{lem:lpcc-rbla-basic-stabilization}.  Then $x^\infty$
is B-stationary, and hence a local minimizer of the LPCC
\eqref{prob:standard-lcc}, almost surely.
\end{theorem}

\begin{proof}
Recall the event $\mathcal E_{\rm ps}$ defined in the proof of
Lemma~\ref{lem:lpcc-finite-sampling}, which satisfies
$\mathbb P(\mathcal E_{\rm ps})=1$.  Fix a realization in
$\mathcal E_{\rm ps}$, and let $N$ be a stabilization index supplied by
Lemma~\ref{lem:lpcc-rbla-basic-stabilization}.  Then
$x^k=x^\infty=x^N$ for every $k\ge N$, so every
$J\in\cP(x^\infty)$ remains compatible, so $A_k^J$ occurs for every
$k\ge N$; by the definition of $\mathcal E_{\rm ps}$, $E_k^J$ occurs
infinitely often.
Suppose that $x^\infty$ is not B-stationary.
The certificate \eqref{eq:lpcc-full-branch-certificate} then gives a label
$\widetilde J\in\cP(x^\infty)$ such that
$f_{\widetilde J}^*<c\T x^\infty$.  In particular, $E_k^{\widetilde J}$ occurs for infinitely many $k\ge N$.  At any
such iteration, exact solution of the branch LP gives
$c\T\widehat x^k=f_{\widetilde J}^*<c\T x^\infty$, so \RBLA{} accepts
$\widehat x^k$.  This contradicts eventual constancy.  Thus $x^\infty$ is
B-stationary for every realization in $\mathcal E_{\rm ps}$.  Since
$\mathbb P(\mathcal E_{\rm ps})=1$, the claimed almost-sure conclusion follows.
Local minimality follows from
Proposition~\ref{prop:convex-bstationarity-local-minimum}.
\end{proof}

\subsection{Exact-penalty reformulation and connection with DCA}
\label{subsec:lpcc-objective-selection}

At the $k$-th iteration, \RBLA{} solves the sampled full branch LP
$\min\{c\T x\mid x\in\cOmega_{J^k}\}$ to optimality, which can be
computationally expensive.  We therefore seek a method that interleaves
progress in solving the LP with changes of the complementarity branch, without
requiring an exact full branch solve after every change.  Before presenting
such a method, we reinterpret each full branch LP as a penalty problem over a
fixed feasible polyhedron.  For a sufficiently large penalty parameter,
changing the complementarity branch changes only the linear objective, while
the feasible polyhedron remains unchanged.  This viewpoint also reveals the
connection with DCA.

To obtain the reformulation, remove only the complementarity constraints and
define the relaxed polyhedron
\[
    \cC:=\{x\in\R^n\mid
    Ax-a=0,\ Bx-b\ge0,\ x_1\ge0,\ x_2\ge0\}
\]
which contains every branch $\cOmega_J$.  Fix
$J=(D_1,D_2)\in\cP$.  Relative to $\cC$, the additional requirements defining
$\cOmega_J$ are $x_{1,i}=0$ for $i\in D_1$ and $x_{2,i}=0$ for $i\in D_2$.
Their aggregate violation is measured by the linear branch residual
\[
    \psi_J(x):=\sum_{i\in D_1}x_{1,i}+\sum_{i\in D_2}x_{2,i}.
\]
Because $x_1,x_2\ge0$ on $\cC$, the residual $\psi_J$ is nonnegative and
\[
    \cOmega_J
    =\{x\in\cC\mid\psi_J(x)=0\}
    =\{x\in\cC\mid\psi_J(x)\le0\}.
\]
The preceding identity suggests replacing the branch constraints by a penalty
and minimizing $c\T x+\rho\psi_J(x)$ over $\cC$.  Based on the classical
multiplier characterization of exact penalties for convex programs
\cite[Theorem~2.1]{Mangasarian1985ExactPenalty}, the next proposition shows
that, for a sufficiently large $\rho$, this penalized LP and the full LP on
$\cOmega_J$ have exactly the same solutions.

\begin{proposition}
\label{prop:lpcc-branch-penalty-equivalence}
Fix $J\in\cP$.  Suppose that the full branch LP
\[
    f_J^*=\min\{c\T x\mid x\in\cOmega_J\}
\]
is feasible and has a finite optimal value; that is,
$\cOmega_J\ne\varnothing$ and $f_J^*>-\infty$.
Then there exists a finite number $\rho_J\ge0$ such that, for every
$\rho>\rho_J$,
\[
    \argmin_{x\in\cC}\{c\T x+\rho\psi_J(x)\}
    =
    \argmin_{x\in\cOmega_J} c\T x.
\]
\end{proposition}

\begin{proof}
Since $\psi_J(x)\ge0$ on $\cC$, the full branch LP is equivalently
$f_J^*=\min\{c\T x\mid x\in\cC,\ \psi_J(x)\le0\}$.  Associate a multiplier
$\lambda\ge0$ with $\psi_J(x)\le0$.  Its partial Lagrangian is
$L_J(x,\lambda)=c\T x+\lambda\psi_J(x)$, and its dual function is
$g_J(\lambda)=\inf_{x\in\cC}\{c\T x+\lambda\psi_J(x)\}$.

Since the full branch LP is feasible and has a finite optimal value, LP strong
duality \cite[Sections~4.3 and~4.10]{BertsimasTsitsiklis1997} yields an
optimal multiplier $\lambda_J^*\ge0$ such that
$f_J^*=g_J(\lambda_J^*)=\inf_{x\in\cC}
\{c\T x+\lambda_J^*\psi_J(x)\}$.

Choose any $\rho>\lambda_J^*$.  For every $x\in\cC$,
\[
\begin{aligned}
    c\T x+\rho\psi_J(x)
    &=
    c\T x+\lambda_J^*\psi_J(x)
    +(\rho-\lambda_J^*)\psi_J(x)  \\
    &\ge
    f_J^*+(\rho-\lambda_J^*)\psi_J(x).
\end{aligned}
\]
If $x\notin\cOmega_J$, then $\psi_J(x)>0$ and hence
$c\T x+\rho\psi_J(x)>f_J^*$.  If $x\in\cOmega_J$, then $\psi_J(x)=0$ and
$c\T x+\rho\psi_J(x)=c\T x$.  The two problems therefore have the same
minimizers, and $\lambda_J^*$ is an admissible choice of $\rho_J$.
\end{proof}

\begin{remark}
\label{rem:lpcc-hoffman-penalty-bound}
Proposition~\ref{prop:lpcc-branch-penalty-equivalence} uses an optimal
multiplier to obtain an admissible threshold $\rho_J$.  Alternatively, a
sufficient threshold can be expressed in terms of the Hoffman constant of the
linear system defining $\cOmega_J$.  For every nonempty branch, Hoffman's
polyhedral error bound \cite{Hoffman1952,PenaVeraZuluaga2021} yields
$\kappa_J>0$ such that
\[
    \dist_2(x,\cOmega_J)\le\kappa_J\psi_J(x),
    \qquad x\in\cC.
\]
For $y=\proj_{\cOmega_J}(x)$, this gives
$c\T x\ge c\T y-\|c\|_2\|x-y\|_2
\ge f_J^*-\|c\|_2\kappa_J\psi_J(x)$ and hence
$c\T x+\rho\psi_J(x)\ge
f_J^*+(\rho-\|c\|_2\kappa_J)\psi_J(x)$.  Therefore
$\|c\|_2\kappa_J$ is another admissible choice of $\rho_J$: every
$\rho>\|c\|_2\kappa_J$ excludes points outside $\cOmega_J$ from the
penalized solution set, while the penalty vanishes on $\cOmega_J$.  Since
$\cP$ is finite, the single bound
\[
    \rho>
    \|c\|_2\max\{\kappa_J\mid
    J\in\cP,\ \cOmega_J\ne\varnothing\}
\]
is sufficient simultaneously for all nonempty branches.
In practice, useful estimates of the Hoffman constants $\kappa_J$, and
hence of a sufficient penalty threshold, are difficult to obtain.
\end{remark}

The branchwise exact-penalty construction above is closely related to
existing exact-penalty formulations for LPCCs.  Such formulations were studied
for programs with linear complementarity constraints in
\cite{mangasarian1997exact}, while the piecewise-linear LPCC penalty used below
was analyzed in \cite{JaraMoroniPangWachter2018LPCC}.  Following this latter
formulation, define
\[
    \Phi(x)
    :=\sum_{i=1}^m\min\{x_{1,i},x_{2,i}\},
    \qquad x\in\cC.
\]
For every $x\in\cC$, the branch residuals satisfy
$\Phi(x)=\min_{J\in\cP}\psi_J(x)$.  Hence
$\cOmega=\{x\in\cC\mid\Phi(x)=0\}$ and, for $x\in\cOmega$, the active affine
pieces are exactly the compatible branches:
$\argmin_{J\in\cP}\psi_J(x)=\cP(x)$.  For any $\rho>0$, consider the penalty
problem
\begin{equation}\label{eq:lpcc-piecewise-linear-penalty}
    \min_{x\in\cC}\ F_\rho(x),
    \qquad
    F_\rho(x):=c\T x+\rho\Phi(x)
    =\min_{J\in\cP}\{c\T x+\rho\psi_J(x)\}.
\end{equation}
\cite[Proposition~3]{JaraMoroniPangWachter2018LPCC}
establishes local
exactness: an LPCC-feasible point is a local minimizer of the LPCC if and only
if it is a local minimizer of the penalty problem for all sufficiently large
penalty parameters.

The piecewise-linear penalty has a natural DC structure since, for every
$i\in[m]$, $\min\{x_{1,i},x_{2,i}\}=\allowbreak
x_{1,i}+x_{2,i}-\max\{x_{1,i},x_{2,i}\}$.  It follows that
\[
    F_\rho(x)=G_\rho(x)-\rho H(x).
\]
Here $G_\rho(x):=c\T x+\rho\sum_{i=1}^m(x_{1,i}+x_{2,i})$ and
$H(x):=\sum_{i=1}^m\max\{x_{1,i},x_{2,i}\}$ are convex.  This places the
penalty problem within the standard DCA framework \cite{tao1997convex}.
The approach has been developed for complementarity penalties in LCPs
\cite{le2011solving} and LPCCs \cite{JaraMoroniPangWachter2018LPCC}.  At
$x^k$, DCA chooses
$\eta^k\in\partial H(x^k)$ and solves
\[
    \min_{x\in\cC}\{G_\rho(x)-\rho(\eta^k)\T x\}.
\]
To make this subproblem explicit, let $h(u,v):=\max\{u,v\}$.  If
$x^k\in\cOmega$, the components of $\eta^k$ associated with the $i$th
complementarity pair satisfy
\[
(\eta_{1,i}^k,\eta_{2,i}^k)\T
\in\partial h(x_{1,i}^k,x_{2,i}^k)
=
\begin{cases}
\{(1,0)\T\},&i\in\cI_{+0}(x^k),\\
\{(0,1)\T\},&i\in\cI_{0+}(x^k),\\
\{(1-\xi,\xi)\T\mid\xi\in[0,1]\},&i\in\cI_{00}(x^k).
\end{cases}
\]
Choosing $(\eta_{1,i}^k,\eta_{2,i}^k)=(1-\xi_i^k,\xi_i^k)$ at each biactive
index gives the auxiliary LP
\[
\begin{aligned}
    \min_{x\in\cC}\quad c\T x+\rho\Bigg(&
    \sum_{i\in\cI_{0+}(x^k)}x_{1,i}
    +\sum_{i\in\cI_{+0}(x^k)}x_{2,i}\\
    &+\sum_{i\in\cI_{00}(x^k)}
    \bigl(\xi_i^k x_{1,i}+(1-\xi_i^k)x_{2,i}\bigr)\Bigg).
\end{aligned}
\]
This auxiliary LP is equivalent to the formulation in
\cite[Section~4.1.2, equation~(12)]{JaraMoroniPangWachter2018LPCC}.
For the extreme choices $\xi_i^k\in\{0,1\}$, define
$J^k:=\bigl(
    \cI_{0+}(x^k)\cup
    \{i\in\cI_{00}(x^k)\mid\xi_i^k=1\},
    \cI_{+0}(x^k)\cup
    \{i\in\cI_{00}(x^k)\mid\xi_i^k=0\}
    \bigr)\in\cP(x^k)$.
The DCA subproblem then becomes
the subproblem in \RBLA{}:
\[
    \min_{x\in\cC}\{c\T x+\rho\psi_{J^k}(x)\}.
\]
Whenever
$\rho>\bar\rho:=\max\{\rho_J\mid
J\in\cP,\ \cOmega_J\ne\varnothing\}$,
Proposition~\ref{prop:lpcc-branch-penalty-equivalence} shows that this
auxiliary LP has the same solution set as the
branch LP on $\cOmega_{J^k}$.
Thus, at the subproblem level, Algorithm~\ref{alg:lpcc-rbla} samples a
compatible extreme subgradient and solves the corresponding DCA auxiliary LP
exactly.  The equivalence concerns only these subproblems: \RBLA{}
accepts strict improvement in the original objective and resamples after an
unsuccessful choice, whereas standard DCA has different update and stopping
rules.

\begin{remark}
\cite[Proposition~3]{JaraMoroniPangWachter2018LPCC} relates local minimizers
of the LPCC and the full disjunctive penalty only at LPCC-feasible points,
whereas fixing $J$ replaces the disjunctive residual
$\Phi=\min_{J'\in\cP}\psi_{J'}$ with the linear residual $\psi_J$, allowing
Proposition~\ref{prop:lpcc-branch-penalty-equivalence} to identify the global
minimizer sets without assuming a priori that a penalized minimizer is
feasible.
\end{remark}

\begin{remark}
Here $x\in\cC$ is d-stationary for $F_\rho$ on $\cC$ if
$F_\rho'(x;d)\ge0$ for every $d\in T_{\cC}(x)$.  For $x\in\cOmega$ and
$\rho>\bar\rho$,
\cite[Propositions~2 and~3]{JaraMoroniPangWachter2018LPCC}
show that $x$ is B-stationary for the
LPCC if and only if it is d-stationary for $F_\rho$ on $\cC$.
\end{remark}

\subsection{Randomized Penalty-objective Simplex search Algorithm}
\label{subsec:lpcc-simplex}

We now use the preceding fixed-polyhedron reformulation to construct an
inexact alternative to \RBLA{}.  For $J\in\cP(x)$ and sufficiently large
$\rho$, solving the full LP on $\cOmega_J$ is equivalent to minimizing
$\phi_J(u):=c\T u+\rho\psi_J(u)$ over $u\in\cC$.  Instead of solving the
penalty LP to optimality, the new method accepts the first visited vertex
$z\in\ext\cC$ satisfying $\phi_J(z)<\phi_J(x)$ and $\psi_J(z)=0$.  Since
$\psi_J(x)=0$, the accepted point lies in $\cOmega_J\subseteq\cOmega$ and
satisfies $c\T z<c\T x$.  A compatible branch is then sampled at the new
iterate; an unsuccessful search retains $x$ and samples again from $\cP(x)$.

The key computational point is that every penalty LP
$\min_{u\in\cC}\{c\T u+\rho\psi_J(u)\}$ has the same feasible polyhedron
$\cC$; changing $J$ changes only its linear objective.
Accordingly, after each objective change, \RPSA{} performs
primal-simplex pivots from the current iterate using a finitely terminating
pivot rule \cite{Bland1977Pivoting,BertsimasTsitsiklis1997}.  When
$\rho>\bar\rho:=\max\{\rho_J\mid J\in\cP,\ \cOmega_J\ne\varnothing\}$, where
$\rho_J$ is given by Proposition~\ref{prop:lpcc-branch-penalty-equivalence},
every penalty-LP optimizer lies in $\cOmega_J$.  Hence, if the sampled branch
admits improvement, the simplex search finds an acceptable vertex after
finitely many pivots, at the latest at a penalty-LP optimizer.  Implementation
details are reported in
Section~\ref{sec:numerical-experiments}, and the method is stated in
Algorithm~\ref{alg:lpcc-rpsa}.

\begin{breakablealgorithm}
\caption{Randomized Penalty-objective Simplex search Algorithm (\RPSA)}
\label{alg:lpcc-rpsa}
\begin{algorithmic}[1]
\Require feasible point $x^0\in\cOmega$ and penalty parameter
    $\rho>\bar\rho$
\Statex \textit{Initialization}
\If{$x^0\notin\ext\cC$}
    \State Choose any $J^{\rm init}\in\cP(x^0)$
    \State Replace $x^0$ by an optimal vertex of the full LP on
        $\cOmega_{J^{\rm init}}$
\EndIf
\State Store a primal-feasible basis of $\cC$ representing $x^0$
\Statex \textit{Randomized branch--simplex iterations}
\For{$k=0,1,2,\ldots$}
    \State Sample $\xi_i^k\in\{0,1\}$ conditionally independently and
        uniformly for $i\in\cI_{00}(x^k)$
    \State Form $J^k\gets(D_1^k,D_2^k)\in\cP(x^k)$, where
    \Statex \hspace{\algorithmicindent}%
        $D_1^k\gets\cI_{0+}(x^k)\cup
        \{i\in\cI_{00}(x^k)\mid\xi_i^k=1\}$ and
        $D_2^k\gets[m]\setminus D_1^k$
    \State Set $\phi_k(u)\gets c\T u+\rho\psi_{J^k}(u)$
    \State Take primal-simplex steps from the stored basis until
    \Statex \hspace{\algorithmicindent}%
        the first vertex $z$ with $\phi_k(z)<\phi_k(x^k)$ and
        $\psi_{J^k}(z)=0$ is found, or stop at penalty-LP optimality
    \If{an acceptable vertex $z$ is found}
        \State Set $x^{k+1}\gets z$ and store its current basis
    \Else
        \State Set $x^{k+1}\gets x^k$ and retain its stored basis
    \EndIf
\EndFor
\end{algorithmic}
\end{breakablealgorithm}

\begin{remark}
Lines~1--4 of Algorithm~\ref{alg:lpcc-rpsa} initialize a simplex vertex when
the supplied $x^0$ is not a vertex of $\cC$.  This step is well defined under
Assumption~\ref{ass:lpcc-bounded-polyhedra}.  For
$J^{\rm init}\in\cP(x^0)$, the branch $\cOmega_{J^{\rm init}}$ contains
$x^0$ and is a bounded subset of $\cOmega$.  Its full LP therefore has an
optimal vertex $v^0$ with $c\T v^0\le c\T x^0$
\cite{BertsimasTsitsiklis1997}.  Since
$\psi_{J^{\rm init}}\ge0$ on $\cC$, the set
$\cOmega_{J^{\rm init}}=\{x\in\cC\mid\psi_{J^{\rm init}}(x)=0\}$ is a face
of $\cC$.  Every vertex of a face is a vertex of the ambient polyhedron, so
$v^0\in\ext\cC$ and admits a primal-feasible basis
\cite{BertsimasTsitsiklis1997}.
\end{remark}

\begin{remark}
Although $\bar\rho$ is difficult to estimate, $\rho$ can be increased
adaptively as follows while retaining the almost-sure B-stationarity
guarantee. The analogous convergence proof is omitted.  In practice, one may start with a moderate
$\rho_0$ and increase it gradually.  If the sampled penalty LP is unbounded,
or if its optimal solution $\widehat u^k$ satisfies
$\psi_{J^k}(\widehat u^k)>0$, one may set $\rho\gets\pi\rho$ for some
$\pi>1$ and restart the same branch search.
\end{remark}

For \RPSA{}, we use the same probability space and filtration notation
$(\Xi,\mathcal G,\mathbb P,\{\mathcal F_k\}_{k\in\mathbb N})$ as in
Subsection~\ref{subsec:lpcc-exact-branch}, with $\mathcal F_k$ containing all
information available immediately before the draw at the $k$-th iteration. Thus,
$x^k$ and $\cP(x^k)$ are
$\mathcal F_k$-measurable, whereas $J^k$ is
$\mathcal F_{k+1}$-measurable. For the iterate sequence $\{x^k\}$ and sampled-label sequence
$\{J^k\}$ generated by \RPSA{}, we reuse the notation
\[
    A_k^J:=\{J\in\cP(x^k)\},\qquad E_k^J:=\{J^k=J\},
    \qquad J\in\cP.
\]
Here $A_k^J$ is the event that branch $J$ is compatible with $x^k$, and
$E_k^J$ is the event that branch $J$ is sampled at the $k$-th iteration.  Conditional on $\mathcal F_k$, the variables
$\{\xi_i^k\}_{i\in\cI_{00}(x^k)}$ are independent and uniform on $\{0,1\}$;
i.e., for every $J\in\cP$,
\begin{equation}\label{eq:lpcc-rpsa-conditional-uniform}
\mathbb P(E_k^J\mid\mathcal F_k)=\frac{\mathbf1_{A_k^J}}{|\cP(x^k)|}.
\end{equation}

\begin{lemma}
\label{lem:lpcc-rpsa-basic-stabilization}
Under Assumption~\ref{ass:lpcc-bounded-polyhedra}, each realization of
Algorithm~\ref{alg:lpcc-rpsa} (\RPSA{}) has an iterate sequence satisfying
\[
    x^k\in\ext\cC\cap\cOmega,
    \qquad
    c\T x^k\le c\T x^{k-1},
    \qquad k\ge1.
\]
The inequality is strict whenever $x^k\ne x^{k-1}$.  Every realization has a
finite, realization-dependent index $N$ such that
\[
    x^k=x^N,\qquad k\ge N.
\]
\end{lemma}

\begin{proof}
The initialization in Algorithm~\ref{alg:lpcc-rpsa} gives
$x^0\in\ext\cC\cap\cOmega$.  Suppose
$x^k\in\ext\cC\cap\cOmega$.  If the search is unsuccessful, then
$x^{k+1}=x^k$.  Otherwise, Algorithm~\ref{alg:lpcc-rpsa} accepts a candidate
vertex $z\in\ext\cC$ only when $\psi_{J^k}(z)=0$ and
$\phi_k(z)<\phi_k(x^k)$, and sets $x^{k+1}=z$.  Since
$z\in\cC$ and $\psi_{J^k}(z)=0$, we have
$z\in\cOmega_{J^k}\subseteq\cOmega$.  Hence
$x^{k+1}\in\ext\cC\cap\cOmega$, proving the first assertion by induction.
Moreover, $J^k\in\cP(x^k)$ gives $\psi_{J^k}(x^k)=0$, and therefore every
accepted $z$ satisfies
\[
    c\T z
    =
    \phi_k(z)
    <
    \phi_k(x^k)
    =
    c\T x^k.
\]
Thus $c\T x^{k+1}<c\T x^k$ whenever $x^{k+1}\ne x^k$.  Since $\cC$ is a
polyhedron defined by finitely many linear constraints, $\ext\cC$ is finite
\cite[Theorem~2.3 and Corollary~2.1]{BertsimasTsitsiklis1997}.  Strict
decrease makes all accepted vertices pairwise distinct.  Hence only finitely
many changes are possible; after the last one, the update rule keeps $x^k$
fixed.
\end{proof}

\begin{lemma}\label{lem:lpcc-rpsa-persistent-sampling}
Let $\{x^k\}$ and $\{J^k\}$ be, respectively, the iterate and sampled-label
sequences generated by Algorithm~\ref{alg:lpcc-rpsa} (\RPSA{}).  With
probability one, the following implication holds simultaneously for every
$N\in\mathbb N$:
\[
\begin{gathered}
    x^k=x^N\quad\forall k\ge N
    \quad\Longrightarrow\\[2pt]
    E_k^J\text{ occurs for infinitely many }k\ge N
    \quad\forall J\in\cP(x^N).
\end{gathered}
\]
\end{lemma}

\begin{proof}
The conditional sampling rules
\eqref{eq:lpcc-rbla-conditional-uniform} and
\eqref{eq:lpcc-rpsa-conditional-uniform} are identical.  Therefore the proof
of Lemma~\ref{lem:lpcc-finite-sampling}, with the sequences generated by
\RPSA{} in place of those generated by \RBLA{}, applies verbatim.  Let
$\mathcal E_{\rm ps}^{\rm R}$ denote the resulting probability-one event on
which the displayed implication holds for every $N\in\mathbb N$.

\end{proof}

\begin{theorem}
\label{thm:lpcc-rpsa-convergence}
Let $\{x^k\}$ be generated by Algorithm~\ref{alg:lpcc-rpsa} (\RPSA{}).
Under Assumption~\ref{ass:lpcc-bounded-polyhedra}, suppose
$\rho>\bar\rho:=\max\{\rho_J\mid J\in\cP,\ \cOmega_J\ne\varnothing\}$, where
each $\rho_J$ is defined in
Proposition~\ref{prop:lpcc-branch-penalty-equivalence}, and let
$x^\infty:=x^N$, where
$N$ is as in Lemma~\ref{lem:lpcc-rpsa-basic-stabilization}.  Then $x^\infty$ is
B-stationary, and hence a local minimizer of the LPCC
\eqref{prob:standard-lcc}, almost surely.
\end{theorem}

\begin{proof}
Recall the event $\mathcal E_{\rm ps}^{\rm R}$ defined in the proof of
Lemma~\ref{lem:lpcc-rpsa-persistent-sampling}, which satisfies
$\mathbb P(\mathcal E_{\rm ps}^{\rm R})=1$.  Fix a realization in
$\mathcal E_{\rm ps}^{\rm R}$, and let $N$ be a stabilization index supplied
by Lemma~\ref{lem:lpcc-rpsa-basic-stabilization}.  Then
$x^k=x^\infty=x^N$ for every $k\ge N$, so every
$J\in\cP(x^\infty)$ remains compatible, so $A_k^J$ occurs for every
$k\ge N$; by the definition of $\mathcal E_{\rm ps}^{\rm R}$, $E_k^J$ occurs
infinitely often.

Suppose that $x^\infty$ is not B-stationary.  The certificate
\eqref{eq:lpcc-full-branch-certificate} then gives
$\widetilde J\in\cP(x^\infty)$ such that
$f_{\widetilde J}^*<c\T x^\infty$.  Whenever $E_k^{\widetilde J}$ occurs with $k\ge N$, Proposition~\ref{prop:lpcc-branch-penalty-equivalence}
shows that the sampled penalty LP has optimal value
$f_{\widetilde J}^*<c\T x^\infty=\phi_k(x^\infty)$ and that every optimizer
lies in $\cOmega_{\widetilde J}$.  By line~10 of
Algorithm~\ref{alg:lpcc-rpsa}, the finitely terminating simplex search visits
successive vertices until it finds an acceptable one or reaches an optimal
vertex $v$.  In the latter case,
$\psi_{\widetilde J}(v)=0$ and
$\phi_k(v)=f_{\widetilde J}^*<\phi_k(x^\infty)$, so $v$ satisfies both tests
in line~10 and is accepted.  Hence \RPSA{} changes the iterate, contradicting
eventual constancy.  Thus $x^\infty$ is B-stationary
for every realization in $\mathcal E_{\rm ps}^{\rm R}$.  Since
$\mathbb P(\mathcal E_{\rm ps}^{\rm R})=1$, the claimed almost-sure conclusion
follows.  Local minimality follows from
Proposition~\ref{prop:convex-bstationarity-local-minimum}.
\end{proof}

\section{Thresholded Branch Trust-Region Method for QPCCs}
\label{sec:qpcc-linearized}

This section considers the QPCC case
$f(x)=\frac12x\T Qx+q\T x$, where $Q=Q\T$ may be indefinite.
We present the Thresholded Branch Trust-Region Algorithm (\TBTA) and
establish its almost-sure B-stationarity guarantees. We use the following
compactness assumption.

\begin{assumption}\label{ass:compact}
The set
\[
    \cL(x^0):=
    \cOmega\cap\{x\in\mathbb R^n\mid f(x)\le f(x^0)\}
\]
is compact.
\end{assumption}

\subsection{\TBTA: Thresholded Branch Trust-Region Algorithm}
\label{sec:threshold-method}

To find a B-stationary point by a direct extension of the exact branch method,
one would solve, for each sampled branch $\cOmega_J$, the branch QP
\[
    \min_{x\in\cOmega_J} f(x),
    \qquad
    f(x):=\frac12x\T Qx+q\T x,
\]
to optimality before changing branches.
We use a trust-region technique with the first-order objective model
$f(y)+\grad f(y)\T d$ \cite{conn2000trust,nocedal2006numerical}.
On each selected branch, retaining the linear constraints and imposing
an infinity-norm trust region yields the following LP subproblem.

For a base point $y\in\R^n$, a branch $\cOmega_J$ indexed by
$J=(D_1,D_2)\in\cP$, and $\Delta>0$, let $\LP(y,J,\Delta)$ denote
the linear program
\begin{equation}\label{eq:cross-lp}
\begin{aligned}
    \LP(y,J,\Delta):\qquad
    \min_{d\in\R^n}\quad
        & \grad f(y)\T d \\
    \text{s.t.}\quad
        & A(y+d)-a=0, \\
        & B(y+d)-b\ge0, \\
        & y_{1,i}+d_{1,i}=0,\quad y_{2,i}+d_{2,i}\ge0,
          && i\in D_1, \\
        & y_{1,i}+d_{1,i}\ge0,\quad y_{2,i}+d_{2,i}=0,
          && i\in D_2, \\
        & \|d\|_\infty\le \Delta.
\end{aligned}
\end{equation}
We solve \eqref{eq:cross-lp} to optimality, or detect infeasibility,
using a standard LP solver (Gurobi in our implementation).
A natural inexact strategy is therefore to use \eqref{eq:cross-lp} for the
branch search in place of an exact bQP solve, and to sample another branch
once a step is accepted.  Although this reduces every branch search to an LP,
it may miss descent branches required for B-stationarity.  At a finite
iterate, a complementarity pair with one nonzero component admits only one
compatible assignment; if that component converges to zero, the pair becomes
biactive and additional branches appear only at the limit.  Consequently, a
generated sequence may satisfy
\[
    x^k\to\bar x,
    \qquad
    \cI_{00}(x^k)\subsetneq\cI_{00}(\bar x),
    \qquad
    \cP(x^k)\subsetneq\cP(\bar x).
\]
In this case, a label
$J_{\rm dec}\in\cP(\bar x)\setminus\cP(x^k)$ whose branch may decrease the
objective cannot be sampled at the current iterate.

\begin{example}\label{ex:thresholding-toy}
Consider
\[
    \min_{x,y}\ x^2-y
    \quad\text{s.t.}\quad
    0\le x\perp y\ge0,\qquad x\le1,\quad y\le1.
\]
The feasible sequence $z^k=(2^{-k},0)$ converges to $\bar z=(0,0)$.  Only the
branch $y=0$ is compatible at every $z^k$, whereas at $\bar z$ the additional
branch $x=0$ becomes compatible and admits the descent direction $(0,1)$,
with directional derivative $-1$.  Hence $\bar z$ is not B-stationary,
although this descent branch is unavailable at every finite iterate.
\end{example}

The example shows that sampling only from $\cP(x^k)$ may fail to examine a
branch that becomes compatible only at an accumulation point.  To expose
such branches earlier, we introduce thresholded auxiliary sampling.  A
nonbiactive complementarity pair whose nonzero component is small is treated
as biactive solely for generating an auxiliary branch.  Fix $\tau>0$, let
$x=(x_0,x_1,x_2)\in\cOmega$, and define
\[
    \cS_\tau(x)
    :=
    \left\{i\notin\cI_{00}(x)\;\middle|\;
    \|(x_{1,i},x_{2,i})\|\le \tau\right\}.
\]
The thresholded vector $T_\tau(x)$ is
\[
    (T_\tau(x))_0=x_0,\qquad
    (T_\tau(x))_{j,i}=
    \begin{cases}
        0, & i\in\cS_\tau(x),\\
        x_{j,i}, & i\notin\cS_\tau(x),
    \end{cases}
    \quad j=1,2.
\]
The vector $T_\tau(x)$ is used only to enlarge the set of candidate branches
and need not belong to $\cOmega$.

At the $k$-th iteration, the method first samples a primary branch
$J^k\in\cP(x^k)$ and applies the linearized trust-region search from $x^k$.
If this search has positive predicted decrease and
$\cS_\tau(x^k)\ne\varnothing$, it also samples an auxiliary branch
$\widehat J^k\in\cP(T_\tau(x^k))$ and applies the same search from $x^k$.
The auxiliary trial is selected only if its search succeeds and its original
objective value is smaller than that of the primary trial; otherwise the
primary trial is retained.  Hence branch assignments associated with nearly
biactive pairs can be tested before they become active at a limit point.  The
ratio test described next converts predicted descent from either search into
actual objective decrease.

To express the repeated solution of the linearized branch LP, the ratio test,
and the radius updates compactly, we use Subalgorithm~\ref{alg:lbtrs} for
both the primary and auxiliary searches.  Given a base point $y$, a branch
$\cOmega_J$, and an initial radius $\delta_0$, the subalgorithm successively
solves $\LP(y,J,\delta_\ell)$.  For an optimal step $d_\ell$ with positive
predicted decrease, set
\[
    \pred_\ell:=-\grad f(y)\T d_\ell,
    \qquad
    \ared_\ell:=f(y)-f(y+d_\ell),
    \qquad
    \varrho_\ell:=\frac{\ared_\ell}{\pred_\ell}.
\]
The identity
$f(y+d_\ell)-f(y)=\grad f(y)\T d_\ell+\frac12d_\ell\T Qd_\ell$
gives $\varrho_\ell=1-\frac{d_\ell\T Qd_\ell}{2\pred_\ell}$.
For a descent branch, contracting the radius
makes the quadratic model error negligible relative to the predicted
decrease, so $\varrho_\ell\to1$.  This motivates accepting the trial when
$\varrho_\ell\ge\eta_1$ and contracting the radius otherwise.

Fix the \LBTRS{} parameters $0<\eta_1\le\eta_2<1$,
$0<\gamma_{\rm dec}<1<\gamma_{\rm inc}$, and $\Delta_{\min}>0$.
A rejected trial uses
$\delta_{\ell+1}=\gamma_{\rm dec}\delta_\ell$.  A successful call returns
the accepted point together with the next initial radius
\begin{equation}\label{eq:radius-update}
    \delta^+ =
    \begin{cases}
    \max\{\Delta_{\min},\gamma_{\rm inc}\delta_\ell\},
        & \pred_\ell>0,
          \varrho_\ell\ge\eta_2,
          \text{ and }\|d_\ell\|_\infty=\delta_\ell,\\[1mm]
    \max\{\Delta_{\min},\delta_\ell\},
        & \text{otherwise}.
    \end{cases}
\end{equation}

\begin{breakablesubalgorithm}
\caption{Linearized branch trust-region search (\LBTRS)}
\label{alg:lbtrs}
\begin{algorithmic}[1]
\Statex \textbf{Input:} base point \(y\in\R^n\) (not necessarily in
    \(\cOmega_J\)), label
    \(J\), and radius
    \(\delta_0\ge\Delta_{\min}\)
\Statex \textbf{Output:} status \(\mathsf{success}\), \(\mathsf{zero}\),
    or \(\mathsf{failure}\); on \(\mathsf{success}\), a point
    \(y^+\in\cOmega_J\) and radius \(\delta^+\)
\For{\(\ell=0,1,2,\ldots\)}
    \State Solve \(\LP(y,J,\delta_\ell)\), obtaining an optimal step
        \(d_\ell\), or certify infeasibility
    \If{the LP is infeasible}
        \State \Return \(\mathsf{failure}\)
    \EndIf
    \State Set \(\pred_\ell=-\grad f(y)\T d_\ell\)
    \If{\(\ell=0\) and \(\pred_\ell=0\)}
        \State \Return \(\mathsf{zero}\)
    \ElsIf{\(\pred_\ell\le0\)}
        \State \Return \(\mathsf{failure}\)
    \EndIf
    \State Set \(y_\ell=y+d_\ell\),
        \(\ared_\ell=f(y)-f(y_\ell)\), and
        \(\varrho_\ell=\ared_\ell/\pred_\ell\)
    \If{\(\varrho_\ell\ge\eta_1\)}
        \State Set \(\delta^+\) by \eqref{eq:radius-update}
        \State \Return \(\mathsf{success}\) with point \(y^+=y_\ell\)
            and radius \(\delta^+\)
    \EndIf
    \State Set \(\delta_{\ell+1}=\gamma_{\rm dec}\delta_\ell\)
\EndFor
\end{algorithmic}
\end{breakablesubalgorithm}

A successful call returns a point $y^+\in\cOmega_J$ with $f(y^+)<f(y)$.
A failed call means only that this branch search produced no acceptable trial;
in particular, failure of an auxiliary search is not a stationarity
certificate.  For a compatible primary branch, the call returns either
$\mathsf{zero}$ or $\mathsf{success}$; see
Corollary~\ref{cor:primary-well-defined}.  Combining \LBTRS{} with the primary
and auxiliary sampling rules gives the complete method below.

\begin{breakablealgorithm}
\caption{Thresholded Branch Trust-Region Algorithm (\TBTA)}
\label{alg:threshold-tr}
\begin{algorithmic}[1]
\Require feasible \(x^0\in\cOmega\), radius
  \(\Delta_0\)%
    \(\ge\Delta_{\min}\), threshold \(\tau>0\), and the
    \LBTRS{} constants
\For{\(k=0,1,2,\ldots\)}
    \State Choose \(J^k\) conditionally uniformly from \(\cP(x^k)\)
    \State \textbf{Primary search:} Apply Subalgorithm~\ref{alg:lbtrs} to
        \((x^k,J^k,\Delta_k)\)
    \If{the primary search returns \(\mathsf{zero}\)}
        \If{\(\cI_{00}(x^k)=\varnothing\)}
            \State \Return \(x^k\)
        \EndIf
        \State Set \((x^{k+1},\)%
           \(\Delta_{k+1}\)%
            \()=(x^k,\Delta_k)\)
        \State \textbf{continue} with iteration \(k+1\)
    \EndIf
    \State Denote the returned point and radius by
        \((x_P^k,\)%
      \(\Delta_P^k\)%
        \()\)
    \State Set \((x_A^k,\)%
     \(\Delta_A^k\)%
        \()=(x^k,\Delta_k)\) and
        \(\cS_\tau^k=\cS_\tau(x^k)\)
    \If{\(\cS_\tau^k\ne\varnothing\)}
        \State Choose \(\widehat J^k\) conditionally uniformly from
            \(\cP(T_\tau(x^k))\)
        \State \textbf{Auxiliary search:} Apply Subalgorithm~\ref{alg:lbtrs}
            to \((x^k,\widehat J^k,\Delta_k)\)
        \If{the auxiliary search is successful}
            \State Set \((x_A^k,\)%
           \(\Delta_A^k\)%
                \()\) to its returned point and radius
        \EndIf
    \EndIf
    \If{\(f(x_A^k)<f(x_P^k)\)}
        \State Set \((x^{k+1},\)%
        \(\Delta_{k+1}\)%
            \()=(x_A^k,\)%
          \(\Delta_A^k\)%
            \()\)
    \Else
        \State Set \((x^{k+1},\)%
        \(\Delta_{k+1}\)%
            \()=(x_P^k,\)%
          \(\Delta_P^k\)%
            \()\)
    \EndIf
\EndFor
\end{algorithmic}
\end{breakablealgorithm}

\begin{remark}\label{rem:primary-call}
Let $d_0^k$ be an optimal step of the initial primary LP
$\LP(x^k,J^k,\Delta_k)$.  Since $J^k\in\cP(x^k)$, the zero step is feasible,
so $-\grad f(x^k)\T d_0^k\ge0$.  A $\mathsf{zero}$ return means equality and
certifies first-order stationarity on the sampled branch.  If
$\cI_{00}(x^k)=\varnothing$, this is the unique compatible branch, so
Proposition~\ref{prop:branch-lp-certificate} gives B-stationarity.  If instead
$-\grad f(x^k)\T d_0^k>0$, the \LBTRS{} in the primary search terminates
successfully after finitely many inner iterations, as shown in
Corollary~\ref{cor:primary-well-defined}.
\end{remark}

\subsection{Convergence Analysis}
\label{sec:threshold-convergence}

For the convergence analysis of Algorithm~\ref{alg:threshold-tr} (\TBTA{}),
all random quantities in this subsection are defined on a common probability space
$(\Xi,\mathcal G,\mathbb P)$.  Here $\Xi$ consists of the possible primary
and auxiliary branch-selection sequences, $\mathcal G$ is the sigma-algebra
generated by these selections, and $\mathbb P$ is their probability law.  Let
$\{\mathcal F_k^P\}_{k\in\mathbb N}$ and
$\{\mathcal F_k^A\}_{k\in\mathbb N}$ be the filtrations representing the
information available immediately before the primary draw and the potential
auxiliary draw at the $k$-th iteration, respectively.  They satisfy
$\mathcal F_k^P\subseteq\mathcal F_k^A\subseteq\mathcal F_{k+1}^P$.
In particular, $x^k$ and $\cP(x^k)$ are $\mathcal F_k^P$-measurable, while
the primary draw, the primary-stage computations, the decision whether to
perform the auxiliary draw, and $\cP(T_\tau(x^k))$ are
$\mathcal F_k^A$-measurable.  The auxiliary draw and the remaining outputs of
iteration $k$ are $\mathcal F_{k+1}^P$-measurable.  Conditional on the
corresponding history, every performed draw is uniform.  Thus, for each
$J\in\cP$,
\begin{equation}\label{eq:tbta-primary-uniform}
\mathbb P(J^k=J\mid\mathcal F_k^P)
=
\begin{cases}
|\cP(x^k)|^{-1},&J\in\cP(x^k),\\
0,&J\notin\cP(x^k).
\end{cases}
\end{equation}
When no auxiliary draw is performed, define $\widehat J^k$ arbitrarily.  Let
$D_k^A\in\mathcal F_k^A$ denote the event that the auxiliary draw is
performed at the $k$-th iteration.  On this event,
\begin{equation}\label{eq:tbta-auxiliary-uniform}
\mathbb P(\widehat J^k=J\mid\mathcal F_k^A)
=
\begin{cases}
|\cP(T_\tau(x^k))|^{-1},&J\in\cP(T_\tau(x^k)),\\
0,&J\notin\cP(T_\tau(x^k)).
\end{cases}
\end{equation}
Probability-one statements below refer to $\mathbb P$.

The main convergence result is stated next; its proof is completed after the
supporting results.

\begin{theorem}\label{thm:threshold-bstationary}
Let $\{x^k\}$ be generated by Algorithm~\ref{alg:threshold-tr} (\TBTA{})
with fixed $\tau>0$ and $\Delta_{\min}>0$.  Suppose
Assumption~\ref{ass:compact} holds.
\begin{enumerate}[label=(\roman*)]
    \item If the algorithm terminates, its returned point is B-stationary.
    \item If the iteration is infinite, then $\{x^k\}$ has accumulation points and,
    with probability one, every accumulation point is B-stationary.
\end{enumerate}
\end{theorem}

Part~(i) follows directly from the stopping test in
Algorithm~\ref{alg:threshold-tr}.  Figure~\ref{fig:threshold-proof-roadmap}
summarizes the proof structure.  The supporting results below address the
main difficulty: the infinite-run assertion in part~(ii).

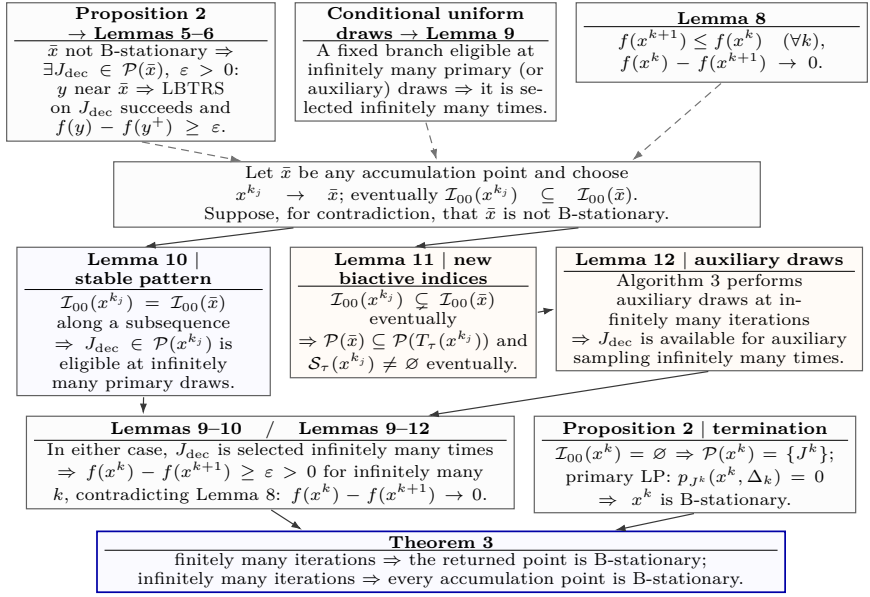
\begin{figure}[tbp]
\centering
\begingroup
\newcommand{\roadmapentry}[2]{%
    {\color{black}\bfseries #1}\par
    \vspace{0.25ex}\hrule height 0.3pt\vspace{0.35ex}%
    #2}
\begin{tikzpicture}[
    roadmap box/.style={
        draw=black!62,
        fill=black!1,
        line width=0.45pt,
        align=center,
        inner xsep=1mm,
        inner ysep=0.6mm,
        font=\fontsize{6.5}{7.2}\selectfont
    },
    stable box/.style={roadmap box,fill=blue!2},
    new box/.style={roadmap box,fill=orange!3},
    theorem box/.style={
        roadmap box,
        draw=blue!65!black,
        fill=blue!2,
        line width=0.65pt
    },
    roadmap arrow/.style={
        -{Latex[length=1.45mm,width=0.95mm]},
        draw=black!78,
        line width=0.45pt,
        shorten <=0.4pt,
        shorten >=0.4pt
    },
    input arrow/.style={roadmap arrow,densely dashed,draw=black!55}
]

\node[roadmap box,text width=3.35cm,minimum height=10.5mm] (decrease) {%
    \roadmapentry{Proposition~\ref{prop:branch-lp-certificate}
        $\to$ Lemmas~\ref{lem:nonB-proj}--\ref{lem:cross-uniform}}{%
        $\bar x$ not B-stationary
        $\Rightarrow \exists J_{\rm dec}\in\cP(\bar x),\ \varepsilon>0$:\par
        $y$ near $\bar x$ $\Rightarrow$ \LBTRS{} on $J_{\rm dec}$ succeeds and
        $f(y)-f(y^+)\ge\varepsilon$.}%
};
\node[roadmap box,text width=3.25cm,minimum height=10.5mm,
      anchor=north west] (fairness) at ([xshift=2.5mm]decrease.north east) {%
    \roadmapentry{Conditional uniform draws $\to$
        Lemma~\ref{lem:fairness}}{%
        A fixed branch eligible at infinitely many primary (or auxiliary)
        draws $\Rightarrow$ it is selected infinitely many times.}%
};
\node[roadmap box,text width=3.65cm,minimum height=10.5mm,
      anchor=north west] (monotonicity)
      at ([xshift=2.5mm]fairness.north east) {%
    \roadmapentry{Lemma~\ref{lem:monotonicity}}{%
        \mbox{$f(x^{k+1})\le f(x^k)\quad(\forall k),$}\par
        $f(x^k)-f(x^{k+1})\to0$.}%
};
\node[fit=(decrease)(fairness)(monotonicity),inner sep=0pt] (toprow) {};

\node[roadmap box,text width=8.35cm,minimum height=7.5mm,
      below=2.5mm of toprow.south] (setup) {%
    Let $\bar x$ be any accumulation point and choose
    $x^{k_j}\to\bar x$; eventually
    $\cI_{00}(x^{k_j})\subseteq\cI_{00}(\bar x)$.\par
    Suppose, for contradiction, that $\bar x$ is not B-stationary.%
};
\draw[input arrow] (decrease.south) --
    ($(setup.north west)!0.20!(setup.north east)$);
\draw[input arrow] (fairness.south) --
    ($(setup.north west)!0.50!(setup.north east)$);
\draw[input arrow] (monotonicity.south) --
    ($(setup.north west)!0.80!(setup.north east)$);

\node[stable box,text width=3.15cm,minimum height=11.5mm,anchor=north]
      (stable) at ([xshift=-3.88cm,yshift=-2.5mm]setup.south) {%
    \roadmapentry{Lemma~\ref{lem:stable} $\mid$ stable pattern}{%
        $\cI_{00}(x^{k_j})=\cI_{00}(\bar x)$ along a subsequence\par
        $\Rightarrow J_{\rm dec}\in\cP(x^{k_j})$ is eligible at infinitely
        many primary draws.}%
};
\node[new box,text width=3.05cm,minimum height=11.5mm,
      anchor=north west] (newbiactive)
      at ([xshift=2.5mm]stable.north east) {%
    \roadmapentry{Lemma~\ref{lem:one-sided} $\mid$ new biactive indices}{%
        $\cI_{00}(x^{k_j})\subsetneq\cI_{00}(\bar x)$ eventually\par
        $\Rightarrow \cP(\bar x)\subseteq\cP(T_\tau(x^{k_j}))$ and
        $\cS_\tau(x^{k_j})\ne\varnothing$ eventually.}%
};
\node[new box,text width=3.85cm,minimum height=11.5mm,
      anchor=north west] (auxiliary)
      at ([xshift=2.5mm]newbiactive.north east) {%
    \roadmapentry{Lemma~\ref{lem:new-biactive} $\mid$ auxiliary draws}{%
        Algorithm~\ref{alg:threshold-tr} performs\par
        auxiliary draws at infinitely many iterations\par
        $\Rightarrow J_{\rm dec}$ is available for auxiliary sampling
        infinitely many times.}%
};
\node[fit=(stable)(newbiactive)(auxiliary),inner sep=0pt] (branches) {};
\draw[roadmap arrow]
    ($(setup.south west)!0.28!(setup.south east)$) -- (stable.north);
\draw[roadmap arrow]
    ($(setup.south west)!0.72!(setup.south east)$) -- (newbiactive.north);
\draw[roadmap arrow] (newbiactive.east) -- (auxiliary.west);

\node[roadmap box,text width=6.25cm,minimum height=10.5mm,anchor=north]
      (contradiction) at ([xshift=-2.24cm,yshift=-2.5mm]branches.south) {%
    \roadmapentry{Lemmas~\ref{lem:fairness}--\ref{lem:stable}
        $\;/\;$ Lemmas~\ref{lem:fairness}--\ref{lem:new-biactive}}{%
        In either case, $J_{\rm dec}$ is selected infinitely many times\par
        $\Rightarrow f(x^k)-f(x^{k+1})\ge\varepsilon>0$ for infinitely many
        $k$, contradicting Lemma~\ref{lem:monotonicity}:
        $f(x^k)-f(x^{k+1})\to0$.}%
};
\node[roadmap box,text width=4.05cm,minimum height=10.5mm,
      anchor=north west] (finite)
      at ([xshift=2.5mm]contradiction.north east) {%
    \roadmapentry{Proposition~\ref{prop:branch-lp-certificate}
        $\mid$ termination}{%
        $\cI_{00}(x^k)=\varnothing\Rightarrow\cP(x^k)=\{J^k\}$;\par
        primary LP: $p_{J^k}(x^k,\Delta_k)=0$
        $\Rightarrow x^k$ is B-stationary.}%
};
\node[fit=(contradiction)(finite),inner sep=0pt] (resultrow) {};
\draw[roadmap arrow] (stable.south) --
    ($(contradiction.north west)!0.24!(contradiction.north east)$);
\draw[roadmap arrow] (auxiliary.south) --
    ($(contradiction.north west)!0.82!(contradiction.north east)$);

\node[theorem box,text width=8.90cm,minimum height=8mm,
      below=2.2mm of resultrow.south] (theorem) {%
    \roadmapentry{Theorem~\ref{thm:threshold-bstationary}}{%
        finitely many iterations $\Rightarrow$ the returned point is
        B-stationary;\par
        infinitely many iterations $\Rightarrow$ every accumulation point is
        B-stationary.}%
};
\draw[roadmap arrow] (contradiction.south) --
    ($(theorem.north west)!0.30!(theorem.north east)$);
\draw[roadmap arrow] (finite.south) --
    ($(theorem.north west)!0.75!(theorem.north east)$);

\end{tikzpicture}
\endgroup
\caption{Proof roadmap for Theorem~\ref{thm:threshold-bstationary}}
\label{fig:threshold-proof-roadmap}
\end{figure}

First, for $J\in\cP$, $x\in\cOmega_J$, and $\alpha>0$, define the
projected-gradient residual on $\cOmega_J$ by
\[
    s_J(x,\alpha)
    :=
    \proj_{\cOmega_J}\bigl(x-\alpha\grad f(x)\bigr)-x.
\]
The projection is uniquely defined because $\cOmega_J$ is a closed convex
polyhedron.

Second, the linearized branch problem will also be evaluated at base points
that need not belong to $\cOmega_J$.  Whenever $\LP(y,J,\delta)$ is feasible,
denote its optimal predicted decrease by
\[
    \pi_J(y,\delta)
    :=
    \max\left\{
    -\grad f(y)\T d
    \ \middle|\
    y+d\in\cOmega_J,\ \|d\|_\infty\le\delta
    \right\}.
\]
For $y\in\cOmega_J$, this reduces to $p_J(y,\delta)$ defined in
\eqref{eq:branch-pred}.  The next two lemmas use these quantities to identify
a branch that yields uniform decrease near any non-B-stationary point.

\begin{lemma}\label{lem:nonB-proj}
Let \(\bar x\in\cOmega\) be non-B-stationary.  Then, for every
\(\alpha>0\), there exists \(J\in\cP(\bar x)\) such that
\[
    s_J(\bar x,\alpha)\ne0.
\]
\end{lemma}

\begin{proof}
By Proposition~\ref{prop:branch-lp-certificate}, there are
\(J\in\cP(\bar x)\) and \(\Delta>0\) with
\(p_J(\bar x,\Delta)>0\).  If \(s_J(\bar x,\alpha)=0\), the projection
variational inequality would give
\(\grad f(\bar x)\T(z-\bar x)\ge0\) for every \(z\in\cOmega_J\), which
implies \(p_J(\bar x,\Delta)=0\), a contradiction.
\end{proof}

\begin{lemma}\label{lem:cross-uniform}
Let \(\bar x\in\cOmega\), \(J\in\cP(\bar x)\), and \(\alpha>0\) satisfy
$s_J(\bar x,\alpha)\ne0$.  Then there exist a neighborhood \(U\) of
\(\bar x\) and a constant \(\varepsilon>0\) such that, for every
\(y\in U\) (neither \(y\in\cOmega\) nor \(J\in\cP(y)\) is required) and
every initial radius \(\delta_0\ge\Delta_{\min}\),
Subalgorithm~\ref{alg:lbtrs} (\LBTRS{}) applied to \((y,J,\delta_0)\)
returns $\mathsf{success}$ after finitely many inner iterations, with a point
\(y^+\in\cOmega_J\)
satisfying
\begin{equation}\label{eq:uniform-decrease}
    f(y^+)\le f(y)-\varepsilon.
\end{equation}
\end{lemma}

\begin{proof}
Set
\[
    s:=s_J(\bar x,\alpha),\qquad
    \kappa:=-\grad f(\bar x)\T s,\qquad
    M:=\|s\|_\infty>0,\qquad
    C_Q:=\tfrac n2\|Q\|_2.
\]
Let
\[
    u:=\bar x-\alpha\grad f(\bar x),
    \qquad
    p:=\proj_{\cOmega_J}(u)=\bar x+s.
\]
Since \(J\in\cP(\bar x)\), we have \(\bar x\in\cOmega_J\).  The projection
variational inequality is
\[
    (p-u)\T(z-p)\ge0
    \qquad \forall z\in\cOmega_J.
\]
Taking \(z=\bar x\) gives
\[
    0\le(p-u)\T(\bar x-p)
    =\bigl(s+\alpha\grad f(\bar x)\bigr)\T(-s)
    =-\|s\|^2-\alpha\grad f(\bar x)\T s.
\]
Consequently,
\[
    \kappa=-\grad f(\bar x)\T s
    \ge\alpha^{-1}\|s\|^2>0,
\]
where the strict inequality follows from \(s\ne0\).
Choose \(\bar\delta>0\) such that
\begin{equation}\label{eq:deltabar-choice}
    \bar\delta\le\min\{\Delta_{\min},2M\},
    \qquad
    C_Q\bar\delta
    \le(1-\eta_1)\frac{\kappa}{8M},
\end{equation}
and define
\[
    \delta_\star:=\gamma_{\rm dec}\bar\delta.
\]
By continuity of \(\grad f\), shrink a neighborhood \(U\) of \(\bar x\)
so that every \(y\in U\) satisfies
\begin{equation}\label{eq:nearby-bounds}
    -\grad f(y)\T s\ge\tfrac{\kappa}{2},\qquad
    \|y-\bar x\|_\infty\le\tfrac{\delta_\star}{2},\qquad
    \bigl|\grad f(y)\T(\bar x-y)\bigr|
    \le\frac{\kappa\delta_\star}{8M}.
\end{equation}

Fix \(y\in U\) and
\(\delta\in[\delta_\star,\bar\delta]\).  Define
\[
    t_\delta:=\frac{\delta}{2M},\qquad
    z_\delta:=\bar x+t_\delta s,\qquad
    d_\delta:=z_\delta-y.
\]
By \eqref{eq:deltabar-choice}, \(0<t_\delta\le1\).  Moreover,
\(\bar x\in\cOmega_J\) and
\(\bar x+s=\proj_{\cOmega_J}(u)\in\cOmega_J\), so convexity gives
\(z_\delta\in\cOmega_J\).  From \eqref{eq:nearby-bounds} and
\(\delta\ge\delta_\star\), and recalling that
\(\|s\|_\infty=M\),
\[
    \|d_\delta\|_\infty
    \le\|\bar x-y\|_\infty+t_\delta\|s\|_\infty
    \le\tfrac{\delta_\star}{2}+\tfrac\delta2
    \le\delta.
\]
Hence \(d_\delta\) is feasible for \(\LP(y,J,\delta)\), so this LP is
feasible.  Since \(\pi_J(y,\delta)\) is its optimal predicted decrease,
\(d_\delta=(\bar x-y)+t_\delta s\), and
\eqref{eq:nearby-bounds} holds, we obtain
\[
    \pi_J(y,\delta)
    \ge-\frac{\kappa\delta_\star}{8M}
      +\frac{\kappa\delta}{4M}
    \ge\frac{\kappa}{8M}\delta.
\]
Let \(c:=\kappa/(8M)>0\).  We have proved
\begin{equation}\label{eq:pred-linear}
    \pi_J(y,\delta)\ge c\delta
    \qquad
    \forall\delta\in[\delta_\star,\bar\delta].
\end{equation}
For every \(\delta\ge\bar\delta\), the step
\(d_{\bar\delta}\) remains feasible, and hence
\begin{equation}\label{eq:pred-large}
    \pi_J(y,\delta)\ge c\bar\delta>0.
\end{equation}

Consider Subalgorithm~\ref{alg:lbtrs} with
\(\delta_0\ge\Delta_{\min}\ge\bar\delta\).  At all trial radii
\(\delta_\ell\ge\bar\delta\), \eqref{eq:pred-large} shows that the LP is
feasible and has positive predicted decrease.  Consequently, the algorithm
cannot return failure before either accepting or reaching a radius not
exceeding \(\bar\delta\).

If it accepts at a radius \(\delta_\ell>\bar\delta\), then
\[
    \pred_\ell\ge c\bar\delta\ge c\delta_\star
\]
and the ratio test gives
\[
    f(y+d_\ell)\le f(y)-\eta_1c\delta_\star.
\]
Otherwise, let \(\ell\) be the first index with
\(\delta_\ell\le\bar\delta\).  The initialization and the geometric radius
update in line~17 of Subalgorithm~\ref{alg:lbtrs} imply
\(\delta_\star\le\delta_\ell\le\bar\delta\).
For any optimal LP step \(d_\ell\), \eqref{eq:pred-linear} gives
\(\pred_\ell\ge c\delta_\ell\), while
\(\left|\tfrac12d_\ell\T Qd_\ell\right|
\le C_Q\delta_\ell^2\).  Using
\(\delta_\ell\le\bar\delta\), \eqref{eq:deltabar-choice}, and
\(\pred_\ell\ge c\delta_\ell\), we obtain
\[
    C_Q\delta_\ell^2
    \le C_Q\bar\delta\,\delta_\ell
    \le(1-\eta_1)c\delta_\ell
    \le(1-\eta_1)\pred_\ell.
\]
Using the quadratic expansion of \(f\),
\(\ared_\ell=f(y)-f(y+d_\ell)
=-\grad f(y)\T d_\ell-\tfrac12d_\ell\T Qd_\ell
=\pred_\ell-\tfrac12d_\ell\T Qd_\ell\).  Therefore,
\[
    \ared_\ell
    \ge\pred_\ell-\left|\tfrac12d_\ell\T Qd_\ell\right|
    \ge\eta_1\pred_\ell>0.
\]
Thus \(\varrho_\ell=\ared_\ell/\pred_\ell\ge\eta_1\), so the trial is
accepted.  Using \(\delta_\star\le\delta_\ell\le\bar\delta\) and
\(\pred_\ell\ge c\delta_\ell\),
\[
    f(y+d_\ell)
    =f(y)-\ared_\ell
    \le f(y)-\eta_1\pred_\ell
    \le f(y)-\eta_1c\delta_\ell
    \le f(y)-\eta_1c\delta_\star.
\]
If no earlier trial is accepted, the update in line~17 reaches
\([\delta_\star,\bar\delta]\) after finitely many contractions.
Equations~\eqref{eq:pred-linear}--\eqref{eq:pred-large} exclude failure.
At the first radius in this interval,
\(\varrho_\ell\ge\eta_1\), so line~15 returns $\mathsf{success}$.
In either acceptance case,
\eqref{eq:uniform-decrease} holds with
\(\varepsilon:=\eta_1c\delta_\star>0\).
\end{proof}

\begin{corollary}\label{cor:primary-well-defined}
At the $k$-th iteration of Algorithm~\ref{alg:threshold-tr}, suppose that the
optimal predicted decrease of the primary LP is positive, namely,
$\pi_{J^k}(x^k,\Delta_k)>0$.  Then Subalgorithm~\ref{alg:lbtrs}, called with
input $(x^k,J^k,\Delta_k)$, returns $\mathsf{success}$ after finitely many
inner iterations and produces $x_P^k\in\cOmega_{J^k}$ satisfying
$f(x_P^k)<f(x^k)$.
\end{corollary}

\begin{proof}
The condition $\pi_{J^k}(x^k,\Delta_k)>0$ and the projection variational
inequality imply
\(s_{J^k}(x^k,\alpha)\ne0\) for every \(\alpha>0\).  Apply
Lemma~\ref{lem:cross-uniform} with
\(\bar x=y=x^k\), \(J=J^k\), and \(\delta_0=\Delta_k\).
\end{proof}

\begin{lemma}
\label{lem:lbtrs-finite}
Let \(y\in\R^n\), \(J\in\cP\), and
\(\delta_0\ge\Delta_{\min}\).  Applied to input
\((y,J,\delta_0)\), Subalgorithm~\ref{alg:lbtrs} successively solves the
linearized branch trust-region problems \(\LP(y,J,\delta_\ell)\) and
terminates after finitely many inner iterations with one of the outputs
$\mathsf{success}$, $\mathsf{zero}$, or $\mathsf{failure}$.
\end{lemma}

\begin{proof}
We distinguish three cases.  First, suppose that
\(\cOmega_J=\varnothing\).  The initial problem
\(\LP(y,J,\delta_0)\) is then infeasible.  Hence lines~3--4 of
Subalgorithm~\ref{alg:lbtrs} return $\mathsf{failure}$ at \(\ell=0\).

Suppose next that \(\cOmega_J\ne\varnothing\) but
\(y\notin\cOmega_J\).  Because \(\cOmega_J\) is closed,
\(r:=\dist_\infty(y,\cOmega_J)>0\).  Assume, to the contrary, that the call
does not terminate after finitely many inner iterations.  Every preceding
iteration then contracts the trust-region radius according to line~17, and hence
\[
    \delta_\ell=\gamma_{\rm dec}^{\,\ell}\delta_0\downarrow0,
    \qquad
    L:=\min\{\ell\in\mathbb N\mid\delta_\ell<r\}<\infty.
\]
At the \(L\)-th iteration, no step can satisfy both
\(y+d\in\cOmega_J\) and \(\|d\|_\infty\le\delta_L\).  Thus
\(\LP(y,J,\delta_L)\) is infeasible, and lines~3--4 return
$\mathsf{failure}$, a contradiction.

Finally, let \(y\in\cOmega_J\).  The zero step is feasible at every radius,
so \(\pi_J(y,\delta)\ge0\).  If \(\pi_J(y,\delta_0)=0\), lines~7--8 return
$\mathsf{zero}$ at \(\ell=0\).  Otherwise,
\(\pi_J(y,\delta_0)>0\) implies \(s_J(y,1)\ne0\) by the projection
optimality condition.  Lemma~\ref{lem:cross-uniform}, applied with
\(\bar x=y\) and \(\alpha=1\), therefore guarantees that line~15 returns
$\mathsf{success}$ after finitely many inner iterations.  The three cases
exhaust all inputs.
\end{proof}

\begin{lemma}
\label{lem:monotonicity}
Under Assumption~\ref{ass:compact}, each realization of
Algorithm~\ref{alg:threshold-tr} (\TBTA{}) generates feasible iterates in
$\cL(x^0)$ and satisfies $f(x^{k+1})\le f(x^k)$ at every completed
iteration, with strict inequality whenever $x^{k+1}\ne x^k$.  If the iterate
sequence is infinite, then it has accumulation points, $\{f(x^k)\}$
converges, and $f(x^k)-f(x^{k+1})\to0$.
\end{lemma}

\begin{proof}
Fix a realization.  The input $x^0$ is feasible, and suppose inductively that
$x^k\in\cOmega$.  Since line~3 chooses $J^k\in\cP(x^k)$, the primary call in
line~4 returns either $\mathsf{zero}$ or $\mathsf{success}$ by
Corollary~\ref{cor:primary-well-defined}.  In the first case, lines~5--10
either return $x^k$ at line~7 or set $x^{k+1}=x^k$ at line~9.  In the second
case, line~12 records a feasible point $x_P^k$ with $f(x_P^k)<f(x^k)$.
Line~13 initializes $x_A^k=x^k$, and lines~17--18 replace it only by a
feasible point that also lowers $f$.  Lines~21--25 select the candidate with
the smaller objective value.  Thus every update satisfies
$x^{k+1}\in\cOmega$ and $f(x^{k+1})\le f(x^k)$, with strict inequality if
$x^{k+1}\ne x^k$.

Induction proves these properties for every generated iterate.  Hence all
iterates remain in $\cL(x^0)$.  If the sequence is infinite, compactness of
this set gives accumulation points.  Since $f$ is bounded below on
$\cL(x^0)$, its nonincreasing values converge, and consequently
$f(x^k)-f(x^{k+1})\to0$.
\end{proof}

For each \(J\in\cP\) and \(r\in\mathbb N_+\), let
\(\mathcal U_{J,r}\) be the union of all open sets \(U\subseteq\R^n\) such
that every \LBTRS{} call with input \((y,J,\delta_0)\), where \(y\in U\)
and \(\delta_0\ge\Delta_{\min}\), returns \(\mathsf{success}\) with a point
\(y^+\in\cOmega_J\) satisfying \(f(y^+)\le f(y)-r^{-1}\).  By construction,
\(\mathcal U_{J,r}\) is open and has this same uniform-decrease property.

Now let \(\bar x\in\cOmega\) be non-B-stationary.  Lemmas~
\ref{lem:nonB-proj} and~\ref{lem:cross-uniform} provide a label
\(J\in\cP(\bar x)\), an open neighborhood \(U\) of \(\bar x\), and
\(\varepsilon>0\) such that every \LBTRS{} call with input
\((y,J,\delta_0)\), where \(y\in U\) and
\(\delta_0\ge\Delta_{\min}\), decreases \(f\) by at least
\(\varepsilon\).  Choose \(r\in\mathbb N_+\) with
\(r^{-1}\le\varepsilon\).  Then \(U\) is one of the open sets in the
defining union for \(\mathcal U_{J,r}\), and hence
\begin{equation}\label{eq:descent-cover}
    \bar x\ \text{not B-stationary}
    \quad\Longrightarrow\quad
    \exists J\in\cP(\bar x),\ r\in\mathbb N_{+}:
    \quad \bar x\in\mathcal U_{J,r}.
\end{equation}
We next record the iterations at which a fixed pair $(J,r)$ is eligible for
primary or auxiliary sampling.
Let $\mathcal D_A$ be the set of iterations at which the auxiliary draw in
Algorithm~\ref{alg:threshold-tr} is performed.
For every \(J\in\cP\) and \(r\in\mathbb N_+\), define
\[
\begin{aligned}
    \mathcal K_{J,r}
    &:=
    \{k\mid x^k\in\mathcal U_{J,r},\ J\in\cP(x^k)\},\\
    \widehat{\mathcal K}_{J,r}
    &:=
    \{k\in\mathcal D_A\mid x^k\in\mathcal U_{J,r},\
             J\in\cP(T_\tau(x^k))\}.
\end{aligned}
\]
\begin{lemma}\label{lem:fairness}
Let $\{x^k\}$, $\{J^k\}$, and $\{\widehat J^k\}$ be, respectively, the
iterate, primary-label, and auxiliary-label sequences generated by
Algorithm~\ref{alg:threshold-tr} (\TBTA{}).  With probability one, the
following implications hold simultaneously for every $J\in\cP$ and
$r\in\mathbb N_+$:
\[
\begin{aligned}
|\mathcal K_{J,r}|=\infty
&\Longrightarrow
J^k=J\text{ for infinitely many }k\in\mathcal K_{J,r},\\
|\widehat{\mathcal K}_{J,r}|=\infty
&\Longrightarrow
\widehat J^k=J\text{ for infinitely many }
k\in\widehat{\mathcal K}_{J,r}.
\end{aligned}
\]
Equivalently, denote this event by $\mathcal E_{\rm fair}$; then
$\mathbb P(\mathcal E_{\rm fair})=1$.
\end{lemma}

\begin{proof}
Fix $(J,r)$, and define
\[
\begin{aligned}
A_k^P&:=\{x^k\in\mathcal U_{J,r},\ J\in\cP(x^k)\},\\
A_k^A&:=\{k\in\mathcal D_A,\ x^k\in\mathcal U_{J,r},\
J\in\cP(T_\tau(x^k))\}.
\end{aligned}
\]
These events are $\mathcal F_k^P$- and $\mathcal F_k^A$-measurable,
respectively, and $A_k^A$ can occur only when the auxiliary draw is performed.
By \eqref{eq:tbta-primary-uniform}--\eqref{eq:tbta-auxiliary-uniform},
\[
\begin{aligned}
\mathbb P(A_k^P\cap\{J^k=J\}\mid\mathcal F_k^P)
&\ge2^{-m}\mathbf1_{A_k^P},\\
\mathbb P(A_k^A\cap\{\widehat J^k=J\}\mid\mathcal F_k^A)
&\ge2^{-m}\mathbf1_{A_k^A}.
\end{aligned}
\]
Thus the corresponding conditional-probability sum diverges whenever the
associated index set is infinite.  Applying the conditional Borel--Cantelli
argument used in Lemma~\ref{lem:lpcc-finite-sampling} separately with
$\{\mathcal F_k^P\}$ and $\{\mathcal F_k^A\}$ shows that
$A_k^P\cap\{J^k=J\}$ occurs infinitely often whenever $A_k^P$ does, and
likewise for $A_k^A\cap\{\widehat J^k=J\}$.  By definition, the occurrences
of $A_k^P$ and $A_k^A$ are exactly the indices in $\mathcal K_{J,r}$ and
$\widehat{\mathcal K}_{J,r}$, respectively, so these conclusions are precisely
the two stated implications.  Since $\cP$ is finite and $\mathbb N_+$ is countable, the
implications hold simultaneously for all $(J,r)$, which proves
$\mathbb P(\mathcal E_{\rm fair})=1$.
\end{proof}

On $\mathcal E_{\rm fair}$, the preceding implications, the definition of
$\mathcal U_{J,r}$, and the update in
Algorithm~\ref{alg:threshold-tr} give
\begin{equation}\label{eq:fair-use}
\begin{aligned}
|\mathcal K_{J,r}|=\infty
&\Longrightarrow
f(x^k)-f(x^{k+1})\ge r^{-1}
\quad\text{for infinitely many }k\in\mathcal K_{J,r},\\
|\widehat{\mathcal K}_{J,r}|=\infty
&\Longrightarrow
f(x^k)-f(x^{k+1})\ge r^{-1}
\quad\text{for infinitely many }k\in\widehat{\mathcal K}_{J,r}.
\end{aligned}
\end{equation}
It remains to show that an improving branch at a limit point belongs to one
of these eligibility sets infinitely often.  We first treat the stable-pattern
case.

\begin{lemma}\label{lem:stable}
Suppose Assumption~\ref{ass:compact} holds.  On every realization in
$\mathcal E_{\rm fair}$, let $\{x^k\}$ be an infinite iterate sequence
generated by Algorithm~\ref{alg:threshold-tr} (\TBTA{}), and let
\(x^{k_j}\to\bar x\) be any subsequence satisfying
\[
    \cI_{00}(x^{k_j})=\cI_{00}(\bar x)
    \qquad \forall j.
\]
Then \(\bar x\) is B-stationary.
\end{lemma}

\begin{proof}
Fix a realization in $\mathcal E_{\rm fair}$ and suppose that $\bar x$ is
not B-stationary.  Choose $(J,r)$ from
\eqref{eq:descent-cover}, i.e., $J\in\cP(\bar x)$,
$r\in\mathbb N_+$, and $\bar x\in\mathcal U_{J,r}$.  Openness of
$\mathcal U_{J,r}$ gives
$x^{k_j}\in\mathcal U_{J,r}$ for all sufficiently large $j$.  Moreover,
convergence and $\cI_{00}(x^{k_j})=\cI_{00}(\bar x)$ give
$J\in\cP(x^{k_j})$ eventually, and hence $|\mathcal K_{J,r}|=\infty$.
By the first implication in \eqref{eq:fair-use},
$f(x^k)-f(x^{k+1})\ge r^{-1}$ at infinitely many iterations.  On the other
hand, Lemma~\ref{lem:monotonicity} gives convergence of $\{f(x^k)\}$, so
$f(x^k)-f(x^{k+1})\to0$, a contradiction.  Hence $\bar x$ is
B-stationary.
\end{proof}

The preceding lemma handles the stable-pattern case through primary sampling.
For the remaining case, the next deterministic result shows how thresholding
exposes every branch compatible with the limit.

\begin{lemma}\label{lem:one-sided}
Let \(\tau>0\), and let \(\{x^j\}\subseteq\cOmega\) converge to
\(\bar x\in\cOmega\).  Then, for all sufficiently large \(j\),
\begin{equation}\label{eq:branch-inclusion}
    \cP(\bar x)\subseteq\cP(T_\tau(x^j)).
\end{equation}
If additionally
\[
    \cI_{00}(x^j)\subsetneq\cI_{00}(\bar x)
    \qquad\text{for every }j,
\]
then, for all sufficiently large \(j\),
\begin{equation}\label{eq:threshold-nonempty}
    \cS_\tau(x^j)\ne\varnothing.
\end{equation}
\end{lemma}

\begin{proof}
Since $[m]$ is finite and $x^j\to\bar x$, there is $j_0$ such that, for
every $j\ge j_0$,
\[
    \|(x^j_{1,i},x^j_{2,i})\|\le\tau
    \qquad\forall i\in\cI_{00}(\bar x).
\]
Increasing $j_0$ if necessary, the positive component of every pair in
$\cI_{+0}(\bar x)\cup\cI_{0+}(\bar x)$ also remains positive whenever
$j\ge j_0$.
Fix such a $j$.  If $i\in\cI_{00}(\bar x)$, then either
$i\in\cI_{00}(x^j)$ or $i\in\cS_\tau(x^j)$; in both cases,
$i\in\cI_{00}(T_\tau(x^j))$.  If $i\in\cI_{+0}(\bar x)$,
complementarity gives $i\in\cI_{+0}(x^j)$; thresholding either preserves
this pattern or makes the pair biactive.  The same argument applies to
$\cI_{0+}(\bar x)$ with the two components interchanged.  Consequently,
\[
    \cI_{+0}(T_\tau(x^j))\subseteq\cI_{+0}(\bar x),
    \qquad
    \cI_{0+}(T_\tau(x^j))\subseteq\cI_{0+}(\bar x).
\]
For $J=(D_1,D_2)\in\cP(\bar x)$, these inclusions give
$\cI_{0+}(T_\tau(x^j))\subseteq D_1$ and
$\cI_{+0}(T_\tau(x^j))\subseteq D_2$.  Thus
$J\in\cP(T_\tau(x^j))$, proving \eqref{eq:branch-inclusion}.

Under the additional hypothesis, choose
$i_j\in\cI_{00}(\bar x)\setminus\cI_{00}(x^j)$.  For $j\ge j_0$,
$i_j\notin\cI_{00}(x^j)$ and
$\|(x^j_{1,i_j},x^j_{2,i_j})\|\le\tau$; hence
$i_j\in\cS_\tau(x^j)$, which proves
\eqref{eq:threshold-nonempty}.
\end{proof}

This inclusion allows auxiliary sampling to recover branches that become
compatible only at the limit.

\begin{lemma}\label{lem:new-biactive}
Suppose Assumption~\ref{ass:compact} holds.  On every realization in
$\mathcal E_{\rm fair}$, let $\{x^k\}$ be an infinite iterate sequence
generated by Algorithm~\ref{alg:threshold-tr} (\TBTA{}), and let
\(x^{k_j}\to\bar x\) be any subsequence satisfying
\[
    \cI_{00}(x^{k_j})\subsetneq\cI_{00}(\bar x)
    \qquad \forall j.
\]
Then \(\bar x\) is B-stationary.
\end{lemma}

\begin{proof}
Fix a realization in $\mathcal E_{\rm fair}$ and suppose that $\bar x$ is
not B-stationary.  Choose $(J,r)$ from
\eqref{eq:descent-cover}, i.e., $J\in\cP(\bar x)$,
$r\in\mathbb N_+$, and $\bar x\in\mathcal U_{J,r}$.  Openness gives
$x^{k_j}\in\mathcal U_{J,r}$ eventually.  A convergent sequence taking only
finitely many distinct values is eventually constant, which would contradict
$\cI_{00}(x^{k_j})\subsetneq\cI_{00}(\bar x)$.  We may therefore pass to a
subsequence for which the points $x^{k_j}$ are pairwise distinct.

The indices $k_j$ need not themselves be iterations at which the iterate
changes.  For each $j$, let $\ell_j$ be the last index of the consecutive
constant block containing $k_j$.  Since $x^{k_{j+1}}\ne x^{k_j}$, this block
ends before $k_{j+1}$, and hence
\[
    k_j\le\ell_j<k_{j+1}\le\ell_{j+1},
    \qquad
    x^{\ell_j}=x^{k_j},
    \qquad
    x^{\ell_j+1}\ne x^{\ell_j}.
\]
Thus every $\ell_j$ is finite and the sequence $\{\ell_j\}$ is strictly
increasing.  At the $\ell_j$-th iteration, a $\mathsf{zero}$ primary return would
either terminate the run at line~7 or set
$x^{\ell_j+1}=x^{\ell_j}$ at line~9, contradicting the infinite run and the
block exit $x^{\ell_j+1}\ne x^{\ell_j}$.  Moreover, line~3 gives
$J^{\ell_j}\in\cP(x^{\ell_j})$, so the zero step is feasible for the initial
primary LP and
$\pi_{J^{\ell_j}}(x^{\ell_j},\Delta_{\ell_j})\ge0$.  Equality would make the
primary call return $\mathsf{zero}$; hence this value is positive.
Corollary~\ref{cor:primary-well-defined} therefore shows that the primary
search returns $\mathsf{success}$, after which lines~13--14 evaluate the
auxiliary-search condition $\cS_\tau^{\ell_j}\ne\varnothing$.

Since $x^{\ell_j}=x^{k_j}\to\bar x$, Lemma~\ref{lem:one-sided} gives, for
all sufficiently large $j$,
\[
    J\in\cP(T_\tau(x^{\ell_j})),
    \qquad
    \cS_\tau(x^{\ell_j})\ne\varnothing.
\]
Thus the condition in line~14 holds for every sufficiently large $j$, and
the auxiliary label is drawn in line~15 at each such iteration $\ell_j$.
Since $x^{\ell_j}\in\mathcal U_{J,r}$ as well, we have
$\ell_j\in\widehat{\mathcal K}_{J,r}$ eventually and
$|\widehat{\mathcal K}_{J,r}|=\infty$.  The second implication in
\eqref{eq:fair-use} gives
$f(x^k)-f(x^{k+1})\ge r^{-1}$ at infinitely many iterations.  This
contradicts $f(x^k)-f(x^{k+1})\to0$, which follows from
Lemma~\ref{lem:monotonicity}.  Hence $\bar x$ is B-stationary.
\end{proof}

\begin{proof}[Theorem~\ref{thm:threshold-bstationary}]
If \TBTA{} terminates, the primary call has returned $\mathsf{zero}$ and the
stopping test gives, for the initial optimal step $d_0^k$,
\[
    \cI_{00}(x^k)=\varnothing,\qquad
    \cP(x^k)=\{J^k\},\qquad
    -\grad f(x^k)\T d_0^k=0.
\]
Hence Proposition~\ref{prop:branch-lp-certificate} proves (i).

For an infinite run, Lemma~\ref{lem:monotonicity} and compactness of
$\cL(x^0)$ give accumulation points.  Recall the probability-one event
$\mathcal E_{\rm fair}$ from Lemma~\ref{lem:fairness}, fix a realization in
this event, and let \(\bar x\) be any accumulation point with
\(x^{k_j}\to\bar x\).  If $i\in\cI_{+0}(\bar x)$, then
$x^{k_j}_{1,i}>0$ for all sufficiently large $j$; since
$x^{k_j}\in\cOmega$, this gives $i\in\cI_{+0}(x^{k_j})$.  The analogous
argument applies to $\cI_{0+}(\bar x)$.  Since $[m]$ is finite, after
discarding finitely many terms,
$\cI_{+0}(\bar x)\subseteq\cI_{+0}(x^{k_j})$ and
$\cI_{0+}(\bar x)\subseteq\cI_{0+}(x^{k_j})$.  Consequently,
$\cI_{00}(x^{k_j})=[m]\setminus
(\cI_{+0}(x^{k_j})\cup\cI_{0+}(x^{k_j}))
\subseteq[m]\setminus(\cI_{+0}(\bar x)\cup\cI_{0+}(\bar x))
=\cI_{00}(\bar x)$.
If \(\cI_{00}(x^{k_j})=\cI_{00}(\bar x)\) for infinitely many \(j\), pass
to that subsequence and apply Lemma~\ref{lem:stable}.  Otherwise
\(\cI_{00}(x^{k_j})\subsetneq\cI_{00}(\bar x)\) holds for all sufficiently
large \(j\), and Lemma~\ref{lem:new-biactive} applies after discarding
finitely many terms.  Since \(\bar x\) was arbitrary, (ii) follows.
\end{proof}

\section{Numerical Experiments}
\label{sec:numerics}
\label{sec:numerical-experiments}

We evaluate the proposed algorithms in three settings.  First, we test the
LPCC methods on LPCCs derived from bilevel models, affine-GNE systems, and
inverse convex QPs.  Second, we test the QPCC method on convex and nonconvex
QPCCs derived from the first two problem families.  Third, we embed the
proposed LPCC solvers into the
MPECopt framework for computing B-stationary points of general
MPCCs~\cite{nurkanovic2025mpecopt}, replacing its Phase-II LPCC subproblem
solver.  On the MacMPEC collection~\cite{MacMPEC}, we compare the resulting
variants with the original MPECopt implementation, in which Gurobi solves each
Phase-II LPCC subproblem through a big-\(M\) MILP reformulation.

\begingroup
\sloppy
The MacMPEC benchmark instances used in this study are available from
\url{https://wiki.mcs.anl.gov/leyffer/index.php/MacMPEC}. All other data and
problem instances used in the numerical experiments can be generated using
the scripts provided in the reproducibility package.
The source code and reviewer-oriented reproducibility repository are available
at \url{https://github.com/SUSTech-Optimization/Reproducibility-for-Randomized-Branch-Inexact-Methods-for-Bouligand-Stationarity-in-LPCC-and-QPCC-/}.
The complete reproducibility package is archived on Zenodo at
\url{https://zenodo.org/records/22133427}.
\par
\endgroup

\subsection{Experimental setup, implementations, and metrics}
\label{subsec:numerical-setup}

\subsubsection{Computing environment and software}
\label{subsubsec:hardware-software}

All experiments were run on a 10-core Apple M4 machine with 16 GB of memory
under macOS 15.6.  The main implementation used Python 3.13.13 with NumPy
2.4.6, SciPy 1.17.1, Pyomo 6.10.1, \texttt{gurobipy} 13.0.2, and HiGHS
1.14.0.  Gurobi was used under a noncommercial license.

\subsubsection{Implementations and comparison methods}
\label{subsubsec:methods-baselines}

\runinhead{\textit{Research implementation for the primal-simplex steps in
			Algorithm~\ref{alg:lpcc-rpsa}}.}
Algorithm~\ref{alg:lpcc-rpsa} requires primal-simplex steps rather than
black-box calls that solve each sampled LP to optimality.  We therefore use
our own implementation of the primal simplex method.  It employs standard
simplex techniques, including Dantzig pricing, Bland's anti-cycling rule,
warm starts, and reusable basis snapshots.  This implementation
is research code written by the authors rather than a highly optimized
general-purpose LP solver and is not expected to match mature commercial or
open-source simplex implementations in raw LP solution speed.  Its role here
is to provide the pivot-level control required by
Algorithm~\ref{alg:lpcc-rpsa}.  Appendix~\ref{app:inhouse-simplex} gives the
pseudocode and implementation details.  The implementation of the in-house simplex
solver is publicly available in the
\href{https://github.com/SUSTech-Optimization/Reproducibility-for-Randomized-Branch-Inexact-Methods-for-Bouligand-Stationarity-in-LPCC-and-QPCC-/tree/main/src/bounded_variable_simplex}{corresponding GitHub repository}.
For brevity, this implementation is denoted by
\emph{A-Simplex} in the tables and figures.

\runinhead{\textit{LPCC methods and baselines}.}
The proposed LPCC methods are \RBLA{} in Algorithm~\ref{alg:lpcc-rbla} and
\RPSA{} in Algorithm~\ref{alg:lpcc-rpsa}.  For each sampled compatible label,
\RBLA{} solves the corresponding branch LP to optimality using either Gurobi
or A-Simplex.  \RPSA{} uses A-Simplex to perform the primal-simplex steps.  Its
initialization LPs are solved with HiGHS and Gurobi to provide a warm start.

The LPCC baselines are the big-\(M\) MILP, IPOPT--Scholtes
\cite{scholtes2001regularization}, LPCC branch-and-cut
\cite{YuMitchellPang2019LPCCBC}, PIP~\cite{ZhangHanPang2026PIP},
CCOpt~\cite{pozharskiy2026ccopt}, and PL-E*
\cite{JaraMoroniPangWachter2018LPCC}.  Gurobi solves the big-\(M\) MILP, the
LPs in branch-and-cut and PL-E*, and the restricted MILPs in PIP.
IPOPT--Scholtes uses IPOPT 3.14.19 through Pyomo, while CCOpt is run with
Julia 1.11.7 and CCOpt 0.1.0 using its package-provided relaxation solver.

\runinhead{\textit{QPCC methods and baselines}.}
The proposed QPCC method is \TBTA{} in
Algorithm~\ref{alg:threshold-tr}.  All primary and auxiliary branch
trust-region LPs in \TBTA{} are solved with Gurobi.

The QPCC baselines are the big-\(M\) MIQP and IPOPT--Scholtes.  Depending on the
experiment, we additionally test MPECopt and its early-stopped variant
MPECopt-early~\cite{nurkanovic2025mpecopt}, CCOpt, LCQPow-D and
LCQPow-OSQP~\cite{hall2024lcqpow}, and
DCA3~\cite{LeThiNguyenPhamDinh2023DCLCC}.
The big-\(M\) baseline is a single MIQP solved by Gurobi, while IPOPT--Scholtes
solves each relaxed NLP with IPOPT.  MPECopt and MPECopt-early solve their LPCC
subproblems through big-\(M\) MILPs with Gurobi and their branch NLPs with IPOPT.
CCOpt uses its package-provided relaxation solver.  The LCQPow variants are
called through local C++ wrappers: LCQPow-D uses the dense qpOASES backend,
whereas LCQPow-OSQP uses the sparse OSQP backend.  Our DCA3 implementation
solves its convex QP subproblems with Gurobi.  We also compare a
straightforward extension of Algorithm~\ref{alg:lpcc-rbla} to QPCCs, obtained
by replacing each branch LP solve with a branch QP solve.  We call this method
the branch QP algorithm (bQPA).  In this comparison, IPOPT seeks a local KKT
point of each sampled branch QP, whereas Gurobi seeks a global minimizer.

\runinhead{\textit{MPECopt framework}.}
MPECopt is a two-phase framework for computing B-stationary points of general
MPCCs~\cite{nurkanovic2025mpecopt}.  Phase I combines regularized NLPs with a
linearized LPCC subproblem to identify a feasible branch NLP (BNLP), while
Phase II alternates between LPCC subproblems and smooth BNLPs.  The BNLP solver
returns a stationary point of the selected branch; the LPCC solver either
provides a nonzero feasible descent point and a new active-set estimate, or
certifies that the zero step is optimal and hence that the current point is
B-stationary.

In the original implementation, Gurobi solves big-\(M\) MILP reformulations of
the LPCC subproblems, while IPOPT solves the regularized NLPs and BNLPs.  In
the MacMPEC embedding test, only the Phase-II Gurobi LPCC solver is replaced
by \RBLA{} or \RPSA{}; the regularization-based Phase I, the IPOPT BNLP solver,
and the outer MPECopt logic remain unchanged.

\subsubsection{Stopping criteria and parameters}
\label{subsubsec:parameters-stopping}

Because of finite-precision arithmetic, numerical stopping is based on
changes between successive iterates.  Write $f^k:=f(x^k)$ and consider
\[
\|x^k-x^{k+1}\|_\infty\le\varepsilon_x,
\qquad
\frac{|f^k-f^{k+1}|}{1+|f^k|}\le\varepsilon_f.
\]
Since a finite run cannot establish that all subsequent iterates remain
unchanged, we regard $p$ consecutive small steps as numerical stabilization;
any non-small step resets the count.  The definition of a small step and the
value of $p$ are method-specific.

\runinhead{\textit{\RBLA{}}.}
A step is small if either test above holds.  We use
$\varepsilon_x=\varepsilon_f=10^{-5}$ and $p=3$ in all experiments.

\runinhead{\textit{\RPSA{}}.}
We set $\rho=20$ for the MacMPEC embedding, $\rho=100$ for the bilevel
experiments, $\rho=5000$ for the affine-GNE experiments, and $\rho=1000$ for
the inverse-QP experiments.
A step is small only if both tests
above hold.  We use $\varepsilon_x=\varepsilon_f=10^{-5}$ and $p=3$ in all
experiments.

\runinhead{\textit{\TBTA{}}.}
We use $\Delta_0=5$, $\Delta_{\min}=10^{-4}$, $\eta_1=0.25$,
$\eta_2=0.6$, $\gamma_{\rm dec}=0.1$, and $\gamma_{\rm inc}=2$.  The threshold
is $\tau=10^{-4}$ for the bilevel instances and $\tau=10^{-5}$ for the
affine-GNE instances.  For the consecutive-small-step criterion, a step is
small if both
\[
\|x^k-x^{k+1}\|_\infty
\le\varepsilon_x(1+\|x^k\|_\infty),
\qquad
\frac{|f^k-f^{k+1}|}{1+|f^k|}\le\varepsilon_f,
\]
where $(\varepsilon_x,\varepsilon_f,p)=(10^{-4},10^{-4},3)$.

After solving the initial primary LP in
Algorithm~\ref{alg:threshold-tr}, \TBTA{} also tests whether the current point
has no biactive complementarity pair and the optimal predicted decrease is
numerically zero.  Let
\[
\pred_0^k:=p_{J^k}(x^k,\Delta_k)
=-\nabla f(x^k)\T d_0^k,
\]
be the optimal predicted decrease of that LP.  The implementation regards
$\pred_0^k<10^{-6}$ as zero and returns if
\[
\cI_{00}(x^k)=\varnothing,
\qquad
\pred_0^k<10^{-6}.
\]
If $\pred_0^k<10^{-6}$ but $\cI_{00}(x^k)\ne\varnothing$, \TBTA{} sets
$x^{k+1}=x^k$ and continues.  The unchanged iteration counts as one small
step, so numerical termination still requires $p$ consecutive small steps.

\runinhead{\textit{Baseline parameters}.}
The big-\(M\) MILP and MIQP baselines are solved by Gurobi under the experiment
time limits.  IPOPT--Scholtes uses a homotopy sequence from $\tau=1$ to
$\tau=10^{-8}$ and IPOPT tolerance $10^{-7}$.  MPECopt and MPECopt-early use
$\rho_0=10^{-3}$, shrink factor $0.1$, $\texttt{tol\_b}=10^{-6}$, and six inner
Phase-II iterations.  CCOpt uses its relaxation solver with package-default
numerical options and the experiment time limit.  LCQPow-D is used in the
bilevel convex QPCC experiment and LCQPow-OSQP in the affine-GNE convex QPCC
experiment; both use the package defaults and experiment time limits.  PIP
uses $p_{\max}=0.6$, $p_{\rm initial}=0.8$, $\alpha=0.1$, $r_{\max}=3$, and the
experiment time limit for each restricted MILP subproblem.  PL-E* starts with
$\rho_0=1$ and updates $\rho\leftarrow10\rho$.  DCA3 starts with $t_1=10$ and,
when an update is triggered, uses $t_{k+1}=10t_k$.

\subsubsection{Metrics for the generated test families}
\label{subsubsec:metrics}

We use the following metrics to compare the algorithms in the bilevel-induced,
inverse-QP-induced, and affine-GNE experiments.  An entry $r/10$ means that
the stated condition
holds in $r$ of the ten runs summarized in that row.  Whether a run reaches
its stopping criterion affects only \textit{Term.}; all other metrics are
reported independently of termination status.

\runinhead{\textbf{Term.}}
This column counts runs in which the method reaches its prescribed numerical
stopping criterion within the time limit.  Thus, $1/10$ means that the
criterion is met in one of the ten runs.

\runinhead{\textbf{Feas.}}
This column counts runs whose returned output is numerically feasible.  An
output is declared feasible if its affine constraint violation is at most
$10^{-6}$ and its complementarity violation is at most $10^{-5}$.  Thus,
$5/10$ means that five of the ten runs return feasible outputs.

\runinhead{\textbf{Max aff. vio.}}
This column reports the largest affine constraint violation over all returned
outputs in the ten runs.

\runinhead{\textbf{Max comp. vio.}}
For a returned output $x=(x_0,x_1,x_2)$, we define its complementarity
violation by
\[
v_{\mathrm{comp}}(x):=
\max_{i\in[m]}\min\bigl\{|x_{1,i}|,|x_{2,i}|\bigr\}.
\]
This column reports the largest value of $v_{\mathrm{comp}}(x)$ over all
returned outputs in the ten runs.

\runinhead{\textbf{Mean objective-improvement gap.}}
For each instance, let $\mathcal F$ contain the methods that return a feasible
output with a finite objective value, and define
\[
\Delta_M:=f(x^0)-f(x_M),
\qquad
\Delta_{\max}:=
\max\bigl(\{0\}\cup\{\Delta_M:M\in\mathcal F\}\bigr).
\]
This column averages
$[\Delta_{\max}-\Delta_M]_+/\max\{1,|\Delta_{\max}|\}$ over the feasible
outputs of method $M$.

\runinhead{\textbf{Near-best observed.}}
This column counts feasible outputs satisfying
$\Delta_M\ge0.95\,\Delta_{\max}$.  The comparison is relative to the best
objective improvement observed on the same instance and does not provide a
global quality guarantee.  Thus, $5/10$ means that five of the ten runs meet
this criterion.

\runinhead{\textbf{Sampled branch-descent diagnostic.}}
For each numerically feasible output $x$, we construct $\cP(x)$ from its
numerically identified complementarity pattern and set
\[
s(x):=\min\bigl\{|\cP(x)|,16,
\max\{8,\lceil0.05|\cP(x)|\rceil\}\bigr\}.
\]
This rule inspects all labels when $|\cP(x)|\le8$; otherwise, the $5\%$
sampling target is clipped to between $8$ and $16$ labels.  We sample $s(x)$
labels uniformly without replacement from $\cP(x)$ and denote the sample by
$\mathcal S(x)$.  Sampling is over labels, so distinct labels are retained
even when they define the same polyhedron.  For each $J\in\mathcal S(x)$,
compute
\[
\theta_J(x):=
\min_d\ \{\nabla f(x)\T d
\mid x+d\in\cOmega_J,\ \|d\|_\infty\le1\},
\]
and define
\[
D_{\rm sbd}(x):=
\max_{J\in\mathcal S(x)}
\frac{[-\theta_J(x)]_+}
{\max\{1,\|\nabla f(x)\|_1\}}.
\]
Thus, $D_{\rm sbd}(x)$ is the largest normalized first-order descent detected
among the sampled compatible labels.  For an exactly feasible $x$ and
$J\in\cP(x)$, $-\theta_J(x)=p_J(x,1)$.  Hence, if all compatible labels are
inspected, the exact condition $D_{\rm sbd}(x)=0$ gives the branchwise
certificate in Proposition~\ref{prop:branch-lp-certificate}.  Because the
reported test may inspect only a subset of the labels and uses the positive
tolerance below, it is an a posteriori diagnostic rather than a
B-stationarity certificate.

\runinhead{\textbf{Mean sampled branch descent.}}
This column reports the arithmetic mean of $D_{\rm sbd}(x)$ over the method's
feasible outputs.

\runinhead{\textbf{Pass sampled branch descent.}}
This column counts feasible outputs satisfying $D_{\rm sbd}(x)\le10^{-3}$.

\runinhead{\textbf{Mean time.}}
This column reports the average computation time, in seconds, over the ten
runs.  The parenthesized value is the standard error.

\subsection{LPCC experiments}
\label{subsec:lpcc-experiments}

The LPCC experiments comprise three problem classes.  The first consists of
bilevel-induced LPCCs obtained by replacing the lower-level problems with
their Karush--Kuhn--Tucker (KKT) systems; these instances are constructed so
that the linear independence constraint qualification (LICQ) fails at the
initial point.  The second consists of sparse affine generalized Nash
equilibrium (affine-GNE) network complementarity systems with linear selection
objectives.  The third consists of LPCC reformulations of inverse convex
quadratic programs generated according to the setting in
\cite{JaraMoroniPangWachter2018LPCC}.

In the tables in this subsection, ``A-Simplex'' denotes the simplex
implementation written by the authors.  A parenthesized name identifies the
LP or MILP solver implementation used by a method.  The abbreviations \textit{Term.},
\textit{Feas.},
\textit{aff. vio.}, \textit{comp. vio.}, \textit{obj. improv.}, and
\textit{samp. br. desc.} stand for termination, feasibility, affine violation,
complementarity violation, objective improvement, and sampled branch descent,
respectively.  The corresponding metrics are defined in
Section~\ref{subsubsec:metrics}.

\newcommand{\methodcell}[1]{%
	\begin{tabular}[c]{@{}l@{}}#1\end{tabular}}

\subsubsection{Bilevel-induced LPCC instances}
\label{subsec:bilevel-instances}
\label{sec:bilevel-degenerate}
\label{subsubsec:bilevel-lpcc}
\label{sec:bilevel-degenerate-lpcc}

We consider bilevel problems of the form
\begin{equation}
	\begin{aligned}
		\min_{x,y}\quad & c_x\T x+c_y\T y \\
		\text{s.t.}\quad
		& -1\le x\le1,\qquad -1\le y\le1, \\
		& y\in\argmin_{\widehat y}\left\{
		\frac12\widehat y\T H\widehat y+(Nx+h)\T\widehat y
		\ \middle|\ G\widehat y+Ex+g\ge0
		\right\},
		\qquad H\succ0.
	\end{aligned}
	\label{prob:bilevel-source-lpcc}
\end{equation}
Replacing the lower-level problem with its KKT system yields the following
equivalent LPCC reformulation:
\begin{equation}
	\begin{aligned}
		\min_{x,y,\lambda,w}\quad & c_x\T x+c_y\T y \\
		\text{s.t.}\quad
		& Hy+Nx+h-G\T\lambda=0, \\
		& w-Gy-Ex-g=0, \\
		& -1\le x\le1,
		\qquad -1\le y\le1, \\
		& 0\le\lambda\perp w\ge0.
	\end{aligned}
	\label{prob:bilevel-generated-lpcc}
\end{equation}
For greater generality, the generated LPCCs use the linear objective $c\T z$,
where $z=(x,y,\lambda,w)$ and the coefficients of $\lambda$ and $w$ may be
nonzero.  We also impose the box constraints $0\le\lambda,w\le10$ in
\eqref{prob:bilevel-generated-lpcc}.

The generator provides a feasible initial point
$z^0=(x^0,y^0,\lambda^0,w^0)$.  At this point, $60\%$ of the lower-level
inequalities are active, and the active constraints are arranged in identical
groups of three.  Their gradients are therefore linearly dependent, so LICQ
fails at $z^0$.  All algorithms start from this point.

All reported instances use $n_x=n_y=400$ and
$m\in\{400,800,1600\}$.  For each value of $m$, we generate ten instances.
The complete generation procedure is given in
Algorithm~\ref{alg:bilevel-lpcc-generation} in
Appendix~\ref{app:bilevel-instance-generation}.

\runinhead{\textbf{Computational results.}}
Table~\ref{tab:lpcc-bilevel-summary} compares \RBLA{}, \RPSA{}, the big-\(M\)
MILP, LPCC branch-and-cut, IPOPT--Scholtes, PIP(0.6), and PL-E*, with Gurobi
used by the indicated baselines.  Here PIP(0.6) denotes PIP with
$p_{\max}=0.6$.  The time limits are $600$, $900$, and $1800$ seconds for
$m=400$, $800$, and $1600$, respectively.

\begin{table}[H]
	\centering
	\caption{Performance on bilevel-induced LPCC instances.}
	\label{tab:lpcc-bilevel-summary}
	\scriptsize
	\setlength{\tabcolsep}{1.5pt}
	\resizebox{\textwidth}{!}{%
		\begin{tabular}{@{}c l@{\hspace{6pt}} r@{\hspace{6pt}} r@{\hspace{6pt}} r r r r r r r@{}}
			\hline
			$m$ & Method & Term. & Feas. & \shortstack{Max aff.\\vio.} & \shortstack{Max comp.\\vio.} & \shortstack{Mean obj.\\improv. gap} & \shortstack{Near-best\\observed} & \shortstack{Mean samp.\\br. desc.} & \shortstack{Pass samp.\\br. desc.} & \shortstack{Mean time\\(s)} \\
			\hline
			400 & \RBLA{} (Gurobi) & 10/10 & 10/10 & $6.4{\times}10^{-14}$ & $0.0$ & 0.0308 & 9/10 & $5.87{\times}10^{-4}$ & 6/10 & 4.305 (0.339) \\
			400 & \RBLA{} (A-Simplex) & 10/10 & 10/10 & $8.3{\times}10^{-12}$ & $0.0$ & 0.0282 & 10/10 & $9.20{\times}10^{-4}$ & 5/10 & 9.707 (0.811) \\
			400 & \RPSA{} (A-Simplex) & 10/10 & 10/10 & $2.9{\times}10^{-13}$ & $0.0$ & 0.0213 & 10/10 & $7.73{\times}10^{-9}$ & 10/10 & 7.568 (0.720) \\
			400 & Big-\(M\) (Gurobi) & 0/10 & 10/10 & $3.8{\times}10^{-11}$ & $0.0$ & 0.0277 & 9/10 & $4.12{\times}10^{-3}$ & 0/10 & 600.037 (0.012) \\
			400 & LPCC branch-and-cut & 0/10 & 10/10 & $6.3{\times}10^{-14}$ & $0.0$ & 0.0283 & 9/10 & $2.93{\times}10^{-3}$ & 3/10 & 600.004 (0.000) \\
			400 & IPOPT--Scholtes & 2/10 & 4/10 & $6.2{\times}10^{-8}$ & $9.4{\times}10^{-5}$ & 0.0029 & 4/10 & $2.72{\times}10^{-7}$ & 4/10 & 566.238 (24.464) \\
			400 & PIP(0.6) (Gurobi) & 0/10 & 10/10 & $2.5{\times}10^{-13}$ & $0.0$ & 0.0007 & 10/10 & $5.40{\times}10^{-6}$ & 10/10 & 600.036 (0.002) \\
			400 & PL-E* (Gurobi) & 10/10 & 10/10 & $5.8{\times}10^{-14}$ & $0.0$ & 0.2975 & 0/10 & $5.07{\times}10^{-2}$ & 0/10 & 2.922 (0.047) \\
			\hline
			800 & \RBLA{} (Gurobi) & 10/10 & 10/10 & $1.8{\times}10^{-13}$ & $0.0$ & 0.0537 & 5/10 & $3.91{\times}10^{-4}$ & 5/10 & 14.348 (1.002) \\
			800 & \RBLA{} (A-Simplex) & 10/10 & 10/10 & $6.5{\times}10^{-13}$ & $0.0$ & 0.0601 & 3/10 & $4.19{\times}10^{-4}$ & 6/10 & 41.399 (2.096) \\
			800 & \RPSA{} (A-Simplex) & 10/10 & 10/10 & $4.4{\times}10^{-13}$ & $0.0$ & 0.0530 & 7/10 & $2.38{\times}10^{-7}$ & 10/10 & 35.183 (2.454) \\
			800 & Big-\(M\) (Gurobi) & 0/10 & 10/10 & $5.0{\times}10^{-11}$ & $0.0$ & 0.0573 & 5/10 & $9.04{\times}10^{-3}$ & 0/10 & 900.049 (0.013) \\
			800 & LPCC branch-and-cut & 0/10 & 10/10 & $1.6{\times}10^{-13}$ & $0.0$ & 0.0729 & 2/10 & $1.27{\times}10^{-3}$ & 2/10 & 900.007 (0.001) \\
			800 & IPOPT--Scholtes & 0/10 & 0/10 & $3.4{\times}10^{-5}$ & $8.0{\times}10^{-3}$ & -- & 0/10 & -- & 0/10 & 875.203 (26.990) \\
			800 & PIP(0.6) (Gurobi) & 0/10 & 10/10 & $9.8{\times}10^{-10}$ & $8.6{\times}10^{-15}$ & 0.0010 & 10/10 & $1.13{\times}10^{-4}$ & 7/10 & 900.058 (0.005) \\
			800 & PL-E* (Gurobi) & 10/10 & 10/10 & $1.6{\times}10^{-13}$ & $0.0$ & 0.3875 & 0/10 & $5.54{\times}10^{-2}$ & 0/10 & 6.780 (0.085) \\
			\hline
			1600 & \RBLA{} (Gurobi) & 10/10 & 10/10 & $3.8{\times}10^{-13}$ & $0.0$ & 0.1150 & 0/10 & $5.86{\times}10^{-5}$ & 8/10 & 98.029 (5.343) \\
			1600 & \RBLA{} (A-Simplex) & 9/10 & 10/10 & $7.6{\times}10^{-12}$ & $0.0$ & 0.1462 & 0/10 & $1.88{\times}10^{-3}$ & 3/10 & 962.381 (114.855) \\
			1600 & \RPSA{} (A-Simplex) & 10/10 & 10/10 & $1.6{\times}10^{-12}$ & $0.0$ & 0.1136 & 0/10 & $2.30{\times}10^{-7}$ & 10/10 & 398.248 (15.336) \\
			1600 & Big-\(M\) (Gurobi) & 0/10 & 10/10 & $4.7{\times}10^{-11}$ & $0.0$ & 0.3257 & 0/10 & $4.83{\times}10^{-2}$ & 0/10 & 1800.133 (0.020) \\
			1600 & LPCC branch-and-cut & 0/10 & 10/10 & $4.6{\times}10^{-13}$ & $0.0$ & 0.2110 & 0/10 & $5.70{\times}10^{-3}$ & 0/10 & 1800.017 (0.003) \\
			1600 & IPOPT--Scholtes & 0/10 & 0/10 & $5.4{\times}10^{-14}$ & $9.9{\times}10^{-1}$ & -- & 0/10 & -- & 0/10 & 1603.954 (170.196) \\
			1600 & PIP(0.6) (Gurobi) & 0/10 & 10/10 & $4.0{\times}10^{-12}$ & $0.0$ & 0 & 10/10 & $2.05{\times}10^{-4}$ & 6/10 & 1800.098 (0.003) \\
			1600 & PL-E* (Gurobi) & 10/10 & 10/10 & $3.9{\times}10^{-13}$ & $0.0$ & 0.5687 & 0/10 & $5.32{\times}10^{-2}$ & 0/10 & 25.885 (0.544) \\
			\hline
		\end{tabular}
	}
\end{table}

PIP(0.6), implemented with Gurobi, attains the largest near-best-observed
counts and negligible mean objective-improvement gaps at all tested
dimensions.  Its progressive
mixed-integer programming (MIP) subproblems explore complementarity patterns
more broadly and thereby strengthen the search.
This comes at a cost: PIP solves multiple mixed-integer subproblems and does
not terminate within the time limit on these large instances.  By comparison,
\RPSA{} terminates on every instance, has the smallest mean sampled branch
descent, and passes the sampled branch-descent diagnostic in every run.  It
also attains a competitive near-best-observed count at $m=400$.
As $m$ increases, the branch structure becomes more complex, making it harder
for a local simplex-based search to attain very small objective-improvement
gaps.

Next, we compare \RBLA{} and \RPSA{} to illustrate the advantage of replacing
complete branch-LP solves with individual simplex steps.
Table~\ref{tab:lpcc-inhouse-step-summary} reports this comparison, with both
methods using the same A-Simplex implementation.  All entries are means
over the ten instances at each value of $m$.  ``Mean obj.'' is the mean final
objective value.  ``Mean calls'' counts complete branch-LP solves for \RBLA{}
and simplex vertex-search calls for \RPSA{}; in both cases, it also equals the
mean number of outer iterations.  ``Mean steps'' counts basis-changing simplex
steps, and ``Mean time (s)'' is the mean computation time in seconds.  Across
all dimensions, \RPSA{} attains a lower mean objective value with fewer steps
and less computation time.  At $m=1600$, it uses about $52\%$ of \RBLA{}'s
steps and $41\%$ of its runtime.

\begin{table}[H]
	\centering
	\caption{Step counts for \RBLA{} and \RPSA{} with the same A-Simplex
		implementation.}
	\label{tab:lpcc-inhouse-step-summary}
	\scriptsize
	\setlength{\tabcolsep}{5pt}
	\begin{tabular}{c l@{\hspace{10pt}} r r r r}
		\hline
		$m$ & Method & \shortstack{Mean\\obj.} & \shortstack{Mean\\calls} & \shortstack{Mean\\steps} & \shortstack{Mean time\\(s)} \\
		\hline
		400 & \RBLA{} (A-Simplex) & -18.8489 & 8.4 & 1779.8 & 9.707 \\
		400 & \RPSA{} (A-Simplex) & -18.9822 & 1080.3 & 1077.0 & 7.568 \\
		\hline
		800 & \RBLA{} (A-Simplex) & -23.8950 & 17.0 & 4275.4 & 41.399 \\
		800 & \RPSA{} (A-Simplex) & -24.0791 & 2640.4 & 2636.4 & 35.183 \\
		\hline
		1600 & \RBLA{} (A-Simplex) & -41.5030 & 49.1 & 34888.5 & 962.381 \\
		1600 & \RPSA{} (A-Simplex) & -42.9848 & 18269.5 & 18265.8 & 398.248 \\
		\hline
	\end{tabular}
\end{table}

The larger number of \RPSA{} calls therefore does not indicate more LP work:
\RPSA{} stops at improving simplex vertices instead of optimizing each sampled
branch LP, thereby reducing the total number of basis-changing simplex steps.

\runinhead{\textit{$\rho$-sensitivity experiment}.}
We tested \RPSA{} with $\rho\in\{0.1,1,10,100,1000\}$ on all ten LPCC
instances with $m=1600$.  For $\rho=0.1$, none of the ten returned points
satisfies complementarity, whereas all ten outputs are feasible for
$\rho\ge1$.  At $\rho=1$, only one of the ten endpoints agrees numerically
with its corresponding endpoint for $\rho=100$.  By contrast, for
$\rho\in\{10,100,1000\}$, all ten endpoints agree within the $10^{-8}$
comparison tolerance and have the same objective values and reported
algorithmic counts.  Nine endpoints are bit-for-bit identical; for
\texttt{ins009}, the only non-bitwise-identical case, the maximum coordinate
difference is below $6\times10^{-11}$.  Thus, the tested values $\rho\ge10$
form a stable empirical penalty plateau.  This behavior is consistent with a
sufficiently large penalty on these instances, but it neither verifies
$\rho>\bar\rho$ nor estimates a theoretical penalty threshold.

The explicit test $\psi_{J^k}(z)=0$ in Algorithm~\ref{alg:lpcc-rpsa} ensures
that every accepted iterate of the exact algorithm is complementarity
feasible.  Its convergence theorem additionally requires $\rho>\bar\rho$ and
completion of each sampled search until either an acceptable vertex is found
or penalty-LP optimality is reached.  The capped implementation used here
omits the explicit complementarity test at intermediate improving vertices.
Proposition~\ref{prop:lpcc-branch-penalty-equivalence} guarantees branch
feasibility of penalty-LP optimizers for a sufficiently large $\rho$, but not
of intermediate simplex vertices; hence
Theorem~\ref{thm:lpcc-rpsa-convergence} does not apply to this implementation.
The $\rho=0.1$ stress test exposes the resulting possibility of terminating in
$\cC\setminus\cOmega$.

Table~\ref{tab:rpsa-rho-sensitivity} reports detailed results for the
representative instance \texttt{ins000}.  The five runs were performed
serially, with a 600-second limit for each value of $\rho$.  Objective
improvement is reported only for feasible outputs.  Because each row
represents one run, the check marks and crosses in \textit{Term.},
\textit{Feas.}, and \textit{Pass samp. br. desc.} indicate whether that run
satisfies the corresponding condition.

\begin{table}[H]
	\centering
	\caption{\RPSA{} penalty-parameter sensitivity on the $m=1600$ instance
		\texttt{ins000}.}
	\label{tab:rpsa-rho-sensitivity}
	\scriptsize
	\setlength{\tabcolsep}{3.5pt}
	\begin{tabular}{@{}r r r r r r r r@{}}
		\hline
		$\rho$ & Term. & Feas. & \shortstack{Max aff.\\vio.} & \shortstack{Max comp.\\vio.} & \shortstack{Objective\\improvement} & \shortstack{Pass samp.\\br. desc.} & Time (s) \\
		\hline
		$0.1$  & $\checkmark$ & $\times$     & $1.03{\times}10^{-12}$ & $1.11$ & --      & $\times$     & 368.207 \\
		$1$    & $\checkmark$ & $\checkmark$ & $1.26{\times}10^{-12}$ & $0.0$  & 39.6812 & $\checkmark$ & 345.471 \\
		$10$   & $\checkmark$ & $\checkmark$ & $1.15{\times}10^{-12}$ & $0.0$  & 40.0267 & $\checkmark$ & 323.979 \\
		$100$  & $\checkmark$ & $\checkmark$ & $1.15{\times}10^{-12}$ & $0.0$  & 40.0267 & $\checkmark$ & 319.251 \\
		$1000$ & $\checkmark$ & $\checkmark$ & $1.15{\times}10^{-12}$ & $0.0$  & 40.0267 & $\checkmark$ & 332.543 \\
		\hline
	\end{tabular}
\end{table}

\subsubsection{Affine-GNE network LPCC instances}
\label{subsec:gne-instances}
\label{sec:gne-instances}
\label{subsubsec:gne-lpcc}
\label{sec:gne-lpcc}

We use the variational-equilibrium specialization of the affine-GNE model in
Section~4 of \cite{SchiroPangShanbhag2013}.  This model yields sparse,
large-scale affine complementarity systems.  The linear and quadratic
selection objectives follow the quadratic programming with equilibrium
constraints (QPEC) viewpoint, which combines an upper-level objective with
affine variational inequality (VI) or linear complementarity problem (LCP)
constraints~\cite{jiang1999qpecgen}.

Let $N$ be the dimension of $x$ and let $L$ be the number of shared affine
constraints.  The matrices and vectors satisfy
$H\in\mathbb R^{N\times N}$, $A\in\mathbb R^{L\times N}$,
$b,u\in\mathbb R^N$, and $d\in\mathbb R^L$.  The common feasible set is
defined by
\[
\begin{aligned}
	0\le x &\perp Hx+b+\alpha+A\T\lambda\ge0,\\
	0\le \alpha &\perp u-x\ge0,\\
	0\le \lambda &\perp d-Ax\ge0.
\end{aligned}
\]
Hence there are $m=2N+L$ complementarity pairs.  With
$z=(x,\alpha,\lambda)\in\mathbb R^m$, set
\[
q=\begin{bmatrix}b\\u\\d\end{bmatrix},\qquad
M=\begin{bmatrix}
	H&I&A\T\\
	-I&0&0\\
	-A&0&0
\end{bmatrix}.
\]
Thus, $w=q+Mz=(Hx+b+\alpha+A\T\lambda,u-x,d-Ax)$, and the system is
$0\le z\perp w\ge0$.  Setting $y=(z,w)\in\mathbb R^{2m}$ gives the standard
form
\[
[-M\ I]y=q,\qquad 0\le z\perp w\ge0.
\]

For the LPCC experiments, we use the linear selection problem
\[
\min_y\ \bar c\T y=\min_x\ c\T x,
\]
where all coefficients outside the $x$-block are zero.  Every stored sparse
matrix in the reported affine-GNE instances has density below $2\%$.
Because the Gurobi big-\(M\) MILP reformulation already fails to terminate
within the prescribed time limits at $m=800$ and $1600$ in the preceding
bilevel LPCC experiments, we omit it from the substantially larger affine-GNE
LPCC experiment.  The constraint-generation procedure is given in
Appendix~\ref{app:gne-constraint-generation}.

\runinhead{\textbf{Instance generation.}}
The reported LPCC instances use the linear selection problem above with
$N=4008$, $L=1984$, and $m=10000$ complementarity pairs, giving a
standard-form dimension of $20000$.
Table~\ref{tab:gne-lpcc-summary} reports the computational results.

\begin{table}[htbp]
	\centering
	\caption{Performance on affine-GNE network LPCC instances with $N=4008$ and
		$m=10000$.}
	\label{tab:gne-lpcc-summary}
	\scriptsize
	\setlength{\tabcolsep}{1.5pt}
	\resizebox{\textwidth}{!}{%
		\begin{tabular}{@{}l@{\hspace{6pt}} r@{\hspace{6pt}} r@{\hspace{6pt}} r r r r r r r@{}}
			\hline
			Method & Term. & Feas. & \shortstack{Max aff.\\vio.} & \shortstack{Max comp.\\vio.} & \shortstack{Mean obj.\\improv. gap} & \shortstack{Near-best\\observed} & \shortstack{Mean samp.\\br. desc.} & \shortstack{Pass samp.\\br. desc.} & \shortstack{Mean time\\(s)} \\
			\hline
			\RBLA{} (Gurobi) & 10/10 & 10/10 &
			$9.45{\times}10^{-8}$ & $3.17{\times}10^{-10}$ & $1.41{\times}10^{-5}$ &
			10/10 & $6.79{\times}10^{-12}$ & 10/10 & 67.13 (3.94) \\
			\RBLA{} (A-Simplex) & 0/10 & 10/10 &
			$1.14{\times}10^{-8}$ & $2.19{\times}10^{-9}$ & 0.0241 &
			9/10 & $1.50{\times}10^{-3}$ & 3/10 & 1200.25 (0.14) \\
			\RPSA{} (A-Simplex) & 10/10 & 10/10 &
			$6.24{\times}10^{-11}$ & $2.24{\times}10^{-11}$ & $1.41{\times}10^{-5}$ &
			10/10 & $1.71{\times}10^{-12}$ & 10/10 & 208.88 (16.68) \\
			PIP(0.6) (Gurobi) & 10/10 & 10/10 &
			$3.55{\times}10^{-15}$ & $0$ & 1 &
			0/10 & $4.08{\times}10^{-2}$ & 0/10 & 450.76 (0.03) \\
			IPOPT--Scholtes & 5/10 & 1/10 &
			$5.58{\times}10^{-4}$ & $6.22{\times}10^{-4}$ & 0 &
			1/10 & $0$ & 1/10 & 1033.94 (54.72) \\
			CCOpt & 7/10 & 3/10 &
			$1.54{\times}10^{-10}$ & $9.24{\times}10^{-5}$ & 1 &
			0/10 & $3.96{\times}10^{-2}$ & 0/10 & 929.65 (83.82) \\
			PL-E* (Gurobi) & 0/10 & 0/10 &
			$1.27{\times}10^{-11}$ & $1.96$ & -- &
			0/10 & -- & 0/10 & 69.83 (1.95) \\
			\hline
		\end{tabular}%
	}
\end{table}

Both \RBLA{} with Gurobi and \RPSA{} with A-Simplex produce
near-best-observed outputs and pass the sampled branch-descent diagnostic in
all ten runs.  \RBLA{} with Gurobi is fastest, while under the same A-Simplex
LP solver \RPSA{} terminates on all instances whereas \RBLA{} reaches the time
limit, showing the benefit of avoiding repeated complete branch-LP solves.
PIP leaves the initial feasible point unchanged because its restricted MIP
subproblems reach their prescribed time limits without finding an improving
solution.  Thus, unlike on the smaller bilevel LPCC instances, its stronger
mixed-integer search does not translate into an improved solution at this
scale.

\subsubsection{LPCC reformulations of inverse quadratic programs}
\label{subsubsec:inverse-qp-lpcc}

\runinhead{\textbf{Instance generation.}}
We follow the inverse-QP construction and parameter setting in
\cite[Section~6.3]{JaraMoroniPangWachter2018LPCC}.  Let
$Q\in\mathbb R^{r\times r}$ be symmetric positive definite,
$A\in\mathbb R^{m\times r}$, and
$(\bar x,\bar b,\bar c)\in\mathbb R^r\times\mathbb R^m\times\mathbb R^r$
be given reference vectors.  The inverse problem is
\begin{equation}
	\begin{aligned}
		\min_{x,b,c}\quad
		& \|x-\bar x\|_1+\|b-\bar b\|_1+\|c-\bar c\|_1 \\
		\text{s.t.}\quad
		& x\in\argmin_{y\in\mathbb R^r}
		\left\{\frac12y\T Qy+c\T y\ \middle|\ Ay\ge b\right\}.
	\end{aligned}
	\label{prob:inverse-qp}
\end{equation}
Let $u^x,u^b,u^c,u^\lambda>0$ be the bound values used in the reference
construction.  Since $Q\succ0$, replacing the lower-level problem by its KKT
conditions and introducing auxiliary variables
$z^x\in\mathbb R^r$, $z^b\in\mathbb R^m$, and
$z^c\in\mathbb R^r$ to linearize the $\ell_1$-objective gives the LPCC
\begin{equation}
	\begin{aligned}
		\min_{x,b,c,z^x,z^b,z^c,\lambda,w}\quad
		& \bm 1\T z^x+\bm 1\T z^b+\bm 1\T z^c \\
		\text{s.t.}\quad
		& Qx+c-A\T\lambda=0, \\
		& w-Ax+b=0, \\
		& -z^x\le x-\bar x\le z^x, \\
		& -z^b\le b-\bar b\le z^b, \\
		& -z^c\le c-\bar c\le z^c, \\
		& -u^x\le x\le u^x,\quad -u^b\le b\le u^b, \\
		& -u^c\le c\le u^c, \\
		& z^x,z^b,z^c\ge0, \\
		& 0\le\lambda\perp w\ge0,\quad \lambda\le u^\lambda.
	\end{aligned}
	\label{prob:inverse-qp-lpcc}
\end{equation}
This is the standard LPCC form used in this paper and is equivalent to
formulation (33) in the cited paper through $w=Ax-b$.  Hence, the resulting
LPCC has $m$ complementarity pairs.

Following the reference dimension rule, we set $r=0.75m$.  The parameter
$s$ is the target fraction of nonzero entries used when generating $A$ and
the sparse symmetric matrix $Q$.  We set $s=\min\{1,10/r\}$, so each row has
about ten nonzero entries on average.  The eigenvalues of $Q$ lie in
$[0.5,1]$.  A feasible KKT point $(x,b,c,\lambda,w)$ is generated first,
after which $(x,b,c)$ is perturbed to form $(\bar x,\bar b,\bar c)$.  Before
the $4r+2m=5m$ linear inequalities representing the $\ell_1$-objective are
converted to equalities with nonnegative slacks,
\eqref{prob:inverse-qp-lpcc} has $r+m=7m/4$ linear equalities.  The resulting
solver form has $12m$ variables and $27m/4$ equalities.  We generate ten
instances for each $m\in\{500,1000\}$; the corresponding solver forms have
$6000$ and $12000$ variables and $3375$ and $6750$ equalities, respectively.
All methods start from the generated feasible KKT point.

\runinhead{\textbf{Computational results.}}
Table~\ref{tab:inverse-qp-lpcc-summary} compares \RBLA{} with Gurobi and
A-Simplex, \RPSA{} with A-Simplex, IPOPT--Scholtes, and PL-E*.  In this
experiment, \RBLA{} (Gurobi) uses the barrier method without crossover or
presolve, while \RPSA{} uses $\rho=1000$ and at most $500$ simplex pivots per
sampled piece.  The time limits are $120$ seconds for $m=500$ and $600$
seconds for $m=1000$.

\begin{table}[H]
	\centering
	\caption{Performance on inverse-QP-induced LPCC instances.}
	\label{tab:inverse-qp-lpcc-summary}
	\scriptsize
	\setlength{\tabcolsep}{1.5pt}
	\resizebox{\textwidth}{!}{%
		\begin{tabular}{@{}c l@{\hspace{6pt}} r@{\hspace{6pt}} r@{\hspace{6pt}} r r r r r r r@{}}
			\hline
			$m$ & Method & Term. & Feas. & \shortstack{Max aff.\\vio.} & \shortstack{Max comp.\\vio.} & \shortstack{Mean obj.\\improv. gap} & \shortstack{Near-best\\observed} & \shortstack{Mean samp.\\br. desc.} & \shortstack{Pass samp.\\br. desc.} & \shortstack{Mean time\\(s)} \\
			\hline
			500 & \RBLA{} (Gurobi) & 10/10 & 10/10 & $1.69{\times}10^{-7}$ & $2.60{\times}10^{-13}$ & 0.0075 & 10/10 & $4.30{\times}10^{-5}$ & 10/10 & 3.138 (0.319) \\
			500 & \RBLA{} (A-Simplex) & 0/10 & 10/10 & $2.13{\times}10^{-14}$ & $0$ & 1 & 0/10 & $2.58{\times}10^{-1}$ & 0/10 & 120.001 (0.001) \\
			500 & \RPSA{} (A-Simplex) & 10/10 & 10/10 & $2.38{\times}10^{-13}$ & $0$ & 0.0059 & 10/10 & $1.30{\times}10^{-6}$ & 10/10 & 25.162 (2.068) \\
			500 & IPOPT--Scholtes & 9/10 & 9/10 & $2.84{\times}10^{-14}$ & $5.86{\times}10^{-5}$ & 0 & 9/10 & $2.18{\times}10^{-8}$ & 9/10 & 72.911 (6.566) \\
			500 & PL-E* (Gurobi) & 10/10 & 10/10 & $1.57{\times}10^{-10}$ & $0$ & 0.0198 & 10/10 & $2.31{\times}10^{-3}$ & 1/10 & 5.286 (0.137) \\
			\hline
			1000 & \RBLA{} (Gurobi) & 10/10 & 10/10 & $7.57{\times}10^{-9}$ & $1.10{\times}10^{-14}$ & 0.0051 & 10/10 & $6.04{\times}10^{-6}$ & 10/10 & 9.555 (0.519) \\
			1000 & \RBLA{} (A-Simplex) & 0/10 & 10/10 & $2.13{\times}10^{-14}$ & $0$ & 1 & 0/10 & $2.57{\times}10^{-1}$ & 0/10 & 600.003 (0.001) \\
			1000 & \RPSA{} (A-Simplex) & 10/10 & 10/10 & $8.37{\times}10^{-13}$ & $0$ & 0.0033 & 10/10 & $1.96{\times}10^{-6}$ & 10/10 & 311.654 (17.123) \\
			1000 & IPOPT--Scholtes & 5/10 & 5/10 & $3.55{\times}10^{-14}$ & $1.43{\times}10^{-4}$ & 0 & 5/10 & $2.18{\times}10^{-8}$ & 5/10 & 483.154 (45.291) \\
			1000 & PL-E* (Gurobi) & 10/10 & 10/10 & $2.39{\times}10^{-11}$ & $0$ & 0.0140 & 10/10 & $2.00{\times}10^{-3}$ & 0/10 & 23.460 (0.864) \\
			\hline
		\end{tabular}%
	}
\end{table}

Both \RBLA{} with Gurobi and \RPSA{} with A-Simplex consistently return
near-best-observed outputs and pass the sampled branch-descent diagnostic.
\RBLA{} gives the best overall efficiency with a mature sparse LP solver,
whereas \RPSA{} combines comparable reliability with slightly better
objective-improvement gaps using the research simplex implementation.  The
contrast with \RBLA{} under the same A-Simplex backend indicates that
pivotwise search scales more effectively on this problem class than repeated
complete branch-LP solves.  The remaining methods do not provide the same
combination of reliability and sampled branch-descent quality.

\subsection{QPCC experiments}
\label{subsec:qpcc-experiments}

The QPCC experiments consider instances derived from bilevel problems and
affine-GNE constraint families, with both convex and nonconvex quadratic
objectives.

\subsubsection{Bilevel-induced QPCC instances}
\label{subsubsec:bilevel-qpcc}
\label{sec:bilevel-degenerate-qpcc}

The bilevel-induced QPCC instances are generated using the same
feasibility-by-construction principle as the LPCC instances in
Section~\ref{subsubsec:bilevel-lpcc}.  The QPCC-specific dimensions,
duplication ratio, bounds, slack distribution, and multiplier rule are given
in Appendix~\ref{app:bilevel-instance-generation}.  Let
$z=(x,y,\lambda,w)$.  The upper-level objective is
$\frac12 z\T Qz+q\T z$.  For the convex instances, we set
$Q=Q_{\rm psd}=R\T R+10^{-4}I$.  For the nonconvex instances, we apply the
low-rank shift
\[
Q=Q_{\rm psd}-0.05\lambda_{\max}(Q_{\rm psd})UU\T,
\]
where the columns of $U$ are the ten eigenvectors associated with the smallest
eigenvalues of $Q_{\rm psd}$.  The entries of $q$ are sampled from $[-1,1]$.
For each objective class, we generate ten instances for every
$m\in\{400,800,1600\}$; the corresponding QPCC dimensions are
$n=1200,2000,3600$.

\runinhead{\textbf{Computational results.}}

On the convex instances, Table~\ref{tab:qpcc-bilevel-psd-summary} compares
\TBTA{}, MPECopt, LCQPow-D, and IPOPT--Scholtes.  The time limits are $180$
seconds for $m=400$, $500$ seconds for $m=800$, and $1200$ seconds for
$m=1600$.

\begin{table}[H]
	\centering
	\caption{Performance on bilevel-induced convex QPCC instances.}
	\label{tab:qpcc-bilevel-psd-summary}
	\scriptsize
	\setlength{\tabcolsep}{1.8pt}
	\resizebox{\textwidth}{!}{%
		\begin{tabular}{@{}c l@{\hspace{5pt}} r r r r r r r r r@{}}
			\hline
			$m$ & Method & Term. & Feas. & \shortstack{Max aff.\\vio.} & \shortstack{Max comp.\\vio.} & \shortstack{Mean obj.\\improv. gap} & \shortstack{Near-best\\observed} & \shortstack{Mean samp.\\br. desc.} & \shortstack{Pass samp.\\br. desc.} & \shortstack{Mean time\\(s)} \\
			\hline
			400 & \TBTA{} & 10/10 & 10/10 & $5.13{\times}10^{-14}$ & $0$ & 0.252 & 0/10 & $1.28{\times}10^{-4}$ & 10/10 & 33.9 (3.2) \\
			400 & MPECopt & 10/10 & 10/10 & $1.43{\times}10^{-12}$ & $4.49{\times}10^{-12}$ & 0.246 & 0/10 & $3.12{\times}10^{-9}$ & 10/10 & 76.0 (7.5) \\
			400 & LCQPow-D & 10/10 & 10/10 & $4.22{\times}10^{-15}$ & $2.66{\times}10^{-16}$ & 0 & 10/10 & $3.57{\times}10^{-17}$ & 10/10 & 44.2 (0.9) \\
			400 & IPOPT--Scholtes & 0/10 & 4/10 & $1.50{\times}10^{-5}$ & $1.01{\times}10^{-3}$ & 0.0841 & 0/10 & $3.99{\times}10^{-2}$ & 0/10 & 183.7 (0.3) \\
			\hline
			800 & \TBTA{} & 10/10 & 10/10 & $1.10{\times}10^{-13}$ & $0$ & 0.327 & 0/10 & $9.49{\times}10^{-5}$ & 10/10 & 107.2 (14.1) \\
			800 & MPECopt & 10/10 & 10/10 & $2.29{\times}10^{-12}$ & $5.42{\times}10^{-12}$ & 0.342 & 0/10 & $1.38{\times}10^{-8}$ & 10/10 & 252.2 (10.7) \\
			800 & LCQPow-D & 10/10 & 10/10 & $5.55{\times}10^{-15}$ & $3.75{\times}10^{-15}$ & 0 & 10/10 & $1.22{\times}10^{-17}$ & 10/10 & 260.4 (5.4) \\
			800 & IPOPT--Scholtes & 0/10 & 0/10 & $1.62{\times}10^{-2}$ & $5.53{\times}10^{-2}$ & -- & 0/10 & -- & 0/10 & 510.5 (0.8) \\
			\hline
			1600 & \TBTA{} & 10/10 & 10/10 & $1.67{\times}10^{-14}$ & $0$ & 0.0112 & 9/10 & $6.19{\times}10^{-5}$ & 10/10 & 193.2 (32.5) \\
			1600 & MPECopt & 10/10 & 10/10 & $1.60{\times}10^{-12}$ & $4.29{\times}10^{-12}$ & 0.0282 & 7/10 & $1.31{\times}10^{-9}$ & 10/10 & 674.3 (38.2) \\
			1600 & LCQPow-D & 0/10 & 3/10 & $3.21{\times}10^{-11}$ & $1.74$ & 1 & 0/10 & $8.58{\times}10^{-2}$ & 0/10 & 1200.4 (0.2) \\
			1600 & IPOPT--Scholtes & 0/10 & 0/10 & $8.08{\times}10^{-2}$ & $2.49{\times}10^{-2}$ & -- & 0/10 & -- & 0/10 & 1230.4 (2.6) \\
			\hline
		\end{tabular}
	}
\end{table}

At $m=400$ and $800$, LCQPow-D has a zero mean gap in objective improvement.
It also produces ten near-best-observed outputs and passes the sampled
branch-descent diagnostic in all ten runs.  \TBTA{} and MPECopt pass this
diagnostic in every run at all three dimensions.  At $m=1600$, LCQPow-D no
longer scales with the dense qpOASES QP solver used here, consistent with the
increasing cost of its dense QP subproblems, whereas \TBTA{} produces more
near-best-observed outputs than MPECopt in less than one-third of its mean
computation time.
Overall, \TBTA{} provides the best balance of objective improvement and
computation time as the dimension increases.

Table~\ref{tab:qpcc-bilevel-summary} compares \TBTA{}, MPECopt,
IPOPT--Scholtes, big-\(M\) Gurobi, and DCA3-Gurobi on the nonconvex instances under
the same time limits.

\begin{table}[H]
	\centering
	\caption{Performance on bilevel-induced nonconvex QPCC instances.}
	\label{tab:qpcc-bilevel-summary}
	\scriptsize
	\setlength{\tabcolsep}{1.8pt}
	\resizebox{\textwidth}{!}{%
		\begin{tabular}{@{}c l@{\hspace{5pt}} r r r r r r r r r@{}}
			\hline
			$m$ & Method & Term. & Feas. & \shortstack{Max aff.\\vio.} & \shortstack{Max comp.\\vio.} & \shortstack{Mean obj.\\improv. gap} & \shortstack{Near-best\\observed} & \shortstack{Mean samp.\\br. desc.} & \shortstack{Pass samp.\\br. desc.} & \shortstack{Mean time\\(s)} \\
			\hline
			400 & \TBTA{} & 10/10 & 10/10 & $2.30{\times}10^{-14}$ & $0$ & 0.0737 & 3/10 & $1.463{\times}10^{-4}$ & 10/10 & 23.06 (1.61) \\
			400 & MPECopt & 10/10 & 10/10 & $1.27{\times}10^{-12}$ & $4.84{\times}10^{-12}$ & 0.0715 & 5/10 & $2.839{\times}10^{-9}$ & 10/10 & 70.32 (6.28) \\
			400 & IPOPT--Scholtes & 0/10 & 1/10 & $1.90{\times}10^{-6}$ & $6.32{\times}10^{-4}$ & 0 & 1/10 & $6.624{\times}10^{-7}$ & 1/10 & 174.16 (3.23) \\
			400 & big-\(M\) Gurobi & 0/10 & 10/10 & $8.88{\times}10^{-16}$ & $0$ & 1 & 0/10 & $2.074{\times}10^{-1}$ & 0/10 & 170.18 (3.31) \\
			400 & DCA3-Gurobi & 10/10 & 10/10 & $1.71{\times}10^{-9}$ & $2.37{\times}10^{-11}$ & 0.0380 & 8/10 & $1.815{\times}10^{-2}$ & 0/10 & 39.87 (1.65) \\
			\hline
			800 & \TBTA{} & 10/10 & 10/10 & $8.29{\times}10^{-15}$ & $0$ & 0.0520 & 7/10 & $1.155{\times}10^{-4}$ & 10/10 & 65.80 (3.81) \\
			800 & MPECopt & 10/10 & 10/10 & $1.43{\times}10^{-12}$ & $5.03{\times}10^{-12}$ & 0.0442 & 9/10 & $7.054{\times}10^{-8}$ & 10/10 & 221.83 (15.82) \\
			800 & IPOPT--Scholtes & 0/10 & 1/10 & $1.65{\times}10^{-5}$ & $2.13{\times}10^{-3}$ & 0 & 1/10 & $9.192{\times}10^{-9}$ & 1/10 & 511.07 (32.63) \\
			800 & big-\(M\) Gurobi & 0/10 & 10/10 & $1.55{\times}10^{-15}$ & $0$ & 1 & 0/10 & $1.384{\times}10^{-1}$ & 0/10 & 500.44 (33.27) \\
			800 & DCA3-Gurobi & 9/10 & 9/10 & $1.33{\times}10^{-9}$ & $5.24{\times}10^{-2}$ & 0.6006 & 0/10 & $5.428{\times}10^{-2}$ & 0/10 & 388.36 (20.14) \\
			\hline
			1600 & \TBTA{} & 10/10 & 10/10 & $3.04{\times}10^{-14}$ & $0$ & 0.0109 & 9/10 & $6.518{\times}10^{-5}$ & 10/10 & 228.33 (45.25) \\
			1600 & MPECopt & 10/10 & 10/10 & $1.43{\times}10^{-12}$ & $4.19{\times}10^{-12}$ & 0.0480 & 7/10 & $1.267{\times}10^{-3}$ & 9/10 & 725.70 (89.71) \\
			1600 & IPOPT--Scholtes & 0/10 & 0/10 & $6.98{\times}10^{-2}$ & $9.91{\times}10^{-2}$ & -- & 0/10 & -- & 0/10 & 1841.15 (202.33) \\
			1600 & big-\(M\) Gurobi & 0/10 & 10/10 & $2.44{\times}10^{-15}$ & $0$ & 1 & 0/10 & $8.756{\times}10^{-2}$ & 0/10 & 1209.34 (1.08) \\
			1600 & DCA3-Gurobi & 0/10 & 0/10 & $9.12{\times}10^{-9}$ & $6.66{\times}10^{-1}$ & -- & 0/10 & -- & 0/10 & 1200.35 (0.08) \\
			\hline
		\end{tabular}
	}
\end{table}

At $m=400$, DCA3-Gurobi has the smallest mean objective-improvement gap and the
largest near-best-observed count among the methods that return feasible outputs
in all ten runs.  Its good objective values need not imply small sampled branch
descent, since the DCA3 convergence theory guarantees weak stationarity under
its assumptions, which is weaker than B-stationarity~\cite{LeThiNguyenPhamDinh2023DCLCC}.
Both \TBTA{} and MPECopt pass the sampled branch-descent
diagnostic in every run.  At $m=800$, MPECopt has the best
objective-improvement statistics, whereas \TBTA{} is more than three times
faster.  At $m=1600$, \TBTA{} has the smallest mean objective-improvement gap,
the largest near-best-observed count, and the shortest mean computation time
among the methods that consistently return feasible outputs.  The failure of
the global MIQP formulation to terminate is also consistent with the
combinatorial cost of resolving the complementarity decisions globally.  Thus, \TBTA{}
remains computationally efficient and provides the strongest observed
objective-improvement performance as the dimension increases.

\runinhead{\textbf{TBTA ablation experiment.}}

We compare bQPA with \TBTA{} to assess the practical effect of using the
linearized branch trust-region LPs in \TBTA{} instead of branch QPs.  At each
iteration, bQPA samples a compatible branch and solves its branch QP.
Appendix~\ref{app:bqpa} gives the pseudocode.  We test two branch-QP solvers.
IPOPT seeks a local KKT point, whereas Gurobi seeks and certifies a global
minimizer.  The two variants otherwise follow the same algorithm.  Unlike
\TBTA{}, both variants omit
objective linearization, the trust-region safeguard, and thresholding.
Therefore, this experiment assesses the combined effect of these three
components rather than providing a one-factor ablation.  All three methods are
tested on the ten nonconvex instances with $m=1600$.
Table~\ref{tab:tbta-ablation} reports the results.  Here the mean objective
improvement is the arithmetic mean of
$\Delta_M=f(x^0)-f(x_M)$ over the ten runs.

\begin{table}[H]
	\centering
	\caption{TBTA ablation on the ten bilevel-induced nonconvex QPCC instances
		with $m=1600$.}
	\label{tab:tbta-ablation}
	\scriptsize
	\setlength{\tabcolsep}{3pt}
	\resizebox{\textwidth}{!}{%
		\begin{tabular}{@{}l r r r r r r r@{}}
			\hline
			Method & Term. & Feas. & \shortstack{Max aff.\\vio.} &
			\shortstack{Max comp.\\vio.} &
			\shortstack{Mean objective\\improvement} &
			\shortstack{Mean samp.\\br. desc.} & \shortstack{Mean time\\(s)} \\
			\hline
			\TBTA{} & 10/10 & 10/10 & $3.041{\times}10^{-14}$ & $0$ &
			157.243 & $6.518{\times}10^{-5}$ & 228.33 (45.25) \\
			bQPA--IPOPT & 10/10 & 10/10 & $2.843{\times}10^{-13}$ & $0$ &
			151.523 & $3.297{\times}10^{-4}$ & 148.05 (7.46) \\
			bQPA--Gurobi & 10/10 & 10/10 & $4.736{\times}10^{-10}$ & $0$ &
			151.144 & $9.114{\times}10^{-4}$ & 403.74 (47.04) \\
			\hline
		\end{tabular}%
	}
\end{table}

All three methods satisfy their stopping criteria and return feasible outputs
on all ten instances.  \TBTA{} achieves the largest mean objective improvement
and the smallest mean sampled branch-descent value.  Its mean objective
improvement is approximately $3.8\%$ larger than that of bQPA--IPOPT and
$4.0\%$ larger than that of bQPA--Gurobi.  Although bQPA--IPOPT is faster than
\TBTA{} on average, bQPA--Gurobi is slower than \TBTA{} despite globally
minimizing each sampled branch QP.
Global optimality of an individual sampled branch QP does not imply a better
outer search; in these experiments, this stronger branch solve does not
improve the final objective relative to \TBTA{}.  Thus, on this test set, the
linearized, trust-region-safeguarded subproblems of \TBTA{} yield a larger mean
objective improvement and a smaller mean sampled branch-descent value, but no
uniform runtime advantage.

\subsubsection{Affine-GNE network QPCC instances}
\label{subsubsec:gne-qpcc}
\label{sec:gne-qpcc}

All instances use the affine-GNE complementarity system defined in
Section~\ref{subsubsec:gne-lpcc}.  The convex quadratic selection problem is
\[
\min_y\ \frac12(x-x^{\rm tar})\T
\operatorname{Diag}(\omega)(x-x^{\rm tar})+c\T x,
\qquad \omega\in\mathbb R_{++}^N,\quad 0\le x^{\rm tar}\le u.
\]
Here $x^{\rm tar}$ is a generated target vector, not a known equilibrium or
minimizer.  The components of $\omega$ are sampled independently and uniformly
from $[0.5,2]$, after which those associated with initially zero-flow users are
multiplied by $3$.  The nonconvex quadratic selection problem is
\[
\min_y\ \tfrac12(x-x^{\rm tar})\T\!\left(
\operatorname{Diag}(\omega)
-U\operatorname{Diag}\bigl((\omega_i+\beta)_{i\in I}\bigr)U\T\right)
(x-x^{\rm tar})+c\T x .
\]
Write $Q_x:=\operatorname{Diag}(\omega)$ for the convex $x$-block Hessian.
Here $I$ is the set of coordinates on which negative curvature is introduced,
$U=[e_i]_{i\in I}$, and $\beta>0$.
Because the big-\(M\) Gurobi reformulation fails to terminate within the
prescribed time limits at $m=800$ and $1600$ in the preceding bilevel
nonconvex QPCC experiment, we omit it from the substantially larger
affine-GNE QPCC experiments.

\runinhead{\textbf{Instance generation.}}
The first QPCC data set uses the convex quadratic tracking problem above, with
$N=5000$, $L=1984$, $m=2N+L=11984$, and standard-form dimension $23968$.
Table~\ref{tab:gne-qpcc-summary} reports the computational results for the
convex affine-GNE QPCC instances.

\begin{table}[H]
	\centering
	\caption{Performance on affine-GNE network QPCC instances with PSD tracking
		objectives, $N=5000$, and $m=11984$.}
	\label{tab:gne-qpcc-summary}
	\scriptsize
	\setlength{\tabcolsep}{2pt}
	\resizebox{\textwidth}{!}{%
		\begin{tabular}{@{}l@{\hspace{5pt}} r r r r r r r r r@{}}
			\hline
			Method & Term. & Feas. & \shortstack{Max aff.\\vio.} & \shortstack{Max comp.\\vio.} & \shortstack{Mean obj.\\improv. gap} & \shortstack{Near-best\\observed} & \shortstack{Mean samp.\\br. desc.} & \shortstack{Pass samp.\\br. desc.} & \shortstack{Mean time\\(s)} \\
			\hline
			\TBTA{} & 10/10 & 10/10 & $2.78{\times}10^{-9}$ &
			$0$ & $2.21{\times}10^{-8}$ & 10/10 & $5.78{\times}10^{-7}$ & 10/10 & 349.19 (23.07) \\
			IPOPT--Scholtes & 5/10 & 0/10 & $2.21{\times}10^{-9}$ &
			$8.01{\times}10^{-4}$ & -- & 0/10 & -- & 0/10 & 998.64 (77.81) \\
			MPECopt-early & 0/10 & 10/10 & $2.63{\times}10^{-9}$ &
			$0$ & $0$ & 10/10 & $5.81{\times}10^{-11}$ & 10/10 & 1198.03 (18.68) \\
			CCOpt & 0/10 & 0/10 & $6.64{\times}10^{-8}$ &
			$1.86{\times}10^{-2}$ & -- & 0/10 & -- & 0/10 & 1209.80 (0.32) \\
			LCQPow-OSQP & 0/10 & 0/10 & -- & -- & -- & 0/10 & -- & 0/10 & 1200.02 (0.00) \\
			\hline
		\end{tabular}
	}
\end{table}

Among the tested methods, only \TBTA{} and MPECopt-early consistently obtain
high-quality feasible outputs.  While their observed solution quality is
comparable, \TBTA{} terminates on all instances and has a substantial runtime
advantage, whereas MPECopt-early reaches the time limit before satisfying its
stopping criterion.

The second QPCC data set uses the nonconvex quadratic selection problem above
with the same network dimensions.  We set $|I|=r=10$ and
$\beta=0.05\lambda_{\max}(Q_x)$, so that the selected ten positive eigenvalues
are replaced by $-\beta$.  Table~\ref{tab:gne-qpcc-nonconvex-summary} reports
the results under a $1200$-second time limit.

\begin{table}[H]
	\centering
	\caption{Performance on affine-GNE network QPCC instances with nonconvex
		objectives, $N=5000$, and $m=11984$.}
	\label{tab:gne-qpcc-nonconvex-summary}
	\scriptsize
	\setlength{\tabcolsep}{2.6pt}
	\resizebox{\textwidth}{!}{%
		\begin{tabular}{@{}l@{\hspace{6pt}} r r r r r r r r r@{}}
			\hline
			Method & Term. & Feas. & \shortstack{Max aff.\\vio.} &
			\shortstack{Max comp.\\vio.} & \shortstack{Mean obj.\\improv. gap} &
			\shortstack{Near-best\\observed} & \shortstack{Mean samp.\\br. desc.} &
			\shortstack{Pass samp.\\br. desc.} & \shortstack{Mean time\\(s)} \\
			\hline
			\TBTA{} & 10/10 & 10/10 & $3.87{\times}10^{-8}$ & $0$ &
			$2.11{\times}10^{-6}$ & 10/10 & $7.17{\times}10^{-7}$ &
			10/10 & 420.23 (53.28) \\
			IPOPT--Scholtes & 2/10 & 1/10 & $4.27{\times}10^{-6}$ &
			$1.98{\times}10^{-4}$ & $0$ & 1/10 & $0$ & 1/10 &
			1189.61 (15.88) \\
			MPECopt-early & 0/10 & 10/10 & $3.51{\times}10^{-10}$ & $0$ &
			$8.12{\times}10^{-6}$ & 10/10 & $6.50{\times}10^{-6}$ &
			10/10 & 1223.63 (12.43) \\
			CCOpt & 0/10 & 0/10 & $3.34{\times}10^{-7}$ &
			$1.63{\times}10^{-2}$ & -- & 0/10 & -- & 0/10 & 1209.63 (0.34) \\
			DCA3-Gurobi & 0/10 & 0/10 & $3.43{\times}10^{-11}$ &
			$3.64{\times}10^{-1}$ & -- & 0/10 & -- & 0/10 & 1200.05 (0.01) \\
			\hline
		\end{tabular}%
	}
\end{table}

Again, \TBTA{} and MPECopt-early are the only methods that consistently return
feasible outputs classified as near-best-observed.  Every output from both
methods passes the sampled branch-descent diagnostic.  \TBTA{} is faster,
requiring $420.23$ seconds on average, compared with $1223.63$ seconds for
MPECopt-early.  The similar robustness of \TBTA{} on the convex and nonconvex
network objectives is consistent with its first-order construction: the
objective is linearized and every branch subproblem remains a trust-region LP,
so the low-rank negative-curvature modification does not change the subproblem
class.

\subsection{MPECopt embedding test on MacMPEC}
\label{subsec:macmpec}
\label{sec:macmpec-embedding}

This final experiment tests the proposed LPCC methods as LPCC solvers in the
MPECopt framework described in Section~\ref{subsubsec:methods-baselines},
using the MacMPEC collection~\cite{nurkanovic2025mpecopt,MacMPEC}.  The
original implementation of MPECopt solves each Phase-II LPCC through its big-\(M\)
MILP reformulation with Gurobi.  We compare it with three
replacements: \RBLA{} (Gurobi), \RBLA{} (A-Simplex), and
\RPSA{} (A-Simplex).  Here A-Simplex is the authors' implementation defined in
Section~\ref{subsubsec:methods-baselines}.  Thus, four Phase-II LPCC solvers
are embedded in the same MPECopt framework.

The MacMPEC page reports a reference value for 131 instances.  This experiment
contains 129 instances.  Two unsupported
instances are excluded: \texttt{ex9.1.2}, which contains a binary variable,
and \texttt{ralph1}, which has multiple objectives.  In the tables below, $m$
denotes the number of complementarity pairs in the extracted MPCC.

The computation time limit is 300 seconds for \RBLA{} (Gurobi),
\RPSA{} (A-Simplex), and MILP (Gurobi), and 600 seconds for
\RBLA{} (A-Simplex).  A run terminates when an outer
MPECopt iterate satisfies the numerical stationarity stopping criterion.  For
every terminated run, a direct AMPL residual check verified feasibility for
the original model.  Let $f_{\rm final}$ and
$f_{\rm ref}$ denote the final and database-reference objective values, and set
\[
\varepsilon_{\rm obj}=10^{-4}\max\{1,|f_{\rm ref}|\}.
\]
A terminated run is classified as \emph{match} if
$|f_{\rm final}-f_{\rm ref}|\le\varepsilon_{\rm obj}$, \emph{better} if
$f_{\rm final}<f_{\rm ref}-\varepsilon_{\rm obj}$, and \emph{worse} if
$f_{\rm final}>f_{\rm ref}+\varepsilon_{\rm obj}$.  A run is classified as a
\emph{time limit} if it does not terminate within its prescribed limit.
Table~\ref{tab:macmpec-summary} gives the final outcomes for the four embedded
subproblem solvers.

\begin{table}[H]
	\centering
	\caption{Final outcomes of the four Phase-II LPCC solver configurations on
		the 129 MacMPEC instances.}
	\label{tab:macmpec-summary}
	\resizebox{0.95\textwidth}{!}{%
		\begin{tabular}{lrrrrr}
			\hline
			Method & Match & Better & Worse &
			\shortstack{Time\\limit} & Infeasible \\
			\hline
			\RBLA{} (Gurobi) & 120 & 3 & 6 & 0 & 0 \\
			\RBLA{} (A-Simplex) & 111 & 2 & 6 & 10 & 0 \\
			\RPSA{} (A-Simplex) & 119 & 3 & 7 & 0 & 0 \\
			MILP (Gurobi) & 120 & 3 & 6 & 0 & 0 \\
			\hline
		\end{tabular}
	}
\end{table}

Figure~\ref{fig:macmpec-time-bestknown} plots the match-or-better ratio against
computation time for the same 129 instances.

\begin{figure}[htbp]
	\centering
	\begin{tikzpicture}
		\node[anchor=south west,inner sep=0] (macmpecplot) at (0,0) {%
			\includegraphics[width=0.65\textwidth]{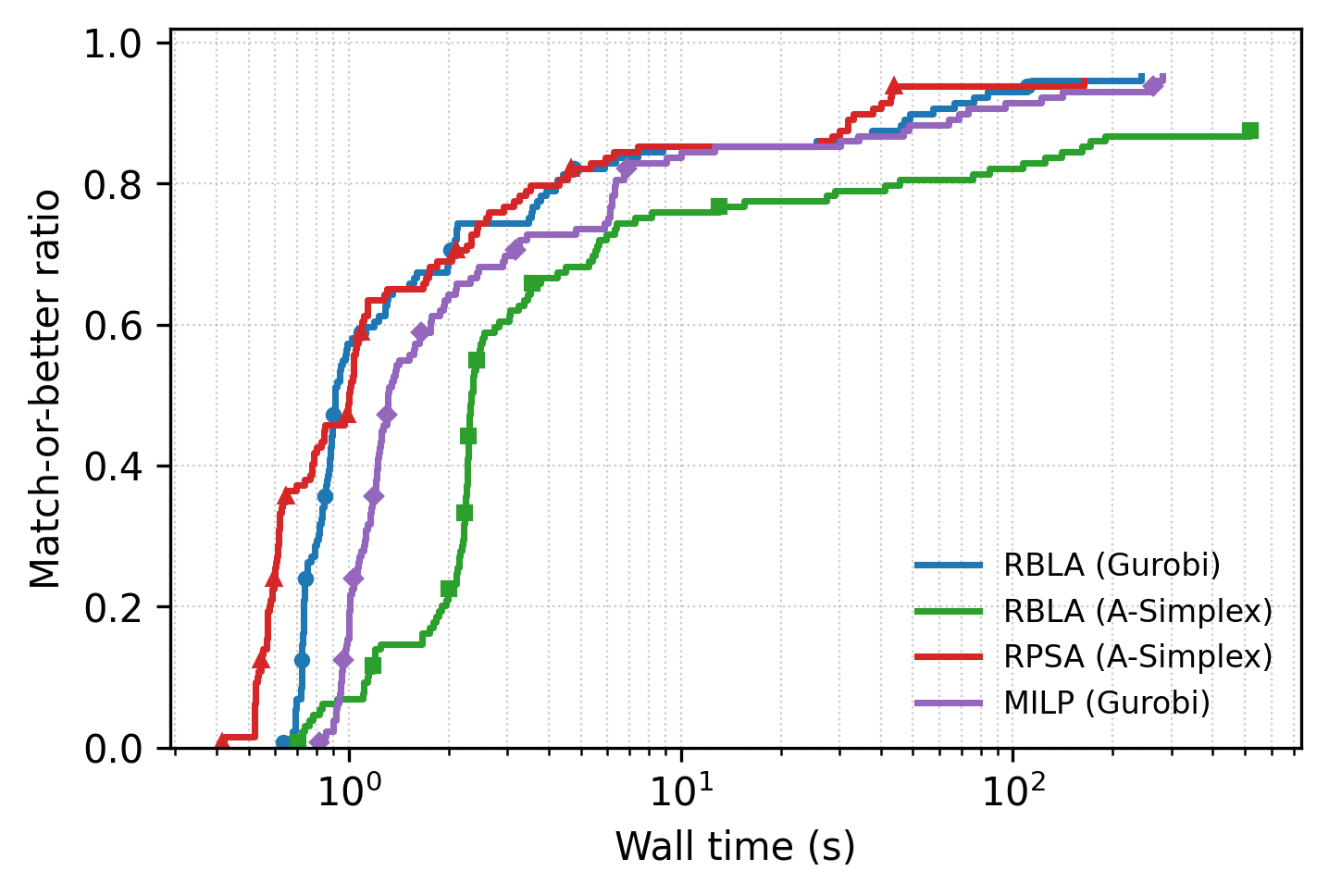}%
		};
		\begin{scope}[x={(macmpecplot.south east)},y={(macmpecplot.north west)}]
			\fill[white] (0.435,0.005) rectangle (0.67,0.085);
			\node[font=\sffamily\fontsize{8.8}{9.5}\selectfont] at (0.5525,0.045)
			{Computation time (s)};
		\end{scope}
	\end{tikzpicture}
	\caption{Match-or-better ratio versus computation time for the four embedded
		Phase-II LPCC solvers}
	\label{fig:macmpec-time-bestknown}
\end{figure}

Table~\ref{tab:macmpec-not-bestknown} lists the ten instances for which at
least one embedded Phase-II LPCC solver does not match the MacMPEC database
objective value.  The \RBLA{} (A-Simplex) column is omitted because ten
of its runs reach the time limit.

\begin{table}[htbp]
	\centering
	\caption{Objective values on the ten nonmatching MacMPEC instances.}
	\label{tab:macmpec-not-bestknown}
	\resizebox{\textwidth}{!}{%
		\begin{tabular}{l r r r r r}
			\hline
			Instance & $m$ & \RBLA{} (Gurobi) & \RPSA{} (A-Simplex) & MILP (Gurobi) & MacMPEC \\
			\hline
			\texttt{b-pn2\_\_bem-milanc30-s} & 1464 & 0.129976 & 0.129976 & 0.129976 & 83.2825 \\
			\texttt{bilevel1} & 6 & 5 & 5 & 5 & 0 \\
			\texttt{bilin} & 6 & 16 & 16 & 16 & 18.4 \\
			\texttt{ex9.2.3} & 6 & 5 & 5 & 5 & -55 \\
			\texttt{ex9.2.5} & 3 & 9 & 9 & 9 & 6 \\
			\texttt{hs044-i} & 10 & 17.0901 & 17.0901 & 17.0901 & 15.6178 \\
			\texttt{pack-comp1\_\_pack-comp-32} & 961 & 0.652985 & 0.653405 & 0.652985 & 0.652979 \\
			\texttt{tap-09\_\_tap-09} & 36 & 110.231 & 110.231 & 110.231 & 109.143 \\
			\texttt{tap-15\_\_tap-15} & 99 & 185.319 & 185.319 & 185.319 & 184.295 \\
			\texttt{water-net\_\_water-net} & 14 & 927.264 & 927.264 & 927.264 & 929.169 \\
			\hline
		\end{tabular}
	}
\end{table}

Table~\ref{tab:macmpec-summary} and Figure~\ref{fig:macmpec-time-bestknown} show
that \RBLA{} (Gurobi) is faster than \RBLA{} (A-Simplex).  The
latter reaches the 600-second limit on ten instances, whereas the longest runs
of the other three methods are all below 282 seconds.  Thus, the current
A-Simplex implementation is not as efficient at solving LPs as mature solvers
such as Gurobi.  By contrast,
Figure~\ref{fig:macmpec-time-bestknown} shows that \RPSA{} (A-Simplex) is
nearly the fastest method, indicating the benefit
of advancing the simplex path without solving every sampled LP to completion.
\RPSA{} records 119 matches, one fewer than the 120 matches obtained by
\RBLA{} (Gurobi) and MILP (Gurobi).  The difference occurs on
\texttt{pack-comp1\_\_}\allowbreak\texttt{pack-comp-32}, where the \RPSA{}
objective differs from
the database reference by approximately $4.3\times10^{-4}$, which is of the
same order as the $10^{-4}$ reporting tolerance.  Thus, \RPSA{} also provides
solution quality comparable to the Gurobi-based variants.

\section{Conclusion}
\label{sec:conclusion}

This paper developed randomized branch-based methods for computing
B-stationary points of LPCCs and QPCCs using only LP subproblems.  The
finite-union geometry of the feasible set provides a branchwise
characterization of B-stationarity.  Based on this characterization,
\RBLA{} solves sampled branch LPs exactly, \RPSA{} permits early acceptance
along simplex paths, and \TBTA{} replaces branch QPs by linearized
trust-region LPs with thresholded branch exploration.  Under the stated
assumptions, the LPCC methods stabilize almost surely at B-stationary points.
For \TBTA{}, the point returned upon finite termination is B-stationary, and
every accumulation point of an infinite run is B-stationary almost surely.

The numerical experiments support the computational potential of this
LP-based branch strategy.  On bilevel-induced, inverse-QP-induced, and sparse
affine generalized Nash equilibrium instances, the proposed methods often achieved objective reductions
comparable to the best feasible results observed within the time limits.
They also performed well on large sparse systems with up to \(11{,}984\)
complementarity pairs.  When embedded as Phase-II LPCC solvers in MPECopt,
\RBLA{} and \RPSA{} attained solution quality and runtimes comparable to those
of the original Gurobi-based big-\(M\) MILP solver on the 129 tested MacMPEC
instances.

\clearpage

\appendix

\section{In-house bounded-variable revised simplex implementation}
\label{app:inhouse-simplex}

We describe the dense implementation used to solve the ordinary LPs arising
inside \RBLA{} and \RPSA{}.  After inequalities are converted to equalities
and free variables are split, the solver works with
\[
    \min_y\ \widetilde c\T y
    \quad\text{s.t.}\quad
    My=h,\qquad \widetilde\ell\le y\le\widetilde u.
\]
If $M$ has $p$ rows, a basis is an ordered set $\mathcal B$ of $p$ columns
such that $B:=M_{\mathcal B}$ is nonsingular.  Every nonbasic variable
$j\in\mathcal N$ is fixed at either $\widetilde\ell_j$ or
$\widetilde u_j$; we call this its \emph{nonbasic side}.  The basic variables
are then determined by the equality system.  The sparse implementation uses
the same pricing, direction, and ratio-test logic, but requires an externally
supplied feasible crash basis and uses sparse basis factorizations; it is not
included in the pseudocode below.

\begin{breakablealgorithm}
\caption{Dense in-house bounded-variable revised simplex}
\label{alg:inhouse-simplex}
\begin{algorithmic}[1]
\Require An LP
\[
    \min_x\ c\T x\quad\text{s.t.}\quad
    A_{\rm eq}x=b_{\rm eq},\quad A_{\rm in}x\ge b_{\rm in},
    \quad \ell\le x\le u,
\]
an optional feasible state $(\mathcal B,\nu)$, where
$\nu_j\in\{L,U\}$ records the bound of each $j\in\mathcal N$,
$\mathcal R\in\{\text{Dantzig},\text{Bland}\}$, and
$\mathsf{mode}\in\{\text{complete},\text{one-pivot}\}$
\Ensure A current point, its objective value and status, and $(\mathcal B,\nu)$
\State Set $A_{\rm in}x-s=b_{\rm in}$, $s\ge0$, replace each free $x_j$ by
       $x_j^+-x_j^-$ with $x_j^\pm\ge0$, and obtain
       $(M,h,\widetilde c,\widetilde\ell,\widetilde u)$ with
       $M\in\mathbb R^{p\times\bar n}$
\If{$M_{\mathcal B}$ is nonsingular and $(\mathcal B,\nu)$ defines a
      feasible point}
    \State Restore $(\mathcal B,\nu)$
\Else
    \State Set $y_j^0=\widetilde\ell_j$ if
           $\widetilde\ell_j> -\infty$, and
           $y_j^0=\widetilde u_j$ otherwise; set $v=h-My^0$
    \State Set $S=\operatorname{Diag}(s_1,\ldots,s_p)$, where
           $s_i=1$ if $v_i\ge0$ and $s_i=-1$ otherwise, and compute
           \[
             \phi_{\rm I}^\star:=
             \min_{y,\,a\in\mathbb R^p}\ \mathbf 1\T a
             \quad\text{s.t.}\quad
             My+Sa=h,\quad
             \widetilde\ell\le y\le\widetilde u,\quad a\ge0,
           \]
           initialized by $(y,a)=(y^0,|v|)$ and basis matrix $S$
    \If{$\phi_{\rm I}^\star>10^{-7}$}
        \State \Return \textsc{infeasible}
    \EndIf
    \State Delete $a$ and retain the resulting feasible $(\mathcal B,\nu)$
\EndIf
\Loop
    \State Set $\mathcal N=\{1,\ldots,\bar n\}\setminus\mathcal B$ and
           $y_j=\widetilde\ell_j$ if $\nu_j=L$, or
           $y_j=\widetilde u_j$ if $\nu_j=U$
    \State Set $B=M_{\mathcal B}$ and solve
           $B y_{\mathcal B}=h-M_{\mathcal N}y_{\mathcal N}$
    \State Solve $B\T\pi=\widetilde c_{\mathcal B}$ and, for
           $j\in\mathcal N$, compute
           $r_j=\widetilde c_j-M_j\T\pi$
    \State Set
           $\mu_j=r_j$ if $\nu_j=L$, and $\mu_j=-r_j$ if $\nu_j=U$
    \If{$\mu_j\ge-10^{-9}$ for every $j\in\mathcal N$}
        \State \Return \textsc{optimal} with the current point and basis
    \EndIf
    \State Let $I=\{j\in\mathcal N:\mu_j<-10^{-9}\}$ and choose
           $q=\min\arg\min_{j\in I}\mu_j$ if $\mathcal R=\text{Dantzig}$,
           or $q=\min I$ if $\mathcal R=\text{Bland}$
    \State Set $\sigma=1$ if $\nu_q=L$, and $\sigma=-1$ if $\nu_q=U$
    \State Solve $B d_{\mathcal B}=-\sigma M_q$
    \State Set
           $\theta_q=\widetilde u_q-y_q$ if $\sigma=1$, and
           $\theta_q=y_q-\widetilde\ell_q$ if $\sigma=-1$
    \For{$i\in\mathcal B$}
        \State Set
        $\displaystyle
        \theta_i=\begin{cases}
        (\widetilde u_i-y_i)/d_i,&d_i>10^{-9},\\
        (\widetilde\ell_i-y_i)/d_i,&d_i<-10^{-9},\\
        +\infty,&\text{otherwise},
        \end{cases}$
    \EndFor
    \State Set $\theta=\min_{j\in\{q\}\cup\mathcal B}\theta_j$ and choose
           $e\in\arg\min_{j\in\{q\}\cup\mathcal B}\theta_j$
    \If{$\theta=+\infty$}
        \State \Return \textsc{unbounded}
    \EndIf
    \State Update $y_q\gets y_q+\sigma\theta$ and
           $y_{\mathcal B}\gets y_{\mathcal B}+\theta d_{\mathcal B}$
    \If{$e=q$}
        \State Set $\nu_q\gets U$ if $\nu_q=L$, and $\nu_q\gets L$ otherwise
    \Else
        \State Replace $e$ by $q$ in $\mathcal B$ and set
               $\nu_e=U$ if $d_e>0$, or $\nu_e=L$ if $d_e<0$; delete $\nu_q$
    \EndIf
    \State Save $(\mathcal B,\nu)$
    \If{$\mathsf{mode}=\text{one-pivot}$}
        \State \Return the current point, objective value, and basis state
    \EndIf
\EndLoop
\end{algorithmic}
\end{breakablealgorithm}

\clearpage

\section{Generation of bilevel-induced LPCC and QPCC instances}
\label{app:bilevel-instance-generation}

\begin{breakablealgorithm}
\caption{Generation of the bilevel-induced LPCC instances}
\label{alg:bilevel-lpcc-generation}
\begin{algorithmic}[1]
\Require $m\in\{400,800,1600\}$, instance index $t\in\{0,\ldots,9\}$,
         and $n_x=n_y=400$
\Ensure LPCC data and a feasible point $z^0$
\State Set $n=n_x+n_y+2m$ and initialize the random-number generator with
       seed $20260606+100000m+t$
\State Sample $x^0\in[-0.2,0.2]^{n_x}$ and
       $y^0\in[-0.2,0.2]^{n_y}$ componentwise uniformly
\State Set $m_{\rm act}=0.6m$ and initialize $G,E,g,\lambda^0,w^0$ to zero
\For{$j=0,\ldots,m_{\rm act}/3-1$}
    \State Set $K_j=\{3j+1,3j+2,3j+3\}$
    \State Sample $\bar G_j\sim\mathcal N(0,n_y^{-1}I)$ and
           $\bar E_j\sim\mathcal N(0,n_x^{-1}I)$
    \For{$i\in K_j$}
        \State Set $G_i=\bar G_j$, $E_i=\bar E_j$, and
               $g_i=-G_i y^0-E_i x^0$
        \State Sample $\lambda_i^0$ uniformly from $[0.5,2]$ and set $w_i^0=0$
    \EndFor
\EndFor
\For{$i=m_{\rm act}+1,\ldots,m$}
    \State Sample $G_i\sim\mathcal N(0,n_y^{-1}I)$ and
           $E_i\sim\mathcal N(0,n_x^{-1}I)$
    \State Sample $s_i$ uniformly from $[0.5,2]$
    \State Set $g_i=s_i-G_i y^0-E_i x^0$, $\lambda_i^0=0$, and $w_i^0=s_i$
\EndFor
\State Sample $R_{kl}\sim\mathcal N(0,n_y^{-1})$ and set
       $H=R\T R+10^{-4}I$
\State Sample $N_{kl}\sim\mathcal N(0,n_x^{-1})$ and set
       $h=G\T\lambda^0-Hy^0-Nx^0$
\State Set $z^0=(x^0,y^0,\lambda^0,w^0)$,
       $-1\le x,y\le1$, and $0\le\lambda,w\le10$
\State Sample $c_i\sim\mathcal N(0,n^{-1})$ for $i=1,\ldots,n$
\State Form the computational LPCC from
       \eqref{prob:bilevel-generated-lpcc} using $c\T z$ and the bounds above,
       and return it with $z^0$
\end{algorithmic}
\end{breakablealgorithm}

The QPCC generators use the same construction of the lower-level equations
and the feasible point, but they also differ from
Algorithm~\ref{alg:bilevel-lpcc-generation} before the objective is generated.
They use
$n_x=n_y=200$, duplicate $40\%$ of the lower-level rows in groups of size two,
bound $\lambda$ and $w$ by $2$, and sample inactive slacks from $[0.1,1]$.
For each active pair, a total multiplier from $[0.1,1]$ is divided equally
between the two copies, whereas the LPCC generator samples the three
multipliers independently.  Thus the pseudocode can be reused as a structural
template after these parameter and multiplier changes; the quadratic objective
is then generated separately.

\clearpage

\section{Generation of Affine-GNE Complementarity Constraints}
\label{app:gne-constraint-generation}

\begin{breakablealgorithm}
\caption{Generation of the affine-GNE complementarity constraints}
\label{alg:gne-constraint-generation}
\begin{algorithmic}[1]
\Require For LPCC, $N=4008$, $K=160$, $\gamma_{\rm act}=0.25$,
a matrix $A\in\{0,1\}^{L\times N}$ satisfying
$A_{\mathord{:},K+i}=A_{\mathord{:},i}$ for $i=1,\ldots,K$, a random seed,
$p_0=0.18$, and $p_\lambda=0.5$
\Ensure $q,M$, finite bounds, and a feasible complementary pair $(z^0,w^0)$
\State Set $P_+=\{1,\ldots,K\}$ and $P_0=\{K+1,\ldots,2K\}$
\State Independently sample
       $\rho_\ell\sim\mathcal U[0.01,0.08]$ for
       $\ell=1,\ldots,L$
\State Independently of $\rho$, sample
       $a_i\sim\mathcal U[0.05,0.2]$ and
       $u_i\sim\mathcal U[0.8,1.5]$ for
       $i=1,\ldots,N$; then set
       $a_i=0$ for $i\in P_+\cup P_0$
\State Sample $\delta_i\sim\operatorname{Bernoulli}(1-p_0)$ independently
       for $i=1,\ldots,N$; then set $\delta_i=1$ for $i\in P_+$ and
       $\delta_i=0$ for $i\in P_0$
\State For each $i$ with $\delta_i=1$, sample
       $r_i\sim\mathcal U[0.25,0.75]$ and set $x_i^0=u_i r_i$; set
       $x_i^0=0$ when $\delta_i=0$
\State Set $v=Ax^0$ and $\mathcal U=\{\ell:v_\ell>10^{-12}\}$; choose
       $\mathcal A\subseteq\mathcal U$ uniformly subject to
       $|\mathcal A|=\lfloor\gamma_{\rm act}|\mathcal U|\rceil$
\State Sample $s_\ell\sim\mathcal U[0.2,1]$ independently for
       $\ell=1,\ldots,L$; set $s_\ell=0$ for $\ell\in\mathcal A$ and
       $d=v+s$
\State Choose $\mathcal P\subseteq\mathcal A$ uniformly subject to
       $|\mathcal P|=\lfloor p_\lambda|\mathcal A|\rceil$; sample
       $\lambda_\ell^0\sim\mathcal U[0.05,1]$ for $\ell\in\mathcal P$ and
       set $\lambda_\ell^0=0$ otherwise
\State Set
       $H=A\T\operatorname{Diag}(\rho)A+\operatorname{Diag}(a)$,
       $\alpha^0=0$, and $b=-Hx^0-A\T\lambda^0$; hence
       $\mu^0:=Hx^0+b+\alpha^0+A\T\lambda^0=0$
\State Set
\[
z^0=\begin{bmatrix}x^0\\\alpha^0\\\lambda^0\end{bmatrix},\quad
w^0=\begin{bmatrix}\mu^0\\u-x^0\\d-Ax^0\end{bmatrix},\quad
q=\begin{bmatrix}b\\u\\d\end{bmatrix},\quad
M=\begin{bmatrix}H&I_N&A\T\\-I_N&0&0\\-A&0&0\end{bmatrix}.
\]
\State Set the componentwise bounds
       $0\le x\le u$, $0\le\alpha\le10$, $0\le\lambda\le10$,
       $0\le u-x\le u$, and $0\le d-Ax\le d$
\State For $i=1,\ldots,N$, set
       $\overline\mu_i=\max\{10,|b_i|+(Hu)_i
       +10(A\T\mathbf1_L)_i\}$ and impose $0\le\mu\le\overline\mu$
\State Return $q,M$, the bounds, and $(z^0,w^0)$
\end{algorithmic}
\end{breakablealgorithm}

Here $\lfloor\cdot\rceil$ denotes nearest-integer rounding, with half-integers
rounded to the nearest even integer.  For the QPCC instances, replace
$N=4008$, $K=160$, and $\gamma_{\rm act}=0.25$ by $N=5000$, $K=800$, and
$\gamma_{\rm act}=0.20$, respectively; all other parameters and steps remain
unchanged.

\clearpage

\section{Branch QP algorithm for the ablation experiment}
\label{app:bqpa}

Algorithm~\ref{alg:bqpa} states the bQPA used in the ablation experiment.
The solver call returns either a local KKT point (IPOPT) or a globally optimal
solution (Gurobi) of the sampled fixed-branch QP.  In both variants, a finite
trial point satisfying the numerical feasibility criteria in
Section~\ref{subsubsec:metrics} is retained as an incumbent even if the outer
wall-time limit is reached during the solver call.

\begin{breakablealgorithm}
\caption{Branch quadratic programming algorithm (bQPA)}
\label{alg:bqpa}
\begin{algorithmic}[1]
\Require feasible $x^0\in\cOmega$, branch-QP solver $\mathcal Q$,
         tolerances $\varepsilon_x,\varepsilon_f>0$, patience $p$, and an
         outer wall-time limit
\Ensure a best feasible incumbent $x^{\rm best}$
\State Set $x^{\rm best}=x^0$ and $f^{\rm best}=f(x^0)$; set $c=0$
\For{$k=0,1,2,\ldots$}
    \State Choose $J^k$ conditionally uniformly from $\cP(x^k)$
    \State Call $\mathcal Q$ from $x^k$ on
           $\min\{f(x):x\in\cOmega_{J^k}\}$, obtaining a trial point
           $\widehat x^{k+1}$ and a solver status
    \If{$\widehat x^{k+1}$ is finite, satisfies the numerical feasibility
          criteria in Section~\ref{subsubsec:metrics}, and
          $f(\widehat x^{k+1})<f^{\rm best}$}
        \State Set $x^{\rm best}=\widehat x^{k+1}$ and
               $f^{\rm best}=f(\widehat x^{k+1})$
    \EndIf
    \If{the solver status is unsuccessful, or the outer time limit is reached}
        \State \Return $x^{\rm best}$
    \EndIf
    \State Set $x^{k+1}=\widehat x^{k+1}$
    \If{$\|x^{k+1}-x^k\|_\infty
          \le\varepsilon_x(1+\|x^k\|_\infty)$ and
          $|f(x^{k+1})-f(x^k)|
          \le\varepsilon_f(1+|f(x^k)|)$}
        \State Set $c=c+1$
    \Else
        \State Set $c=0$
    \EndIf
    \If{$c=p$}
        \State \Return $x^{\rm best}$
    \EndIf
\EndFor
\end{algorithmic}
\end{breakablealgorithm}
\clearpage

\section{Published Results for the Nonmatching MacMPEC Instances}
\label{app:macmpec-notes}

Table~\ref{tab:macmpec-not-bestknown} lists the ten instances for which at
least one embedded solver does not match the MacMPEC database reference value.
For comparison, Table~\ref{tab:macmpec-literature-results} summarizes results
reported for these instances in earlier studies.  These results are provided
for context only, since the formulations, sign conventions, feasibility
tolerances, and solver settings may differ across studies.

\begin{table}[htbp]
	\centering
	\caption{Results reported in earlier studies for the nonmatching instances
		in Table~\ref{tab:macmpec-not-bestknown}.}
	\label{tab:macmpec-literature-results}
	\small
	\begin{tabular}{p{0.32\textwidth}p{0.61\textwidth}}
		\hline
		Instance & Results reported in earlier studies \\
		\hline
		\texttt{b-pn2\_\_}\allowbreak
		\texttt{bem-milanc30-s}
		& $0.09$ and $1020.93$ \cite{hatz2013lifting}; iteration limit
		\cite{raghunathan2005interior}. \\

		\texttt{bilevel1}
		& $5.0$
		\cite{demiguel2005twosided,raghunathan2005interior,kanzow2013regularization}. \\

		\texttt{bilin}
		& $16.0$ \cite{kanzow2013regularization}. \\

		\texttt{ex9.2.3}
		& Approximately $5.0$
		\cite{demiguel2005twosided,raghunathan2005interior,hatz2013lifting,kanzow2013regularization}. \\

		\texttt{ex9.2.5}
		& $9.0$ \cite{hatz2013lifting,demiguel2005twosided,hall2024lcqpow}. \\

		\texttt{hs044-i}
		& Approximately $17.0901$ \cite{hatz2013lifting}. \\

		\texttt{pack-comp1\_\_}\allowbreak
		\texttt{pack-comp-32}
		& $0.653089$ \cite{raghunathan2005interior}. \\

		\texttt{tap-09\_\_tap-09}
		& $121.566$ \cite{kanzow2013regularization}; restoration failure
		\cite{raghunathan2005interior}. \\

		\texttt{tap-15\_\_tap-15}
		& Infeasible \cite{hatz2013lifting}; $187.258$ and $253.404$ in
		\cite{kanzow2013regularization} and
		\cite{raghunathan2005interior}, respectively. \\

		\texttt{water-net\_\_water-net}
		& $927.138$ with maximum constraint violation $7.93\times10^{-3}$
		\cite{kanzow2013regularization}; iteration limit
		\cite{raghunathan2005interior}. \\
		\hline
	\end{tabular}
\end{table}

\clearpage

\bibliographystyle{plainnat}
\bibliography{bib}

\end{document}